\documentclass{amsart}

\usepackage[section]{placeins}

\usepackage{amsmath}             
\usepackage{amscd}
\usepackage{amssymb}             
\usepackage{amsfonts}            
\usepackage{latexsym}            
\usepackage{amsthm}              
\usepackage{graphicx}

\usepackage{enumerate}

\usepackage{thmtools}
\usepackage{thm-restate}

\newcommand{\Z}{\mathbb{Z}}         
            
\newcommand{\Sp}{\operatorname{Sp}}
\newcommand{\SL}{\operatorname{SL}}
\newcommand{\GL}{\operatorname{GL}}
\newcommand{\SO}{\operatorname{SO}}
\newcommand{\rank}{\operatorname{rank}}
\newcommand{\Ind}{\operatorname{Ind}}
\newcommand{\Res}{\operatorname{Res}}

\makeatletter
\newcommand{\bigperp}{%
  \mathop{\mathpalette\bigp@rp\relax}%
  \displaylimits
}

\newcommand{\bigp@rp}[2]{%
  \vcenter{
    \m@th\hbox{\scalebox{\ifx#1\displaystyle2.1\else1.5\fi}{$#1\perp$}}
  }%
}
\makeatother

\newtheorem{lause}{Theorem}[section]

\newtheorem{lemma}[lause]{Lemma}
\newtheorem{seur}[lause]{Corollary}
\newtheorem{prop}[lause]{Proposition}
\newtheorem{prob}[lause]{Problem}
\newtheorem*{lause*}{Theorem}

\theoremstyle{definition}
\newtheorem{maar}[lause]{Definition}

\theoremstyle{remark}
\newtheorem{remark}[lause]{Remark}
\newtheorem{esim}[lause]{Example} 
\newtheorem*{mot*}{Motivation}
\newtheorem*{acknow*}{Acknowledgements}

\numberwithin{equation}{section}

\usepackage{hyphenat}
\allowdisplaybreaks

\begin{document}

\title[Hesselink normal forms of unipotent elements]{Hesselink normal forms of unipotent elements in some representations of classical groups in characteristic two}

\author{Mikko Korhonen}
\address{School of Mathematics, The University of Manchester, Manchester M13 9PL, United Kingdom}
\email{korhonen\_mikko@hotmail.com}
\thanks{The author was supported by a postdoctoral fellowship from the Swiss National Science Foundation (grant number P2ELP2\_181902).}

\subjclass[2010]{Primary 20G05}

\date{\today}

\begin{abstract}

Let $G$ be a simple linear algebraic group over an algebraically closed field $K$ of characteristic two. Any non-trivial self-dual irreducible $K[G]$-module $W$ admits a non-degenerate $G$-invariant alternating bilinear form, thus giving a representation $f: G \rightarrow \operatorname{Sp}(W)$. In the case where $G = \operatorname{SL}_n(K)$ and $W$ has highest weight $\varpi_1 + \varpi_{n-1}$, and in the case where $G = \operatorname{Sp}_{2n}(K)$ and $W$ has highest weight $\varpi_2$, we determine for every unipotent element $u \in G$ the conjugacy class of $f(u)$ in $\operatorname{Sp}(W)$. As a part of this result, we describe the conjugacy classes of unipotent elements of $\operatorname{Sp}(V_1) \otimes \operatorname{Sp}(V_2)$ in $\operatorname{Sp}(V_1 \otimes V_2)$. 
\end{abstract}

\maketitle


\section{Introduction}

Let $G$ be a simple algebraic group over an algebraically closed field $K$ of characteristic $p > 0$, and let $f: G \rightarrow \SL(W)$ be a non-trivial finite-dimensional rational irreducible representation. Recall that an element $u \in G$ is \emph{unipotent}, if its image under every rational representation of $G$ is a unipotent linear map. Equivalently $u$ is unipotent if it has order $p^\alpha$ for some $\alpha \geq 0$. 

In previous work \cite{KorhonenJordanGood}, some special cases of the following problem were solved.

\begin{prob}\label{prob:mainprobofoldpaper}
Let $u \in G$ be a unipotent element. What is the Jordan normal form of $f(u)$?
\end{prob}

There are relatively few cases where a complete answer to Problem \ref{prob:mainprobofoldpaper} is known. Computations done by Lawther \cite{Lawther, LawtherCorrection} give an answer in most cases where $G$ is simple of exceptional type and $W$ is either minimal-dimensional or the adjoint module. Consider the case where $G$ is a simple classical group ($\SL(V)$, $\Sp(V)$, or $\SO(V)$). For almost all irreducible representations $f$ with $\dim W \leq (\rank G)^3 / 8$ (see \cite[Theorem 5.1]{Lubeck}), the main results of \cite{KorhonenJordanGood} solve Problem \ref{prob:mainprobofoldpaper} in the case where $p$ is \emph{good for $G$}. (For a simple classical group $G$, the prime $p$ is \emph{good for $G$} if $G = \SL(V)$ or $p > 2$. Otherwise $p$ is \emph{bad for $G$}.)

In this paper we will extend the results of \cite{KorhonenJordanGood} to the case where $p$ is bad for $G$, but our main concern will be a somewhat more general problem in characteristic two. Suppose from now on that $p = 2$, and suppose that $W$ is self-dual. By Fong's lemma \cite{Fong}, there exists a non-degenerate $G$-invariant alternating bilinear form $b_0$ on $W$, which is unique up to scalar multiples. Thus we may consider $f$ as a representation $f: G \rightarrow \Sp(W,b_0)$. In the main results of this paper, we give a solution to the following problem in some special cases.

\begin{prob}\label{prob:mainprob}
Let $u \in G$ be a unipotent element. What is conjugacy class of $f(u)$ in $\Sp(W,b_0)$?
\end{prob}

\begin{remark}We note here that although in odd characteristic it is true that the Jordan normal form of $u \in \Sp(W,b_0)$ determines the conjugacy class of $u$ in $\Sp(W,b_0)$ \cite[Proposition 2 of Chapter II]{Gerstenhaber}, this no longer holds in characteristic two. Knowing the Jordan normal form of $f(u)$ is essential in the solution of Problem \ref{prob:mainprob}, but one also needs specific information about the action of $u$ on $W$ with respect to the bilinear form $b$.\end{remark}

\begin{mot*}One basic motivation for considering Problem \ref{prob:mainprobofoldpaper} and Problem \ref{prob:mainprob} is in the problem of determining the fusion of unipotent classes in maximal subgroups of simple algebraic groups. That is, for a simple algebraic group $Y$ and a maximal subgroup $X < Y$, what is the $Y$-conjugacy class of each unipotent element $u \in X$? Here solutions to Problem \ref{prob:mainprobofoldpaper} and Problem \ref{prob:mainprob} provide answers in the case where $Y$ is of classical type and $X$ is an irreducible simple subgroup.

Solutions of Problem \ref{prob:mainprobofoldpaper} in specific cases have found applications in various contexts, see for example \cite{Lawther}, \cite[Section 3]{LawtherFusion}, and \cite{TiepZalesski}. It seems that so far there are very few results on Problem \ref{prob:mainprob} in the literature, although some computations are contained in the PhD thesis of the present author. In this paper, we solve Problem \ref{prob:mainprob} in the smallest cases where the answer is not known. As an application of our results, in the final section of this paper we classify some simple subgroups of $\Sp(V,b)$ that contain \emph{distinguished} unipotent elements. (A unipotent element in a simple algebraic group is \emph{distinguished}, if its centralizer does not contain a non-trivial torus.)\end{mot*}

Let $\lambda$ be the highest weight in $W$. As the main result of this paper, we will solve Problem \ref{prob:mainprob} in the following cases:

\begin{itemize}
\item $G = \SL(V)$ and $\lambda = \varpi_1 + \varpi_{n-1}$, where $n = \dim V$ (Theorem \ref{thm:MAINTHMA}).
\item $G = \Sp(V,b)$ and $\lambda = \varpi_2$ (Theorem \ref{thm:MAINTHMC}).
\end{itemize}


In order to describe our main results in more detail, we will first have to describe how the unipotent conjugacy classes are classified in the symplectic groups.  Throughout we will describe the conjugacy class of $u$ in $\Sp(W, b)$ using the \emph{Hesselink normal form} described in \cite{Hesselink}. For the purposes of this introduction we will give a brief description, a more detailed exposition of the relevant results and concepts is given in Sections \ref{section:bilinearmodules} and \ref{section:unipclassesinfo}.

Let $g$ be a generator of a cyclic $2$-group of order $q$, and denote the group algebra of $\langle g \rangle$ by $K[g]$. Then there exist a total of $q$ indecomposable $K[g]$-modules $V_1$, $\ldots$, $V_q$ up to isomorphism, where $\dim V_i = i$ and $g$ acts on $V_i$ as a unipotent $i \times i$ Jordan block. For a $K[g]$-module $V$ we denote $V^0 = 0$ and $V^n = V \oplus \cdots \oplus V$ ($n$ copies) for all $n > 0$. Then notation $V \cong \oplus_{d \geq 1} V_d^{r_d}$ can be used to say that in the action of $g$ on $V$, a Jordan block of size $d$ occurs with multiplicity $r_d$.


For any finite-dimensional $K[g]$-module $V$ equipped with a $g$-invariant alternating bilinear form $b$ (not necessarily non-degenerate), we define a map $\varepsilon_{V,b}: \mathbb{Z}_{\geq 1} \rightarrow \{0,1\}$ by $$\varepsilon_{V, b}(d) = \begin{cases} 0, & \text{ if } b((g-1)^{d-1}v,v) = 0 \text{ for all } v \in V \text{ such that } (g-1)^d v = 0. \\ 1, & \text{ otherwise.}\end{cases}$$ We note that $\varepsilon_{V,b}(d) = 1$ is only possible for $d$ even (Lemma \ref{lemma:basicepsilonoddeven} (i)). 

Let $b$ be a non-degenerate alternating bilinear form on a finite-dimensional $K$-vector space $V$. For a unipotent element $u \in \Sp(V,b)$, as $K[u]$-modules we have $V \cong V_{d_1}^{n_1} \oplus \cdots \oplus V_{d_t}^{n_t}$, where $0 < d_1 < \cdots < d_t$ and $n_i > 0$ for all $1 \leq i \leq t$ (Jordan normal form). It turns out (Theorem \ref{thm:hesselinkepsilon}) that the $\Sp(V,b)$-conjugacy class of $u \in \Sp(V,b)$ is uniquely determined by the integers $d_i$, $n_i$, and $\varepsilon_{V,b}(d_i)$. In our main results, we will describe these integers for $f(u)$, thus describing the conjugacy class of $f(u)$ in $\Sp(W,b_0)$.

For our first main result, let $G = \SL(V)$ and $\lambda = \varpi_1 + \varpi_{n-1}$. To setup the statement, we need the following facts which are proven in Section \ref{section:altformtq}. One can show that $V \otimes V^*$ admits a non-zero alternating $G$-invariant bilinear form $b_V$ which is unique up to scalar multiples, and furthermore $b_V$ is non-degenerate if and only if $\dim V$ is even. We can identify $W = L_G(\varpi_1 + \varpi_{n-1}) = Z^\perp / Z$, where $Z$ is the unique $1$-dimensional $G$-submodule of $V \otimes V^*$. In all cases, the bilinear form $b_V$ induces a non-degenerate $G$-invariant alternating bilinear form on $Z^\perp / Z$.


Given a unipotent element $u \in G$, we can write $V \cong V_{d_1}^{n_1} \oplus \cdots \oplus V_{d_t}^{n_t}$ as $K[u]$-modules, where $0 < d_1 < \cdots < d_t$ and $n_i > 0$ for all $1 \leq i \leq t$. There exists a recursive algorithm --- involving only calculations with the integers $d_i$ and $n_i$ --- for computing the indecomposable summands of $V \otimes V^*$ and their multiplicities (Theorem \ref{thm:tensordecompchar2} and Remark \ref{remark:tensordecompz}). That is, we can assume $V \otimes V^* \cong \oplus_{d \geq 1} V_d^{\lambda(d)}$, where $\lambda(d) \geq 0$ are known integers. 

The following theorem is our first main result, which will be proven in Section \ref{section:mainA} of this paper. It describes the conjugacy class of $f(u)$ in $\Sp(W, b_V)$, and when $\dim V$ is even, the conjugacy class of the image of $u$ in $\Sp(V \otimes V^*, b_V)$. The result is given in terms of the integers $d_i$ and $\lambda(d)$. The theorem also includes the Jordan normal form of $f(u)$, which was described before in \cite[Theorem 6.1]{KorhonenJordanGood}.

\begin{restatable}[]{lausex}{mainresultA}
\label{thm:MAINTHMA}

Let $G = \SL(V)$, where $\dim V = n$ for some $n \geq 2$. Let $u \in G$ be unipotent and $V \cong V_{d_1} \oplus \cdots \oplus V_{d_t}$ as $K[u]$-modules, where $t \geq 1$ and $d_r \geq 1$ for all $1 \leq r \leq t$. Set $\alpha = \nu_2(\gcd(d_1, \ldots, d_t))$. Suppose that $V \otimes V^* \cong \oplus_{d \geq 1} V_d^{\lambda(d)}$ and $L_G(\varpi_1 + \varpi_{n-1}) \cong \oplus_{d \geq 1} V_d^{\lambda'(d)}$ as $K[u]$-modules, where $\lambda(d), \lambda'(d) \geq 0$ for all $d \geq 1$. Set $\varepsilon := \varepsilon_{V \otimes V^*, b_V}$ and $\varepsilon' := \varepsilon_{L_G(\varpi_1 + \varpi_{n-1}), b_V}$.

Then the values of $\lambda'$ are given in terms of $\lambda$ as follows:

\begin{enumerate}[\normalfont (i)]
\item If $2 \nmid n$, then $\lambda'(1) = \lambda(1) - 1$ and $\lambda'(d) = \lambda(d)$ for all $d > 1$.
\item If $2 \mid n$ and $\alpha = 0$, then $\lambda'(1) = \lambda(1) - 2$ and $\lambda'(d) = \lambda(d)$ for all $d > 1$.
\item If $2 \mid n$ and $\alpha > 0$:
	\begin{enumerate}[\normalfont (a)]
		\item If $2 \mid \frac{n}{2^\alpha}$, then $\lambda'(2^{\alpha}) = \lambda(2^{\alpha}) - 2$, $\lambda'(2^{\alpha}-1) = 2$, and $\lambda'(d) = \lambda(d)$ for all $d \neq 2^{\alpha}, 2^{\alpha}-1$.
		\item If $\alpha > 1$ and $2 \nmid \frac{n}{2^{\alpha}}$, then $\lambda'(2^{\alpha}) = \lambda(2^{\alpha}) - 1$, $\lambda'(2^{\alpha}-2) = 1$, and $\lambda'(d) = \lambda(d)$ for all $d \neq 2^{\alpha}, 2^{\alpha}-2$.
		\item If $\alpha = 1$ and $2 \nmid \frac{n}{2}$, then $\lambda'(2) = \lambda(2) - 1$ and $\lambda'(d) = \lambda(d)$ for all $d \neq 2$.
	\end{enumerate}
\end{enumerate}

Furthermore, the values of $\varepsilon$ and $\varepsilon'$ are given as follows:

\begin{enumerate}[\normalfont (i)]
\setcounter{enumi}{3}
\item $\varepsilon(d) = 1$ if and only if $d = 2^{\beta}$ for some $2^{\beta} > 1$ occurring in the consecutive-ones binary expansion (Definition \ref{def:consecutiveones}) of $d_r$ for some $1 \leq r \leq t$.
\item If (iii)(b) holds, then $\varepsilon(2^{\alpha}-2) = 0$, $\varepsilon'(2^{\alpha}-2) = 1$, and $\varepsilon'(d) = \varepsilon(d)$ for all $d \neq 2^{\alpha}-2$.
\item If (iii)(b) does not hold, then $\varepsilon'(d) = \varepsilon(d)$ for all $d \geq 1$.
\end{enumerate}

\end{restatable}

For our other main result, let $G = \Sp(V,b)$ and $\lambda = \varpi_2$. To set up the statement, we need the following facts from Section \ref{section:altformaq} --- these are very much analogous to the case with $\SL(V)$ above. One can show that $\wedge^2(V)$ admits a non-zero $G$-invariant alternating  bilinear form $a_V$ which is unique up to scalar multiples, and furthermore $a_V$ is non-degenerate if and only if $\dim V/2$ is even. We can identify $W = L_G(\varpi_2) = Q^\perp / Q$, where $Q$ is the unique $1$-dimensional $G$-submodule of $\wedge^2(V)$. In all cases, the bilinear form $a_V$ induces a non-degenerate $G$-invariant alternating bilinear form on $Q^\perp / Q$.



Let $u \in G$ be a unipotent element. Let $d_1$, $\ldots$, $d_t$ be the Jordan block sizes of $u$, and let $n_i > 0$ be the multiplicity of block size $d_i$. To compute the Jordan block sizes of $u$ on $\wedge^2(V)$, there exists a recursive algorithm which only involves computations with the integers $d_i$ and $n_i$ (Theorem \ref{thm:GowLaffey} and Remark \ref{remark:computewedgez}). Hence we can assume that $\wedge^2(V) \cong \oplus_{d \geq 1} V_d^{\lambda(d)}$, where the integers $\lambda(d) \geq 0$ are known.


The following theorem is our second main result, which will be proven in Section \ref{section:mainB} of this paper. It describes the conjugacy class of $f(u)$ in $\Sp(W,a_V)$, and when $\dim V/2$ is even, the conjugacy class of the image of $u$ in $\Sp(\wedge^2(V), a_V)$. The result is given in terms of the integers $d_i$, $n_i$, and $\lambda(d)$.


\begin{restatable}[]{lausex}{mainresultC}
\label{thm:MAINTHMC}


%
	
Let $G = \Sp(V, b)$, where $\dim V = 2n$ for some $n \geq 2$. Let $u \in G$ be unipotent. For $t \geq 0$, let $d_1$, $\ldots$, $d_t$ be the Jordan block sizes $d$ of $u$ such that $\varepsilon_{V,b}(d) = 0$, and for $s \geq 0$ let $2d_{t+1}$, $\ldots$, $2d_{t+s}$ be the Jordan block sizes $d$ of $u$ such that $\varepsilon_{V,b}(d) = 1$. Write $V \cong V_{d_1}^{n_1} \oplus \cdots \oplus V_{d_t}^{n_t} \oplus V_{2d_{t+1}}^{n_{t+1}} \oplus \cdots \oplus V_{2d_{t+s}}^{n_{t+s}}$ as $K[u]$-modules, where $n_r > 0$ for all $1 \leq r \leq t+s$.

Set $\alpha = \nu_2(\gcd(d_1, \ldots, d_{t+s}))$. Suppose that $\wedge^2(V) \cong \oplus_{d \geq 1} V_d^{\lambda(d)}$ and $L_G(\varpi_2) \cong \oplus_{d \geq 1} V_d^{\lambda'(d)}$ as $K[u]$-modules, where $\lambda(d), \lambda'(d) \geq 0$ for all $d \geq 1$. Set $\varepsilon := \varepsilon_{\wedge^2(V), a_V}$ and $\varepsilon' := \varepsilon_{L_G(\varpi_2), a_V}$.

Then the values of $\lambda'$ are given in terms of $\lambda$ by the rules (i) -- (iii) of Theorem \ref{thm:MAINTHMA}. Furthermore, the values of $\varepsilon$ and $\varepsilon'$ are given as follows:

\begin{enumerate}[\normalfont (i)]
\setcounter{enumi}{3}
\item We have $\varepsilon(d) = 1$ if and only if one of the following conditions holds:
	\begin{enumerate}[\normalfont (a)]
		\item $d = 2^{\beta}$ for some $2^{\beta} > 1$ occurring in the consecutive-ones binary expansion (Definition \ref{def:consecutiveones}) of $d_r$ for some $1 \leq r \leq t$.
		\item $d > 1$ occurs as a Jordan block size of $\wedge^2(V_{2d_{r}})$ for some $t+1 \leq r \leq t+s$.
		\item $d = d'2^{\beta+1}$, where: 
			\begin{itemize} 
				\item $\beta = \nu_2(d_{r}) = \nu_2(d_{r'})$ for some $t+1 \leq r \leq r' \leq t+s$; 
				\item $n_r > 1$ if $r = r'$;
				\item $d'$ is the unique odd Jordan block size in $V_{d_{r}/2^{\beta}} \otimes V_{d_{r'} / 2^{\beta}}$ (Lemma \ref{lemma:uniqueoddchar2}).	
			\end{itemize}	
	\end{enumerate}
\item If $2 \nmid n$ or $\alpha = 0$, then $\varepsilon(d) = \varepsilon'(d)$ for all $d \geq 1$.
\item If $2 \mid n$ and $\alpha > 0$, then $\varepsilon(d) = \varepsilon'(d)$ for all $d \neq 2^{\alpha}, 2^{\alpha}-2$, and:
	\begin{enumerate}[\normalfont (a)]
		\item $\varepsilon(2^{\alpha}) = 1$.
		\item $\varepsilon'(2^{\alpha}) = 1$ if and only if $\nu_2(d_r) = \alpha$ for some $1 \leq r \leq t$.
		\item If $\alpha > 1$, then $\varepsilon(2^{\alpha}-2) = 0$.
		\item If $\alpha > 1$, then $\varepsilon'(2^{\alpha}-2) = 1$ if and only if $2 \nmid \frac{n}{2^{\alpha}}$.
	\end{enumerate}
\end{enumerate}



\end{restatable}

As part of our main result for $G = \Sp(V,b)$, we will also have to resolve the following problem.

\begin{prob}\label{prob:mainprobtensor}
Let $u_1 \in \Sp(V_1, b_1)$ and $u_2 \in \Sp(V_2, b_2)$ be unipotent. What is the conjugacy class of $u_1 \otimes u_2$ in $\Sp(V_1 \otimes V_2, b_1 \otimes b_2)$?
\end{prob}

Here $b_1 \otimes b_2$ is the usual product form on $V_1 \otimes V_2$ given by $b_1$ and $b_2$, see Definition \ref{def:tensorprod_bilinear}. We will give a complete solution to Problem \ref{prob:mainprobtensor} in Section \ref{section:tensorproductbilinear}.

\begin{remark}
Let $G = \Sp(V,b)$, where $\dim V > 4$. In characteristic two, we always have $\SO(V,q) < G$, where $q$ is a quadratic form on $V$ such that $q(v+w) + q(v) + q(w) = b(v,w)$ for all $v,w \in V$. It follows for example from \cite[Theorem 4.1]{SeitzClassical} that the restriction of the irreducible $K[G]$-module with highest weight $\varpi_2$ to $\SO(V,q)$ remains irreducible. Furthermore, the unipotent conjugacy classes of $\SO(V,q)$ can be described in terms of unipotent conjugacy classes of $G$ \cite[Proposition 6.22]{LiebeckSeitzClass}. Thus from our main result for $G = \Sp(V,b)$ (Theorem \ref{thm:MAINTHMC}), it is straightforward to deduce the corresponding result for $\SO(V,q)$.\end{remark}

\begin{acknow*}
The author would like to acknowledge the anonymous referee for their useful comments and suggestions.
\end{acknow*}

\section{Notation}\label{section:notation}

We fix the following notation and terminology, some of which was already mentioned in the introduction. Throughout the text, let $K$ be an algebraically closed field. \emph{We will always assume that $K$ has characteristic two.} For an integer $n \in \Z$, we will denote the element $n \cdot 1_K$ of $K$ by $n$, and it will be clear from the context when $n$ is considered as an element of $K$.


For a $K$-vector space $V$ and non-negative integer $n$, we use the notation $V^n$ for the direct sum $V \oplus \cdots \oplus V$, where $V$ occurs $n$ times. Note that $V^0 = 0$.

Let $u$ be a generator of a cyclic $2$-group of order $q$. We will denote the group algebra of $\langle u \rangle$ over $K$ by $K[u]$. Recall that $K[u]$ has exactly $q$ indecomposable modules $V_1$, $\ldots$, $V_q$ up to isomorphism, where $\dim V_i = i$ and $u$ acts on $V_i$ as a full $i \times i$ Jordan block. For convenience of notation, we denote $V_0 = 0$. Any non-zero $K[u]$-module $V$ has a decomposition $V \cong V_{d_1}^{n_1} \oplus \cdots \oplus V_{d_t}^{n_t}$, where $t \geq 1$, $0 < d_1 < \cdots < d_t$, and $n_i > 0$ for all $i$ (Jordan normal form). We call the $d_i$ the \emph{Jordan block sizes of $u$ on $V$}, and $n_i$ is the \emph{multiplicity of $d_i$ in $V$}.

When considering $K[u]$-modules, we will denote by $X$ the element $u-1$ of $K[u]$. Let $Y \in K[u]$. If a $K[u]$-module $V$ has a $K[u]$-submodule $W$, we will usually use the notation $Y_W$ for the linear map $Y_W: W \rightarrow W$ induced by the action of $Y$ on $W$, and similarly $Y_{V/W}$ for the linear map $Y_{V/W} : V/W \rightarrow V/W$ induced by the action of $Y$ on $V/W$.

Throughout the text $G$ will always denote a group. Any $K[G]$-module that we consider will be finite-dimensional. If a $K[G]$-module $V$ has a filtration $V = W_1 \supset W_2 \supset \cdots \supset W_{t} \supset W_{t+1} = 0$ with $\operatorname{soc}(V / W_{i+1}) = W_i / W_{i+1} \cong Z_i$ for all $1 \leq i \leq t$, we will denote this by $V = Z_1 | Z_2 | \cdots | Z_t$. Let $G$ be a group and $H < G$ a subgroup. We denote the restriction of a $K[G]$-module $V$ to $H$ by $\Res_H^G(V)$. For a $K[H]$-module $W$, the induced module of $W$ from $H$ to $G$ is $\Ind_H^G(W) := K[G] \otimes_{K[H]} W$. 

A bilinear form $b$ on a vector space $V$ is \emph{non-degenerate}, if its \emph{radical} $\operatorname{rad} b = \{v \in V: b(v,w) = 0 \text{ for all } w \in V \}$ is zero. For a subspace $W$ of $V$, we call $W$ \emph{totally singular} with respect to $b$ if $b(w,w') = 0$ for all $w,w' \in W$. We say that $b$ is \emph{alternating}, if $b(v,v) = 0$ for all $v \in V$, and \emph{symmetric} if $b(v,w) = b(w,v)$ for all $v, w \in V$. Note that since we are working over a field of characteristic two, any alternating bilinear form is also symmetric. If $V$ is a $K[G]$-module, then $b$ is \emph{$G$-invariant} if $b(gv,gw) = b(v,w)$ for all $g \in G$ and $v, w \in V$. For a non-degenerate alternating bilinear form $b$ on $V$, we denote $\Sp(V, b) = \{ g \in \GL(V) : b(gv,gw) = b(v,w) \text{ for all } v,w \in V\}$.

Suppose that $G$ is a simple linear algebraic group over $K$. In the context of algebraic groups, the notation that we use will be as in \cite{JantzenBook}. For basic terminology and results on algebraic groups, see \cite{HumphreysGroupBook}. We note however that not much will be needed from the theory of algebraic groups. For the most part, the only algebraic groups that appear in this paper are $G = \SL(V)$ or $G = \Sp(V, b)$. 

When $G$ is an algebraic group, by a $K[G]$-module we will always mean a finite-dimensional rational $K[G]$-module. We fix a maximal torus $T$ of $G$ with character group $X(T)$, and a base $\Delta = \{ \alpha_1, \ldots, \alpha_{\ell} \}$ for the root system of $G$, where $\ell = \operatorname{rank} G$. Here we use the standard Bourbaki labeling of the simple roots $\alpha_i$, as given in \cite[11.4, p. 58]{Humphreys}. We denote the dominant weights with respect to $\Delta$ by $X(T)^+$, and the fundamental dominant weight corresponding to $\alpha_i$ is denoted by $\varpi_i$. For a dominant weight $\lambda \in X(T)^+$, we denote the rational irreducible $K[G]$-module with highest weight $\lambda$ by $L_G(\lambda)$.

For a simple linear algebraic group $G \leq \GL(V)$, an element $u \in G$ is \emph{unipotent}, if it is unipotent as a linear transformation on $V$. That is, if $(u-1_V)^n = 0$ for some $n > 0$. Since $\operatorname{char} K = 2$, an equivalent definition is that $u \in G$ is unipotent if and only if it has order $2^k$ for some $k \geq 0$.

For non-negative integers $a$ and $b$ we denote by $\binom{a}{b}$ the usual binomial coefficient, using the convention that $\binom{a}{b} = 0$ if $a < b$. We denote by $\nu_2$ the $2$-adic valuation on the integers, so $\nu_2(a)$ is the largest integer $k \geq 0$ such that $2^k$ divides $a$.

\section{Preliminaries}

In this section, we list some preliminary results needed in the paper. All of the results in this section are well known, and furthermore the results and their proofs generalize to arbitrary characteristic $p > 0$. We begin with some basic results about unipotent linear maps.

\begin{lemma}\label{lemma:restrictionpowerp}
Let $u$ be a generator of a cyclic $2$-group of order $q$, and suppose that $2^{\alpha} \leq q$. For an integer $0 < n \leq q$, write $n = a2^{\alpha} + r$ for $0 \leq r < 2^{\alpha}$ and $a \in \mathbb{Z}$. Then $$\Res_{\langle u^{2^{\alpha}} \rangle}^{\langle u \rangle}(V_n) \cong V_{a+1}^r \oplus V_a^{2^{\alpha}-r}.$$

\end{lemma}

\begin{proof}
Let $e_1$, $\ldots$, $e_n$ be a basis of $V_n$ such that $ue_1 = e_1$ and $ue_i = e_i + e_{i-1}$ for all $1 < i \leq n$. Set $e_j = 0$ for $j \leq 0$ and $j > n$. Now $(u-1)^ke_i = e_{i-k}$ for all $k \geq 1$ and $i > 0$. Since $(u-1)^{2^{\alpha}} = u^{2^{\alpha}} - 1$, it follows that \begin{equation}\label{eq:palphaeasy}u^{2^\alpha}e_i = e_i + e_{i-2^{\alpha}}\end{equation} for all $1 \leq i \leq n$. For all $1 \leq i \leq 2^{\alpha}$, define $W_i$ to be the subspace spanned by $\{ e_{i+j2^{\alpha}} \}_{j \geq 0}$. Then $V = W_1 \oplus \cdots \oplus W_{2^{\alpha}}$. Furthermore, from~\eqref{eq:palphaeasy} we find that each $W_i$ is $u^{2^{\alpha}}$-invariant and as $K[u^{2^{\alpha}}]$-modules $W_i \cong V_{a+1}$ for $1 \leq i \leq r$ and $W_i \cong V_a$ for $r < i \leq 2^{\alpha}$. From this the lemma follows.\end{proof}

\begin{lemma}\label{lemma:inductionpowerp}
Let $u$ be a generator of a cyclic $2$-group of order $q$, and suppose that $2^{\alpha} \leq q$. Then $$\Ind_{\langle u^{2^{\alpha}} \rangle}^{\langle u \rangle}(V_n) \cong V_{2^{\alpha}n}$$ for all $0 < n \leq q/2^{\alpha}$.
\end{lemma}

\begin{proof}
The lemma is an immediate consequence of Green's indecomposability theorem \cite[Theorem 8]{GreenIndecomposables}. For an elementary proof, let $0 < n \leq q/2^{\alpha}$ and set $V = \Ind_{\langle u^{2^{\alpha}} \rangle}^{\langle u \rangle}(V_n)$. To prove the lemma, it will suffice to show that the $u$-fixed point space of $V$ is one-dimensional. Let $W = 1 \otimes V_n$, so $$V = \bigoplus_{0 \leq i \leq 2^{\alpha}-1} u^iW,$$ where as $K[u^{2^{\alpha}}]$-modules $u^iW \cong V_n$ for all $0 \leq i \leq 2^{\alpha}-1$. The $u^{2^{\alpha}}$-fixed point space of $W$ is one-dimensional, spanned by some $w \in W$. Then for all $0 \leq i \leq 2^{\alpha}-1$, the $u^{2^{\alpha}}$-fixed point space of $u^iW$ is spanned by $u^iw$. From this it easily follows that the $u$-fixed point space of $V$ is spanned by $\sum_{0 \leq i \leq 2^{\alpha}-1} u^iw$.\end{proof}

\begin{lemma}[{\cite[Lemma 3.3]{KorhonenJordanGood}}]\label{jordanrestriction}
Let $u \in \GL(V)$ be unipotent and denote $X = u - 1$. Suppose that $W \subseteq V$ is a subspace invariant under $u$ such that $\dim V/W = 1$. Write $V \cong \oplus_{d \geq 1} V_d^{\lambda(d)}$ and $W \cong \oplus_{d \geq 1} V_d^{\lambda'(d)}$ as $K[u]$-modules, where $\lambda(d), \lambda'(d) \geq 0$ for all $d \geq 1$.

Let $m \geq 1$ be such that $\operatorname{Ker} X^{m-1} \subseteq W$ and $\operatorname{Ker} X^{m} \not\subseteq W$. Then:

\begin{enumerate}[\normalfont (i)]
\item if $m = 1$, we have $\lambda'(1) = \lambda(1) - 1$ and $\lambda'(d) = \lambda(d)$ for all $d > 1$.
\item if $m > 1$, we have $\lambda'(m) = \lambda(m) - 1$, $\lambda'(m-1) = \lambda(m-1) + 1$, and $\lambda'(d) = \lambda(d)$ for all $d \neq m,m-1$.
\end{enumerate}
\end{lemma}

\begin{lemma}\label{jordanquotient}
Let $u \in \GL(V)$ be unipotent and denote $X = u - 1$. Suppose that $W \subseteq V$ is a subspace invariant under $u$ such that $\dim W = 1$. Write $V \cong \oplus_{d \geq 1} V_d^{\lambda(d)}$ and $V/W \cong \oplus_{d \geq 1} V_d^{\lambda'(d)}$ as $K[u]$-modules, where $\lambda(d), \lambda'(d) \geq 0$ for all $d \geq 1$.

Let $m \geq 1$ be such that $\operatorname{Im} X^{m-1} \supseteq W$ and $\operatorname{Im} X^{m} \not\supseteq W$. Then:

\begin{enumerate}[\normalfont (i)]
\item if $m = 1$, we have $\lambda'(1) = \lambda(1) - 1$ and $\lambda'(d) = \lambda(d)$ for all $d > 1$.
\item if $m > 1$, we have $\lambda'(m) = \lambda(m) - 1$, $\lambda'(m-1) = \lambda(m-1) + 1$, and $\lambda'(d) = \lambda(d)$ for all $d \neq m,m-1$.
\end{enumerate}
\end{lemma}

\begin{proof}It is clear that $\operatorname{Im} X_V^i \subseteq W$ for all $0 \leq i \leq m-1$ and $\operatorname{Im} X_V^i \cap W = 0$ for all $i \geq m$. For all $i \geq 0$, we have $\operatorname{Im} X_{V/W}^i \cong \operatorname{Im} X^i / \operatorname{Im} X^i \cap W$ as vector spaces, so we conclude that $\rank X_{V/W}^i = \rank X_V^i - 1$ for all $0 \leq i \leq m-1$ and $\rank X_{V/W}^i = \rank X_V^i$ for all $i \geq m$. Now the claim follows from \cite[Lemma 3.2]{KorhonenJordanGood}.\end{proof}



The following results are used to construct the irreducible representations that we consider in our main results.

\begin{lemma}[{\cite[Proposition 4.6.10]{McNinch}}]\label{lemma:typeA_VxV}
Let $G = \SL(V)$, where $\dim V = n$ for some $n \geq 2$. Then as $K[G]$-modules, we have $$V \otimes V^* \cong \begin{cases} L_G(\varpi_1 + \varpi_{n-1}) \oplus L_G(0), &\mbox{if } 2 \nmid n, \\
L_G(0) | L_G(\varpi_1 + \varpi_{n-1}) | L_G(0), \text{ } & \mbox{if } 2 \mid n. \end{cases}$$
\end{lemma}

\begin{lemma}[{\cite[1.14, 8.1 (c)]{SeitzClassical}, \cite[Lemma 4.8.2]{McNinch}}]\label{lemma:typeComega}
Let $G = \Sp(V,b)$, where $\dim V = 2n$ for some $n \geq 2$. Then as $K[G]$-modules, we have $$\wedge^2(V) \cong \begin{cases} L_G(\varpi_2) \oplus L_G(0), &\mbox{if } 2 \nmid n, \\
L_G(0) | L_G(\varpi_2) | L_G(0), \text{ } & \mbox{if } 2 \mid n. \end{cases}$$
\end{lemma}

\section{Decomposition of tensor products and exterior squares}\label{section:tensorproducts}

In this section, we give results on the decomposition of tensor products and exterior squares of unipotent linear maps. Throughout, we let $u$ be a generator of a cyclic $2$-group of order $q > 1$, and denote the indecomposable $K[u]$-modules by $V_1$, $\ldots$, $V_q$ as defined in Section \ref{section:notation}. A recursive algorithm for calculating the decomposition of $V_m \otimes V_n$ into indecomposable summands is given by the following theorem, see for example \cite[(2.5a)]{GreenModular} and \cite[Lemma 1]{GowLaffey} for a proof.

\begin{lause}\label{thm:tensordecompchar2}
Let $0 < m \leq n \leq q$ and $2^{\alpha} \leq n < 2^{\alpha+1}$. Then the following statements hold:
\begin{enumerate}[\normalfont (i)]
\item If $m+n > 2^{\alpha+1}$, then $V_m \otimes V_n \cong V_{2^{\alpha+1}}^{m+n-2^{\alpha+1}} \oplus (V_{2^{\alpha+1}-n} \otimes V_{2^{\alpha+1}-m})$.
\item If $n = 2^{\alpha}$, then $V_m \otimes V_n = V_{2^{\alpha}}^m$.
\item If $n > 2^{\alpha}$ and $m+n \leq 2^{\alpha+1}$, then $V_m \otimes V_n \cong V_{2^{\alpha+1}-d_m} \oplus \cdots \oplus V_{2^{\alpha+1}-d_1}$, where $V_m \otimes V_{2^{\alpha+1}-n} \cong V_{d_1} \oplus \cdots \oplus V_{d_m}$.
\end{enumerate}
\end{lause}

\begin{remark}\label{remark:tensordecompz}Taking tensor products of $K[u]$-modules is an additive functor, so using Theorem \ref{thm:tensordecompchar2} one can decompose any tensor product of two $K[u]$-modules into indecomposable summands.\end{remark}

\begin{lemma}\label{lemma:uniqueoddchar2}
Let $m$ and $n$ be odd integers such that $0 < m \leq n \leq q$. Suppose that $V_m \otimes V_n \cong V_{d_1} \oplus \cdots \oplus V_{d_t}$, where $d_i > 0$ for all $i$. There exists a unique $i$ such that $d_i$ is odd.
\end{lemma}

\begin{proof}We prove the lemma by induction on $n$. The case $n = 1$ is obvious. Suppose then that $0 < m \leq n$ are odd integers and $n > 1$. Let $\alpha > 0$ be such that $2^{\alpha} < n < 2^{\alpha+1}$. If $m+n > 2^{\alpha+1}$, by Theorem \ref{thm:tensordecompchar2} (i) $$V_m \otimes V_n \cong V_{2^{\alpha+1}}^{m+n-2^{\alpha+1}} \oplus (V_{2^{\alpha+1}-n} \otimes V_{2^{\alpha+1}-m})$$ so the claim follows by applying induction on the tensor product $V_{2^{\alpha+1}-n} \otimes V_{2^{\alpha+1}-m}$. The other possibility is that $m+n \leq 2^{\alpha+1}$, in which case by Theorem \ref{thm:tensordecompchar2} (iii) we have $V_m \otimes V_n \cong V_{2^{\alpha+1}-d_m} \oplus \cdots \oplus V_{2^{\alpha+1}-d_1}$, where $V_m \otimes V_{2^{\alpha+1}-n} \cong V_{d_1} \oplus \cdots \oplus V_{d_m}$. Thus the claim follows by applying induction on $V_m \otimes V_{2^{\alpha+1}-n}$.\end{proof}

\begin{remark}One can also describe the unique odd Jordan block size of Lemma \ref{lemma:uniqueoddchar2} explicitly. Let $0 < m \leq n \leq q$ be odd integers. Write $m = \sum_{i = 0}^t a_i 2^i$ and $n = \sum_{i = 0}^t b_i 2^i$, where $a_i, b_i \in \{0,1\}$ for all $0 \leq i \leq t$. We shall omit the proof from this paper, but one can show that the unique odd Jordan block size in $V_m \otimes V_n$ is equal to $n + \sum_{i = 1}^t a_i(-1)^{b_i} 2^i.$\end{remark}

\begin{esim}\label{example:tensor3n}
To give an example of Theorem \ref{thm:tensordecompchar2} in a small case, consider $V_m \otimes V_n$ for $m = 3$. This particular example will also be useful later (Example \ref{example:v6tensor}). In any case, with Theorem \ref{thm:tensordecompchar2} it is easy to show that for all $n \geq 3$, $$V_3 \otimes V_n \cong \begin{cases} V_n^3, & \mbox{if } n \equiv 0 \mod{4}. \\ 
V_{n-1}^2 \oplus V_{n+2}, & \mbox{if } n \equiv 1 \mod{4}. \\ 
V_{n-2} \oplus V_{n} \oplus V_{n+2}, & \mbox{if } n \equiv 2 \mod{4}. \\
V_{n-2} \oplus V_{n+1}^2,  & \mbox{if } n \equiv 3 \mod{4}. \end{cases}$$ It also clear from this decomposition that the conclusion of Lemma \ref{lemma:uniqueoddchar2} holds in this case. 
\end{esim}

Following \cite[p. 231]{GlasbyPraegerXiapart}, we make the following definition. 

\begin{maar}\label{def:consecutiveones} The \emph{consecutive-ones binary expansion} of an integer $n > 0$ is the alternating sum $n = \sum_{i = 1}^k (-1)^{i+1} 2^{e_i}$ such that $e_1 > \cdots > e_k \geq 0$ and $k$ is minimal.\end{maar}

The consecutive-ones binary expansion can be calculated as follows. Grouping together the blocks of consecutive ones in the binary expansion of $n$, we write $n = \sum_{i = 1}^{\ell} \sum_{j = b_i}^{a_i - 1} 2^j$, where $\ell \geq 1$ and $a_1 > b_1 > \cdots > a_{\ell} > b_{\ell} \geq 0$. Now $\sum_{j = b_i}^{a_i - 1} 2^j = 2^{a_i} - 2^{b_i}$, and the consecutive-ones binary expansion of $n$ is given by $n = \sum_{i = 1}^{\ell} (2^{a_i} - 2^{b_i})$ if $a_{\ell} > b_{\ell} + 1$ and $n = \sum_{i = 1}^{\ell-1} (2^{a_i} - 2^{b_i}) + 2^{b_{\ell}}$ if $a_{\ell} = b_{\ell}+1$. For example, we have consecutive-ones binary expansions $3 = 2^2 - 2^0$, $4 = 2^2$, $5 = 2^3 - 2^2 + 2^0$, and $6 = 2^3 - 2^1$. Note that $e_{k-1} > e_k + 1$ for any consecutive-ones binary expansion with $k > 1$.

We shall need the following result from \cite{GlasbyPraegerXiapart}, where for $0 < n \leq q$ the decomposition of $V_n \otimes V_n$ was described explicitly in terms of the consecutive-ones binary expansion of $n$. 


\begin{lause}[{\cite[Theorem 15]{GlasbyPraegerXiapart}}]\label{thm:GPXtensorsquare}
Suppose that $0 < n \leq q$, and let $n = \sum_{i = 1}^k (-1)^{i+1} 2^{e_i}$ be the consecutive-ones binary expansion of $n$, where $e_1 > \cdots > e_k \geq 0$. Then $$V_n \otimes V_n \cong \bigoplus_{1 \leq i \leq k} V_{2^{e_i}}^{d_i},$$ where $d_k = 2^{e_k}$, and $d_i = 2^{e_i} - \sum_{j = i+1}^k (-1)^{i+j+1}2^{e_j+1}$ for all $1 \leq i < k$.
\end{lause}

We finish this section by discussing some results on the decomposition of $\wedge^2(V_n)$. The following recursive description of $\wedge^2(V_n)$ is due to Gow and Laffey \cite{GowLaffey}.

\begin{lause}[{\cite[Theorem 2]{GowLaffey}}]\label{thm:GowLaffey}
Suppose that $q/2 < n \leq q$. Then $$\wedge^2(V_n) \cong \wedge^2(V_{q-n}) \oplus V_q^{n-q/2-1} \oplus V_{3q/2 - n}.$$
\end{lause}

\begin{esim}\label{esim:wedge2power2}
Applying Theorem \ref{thm:GowLaffey} with $n = q = 2^\alpha$, it is immediate that $\wedge^2(V_{2^\alpha}) \cong V_{2^{\alpha-1}} \oplus V_{2^{\alpha}}^{2^{\alpha-1}-1}$ for all $\alpha > 0$.
\end{esim}

Note that $$\wedge^2(V \oplus W) = \wedge^2(V) \oplus \wedge^2(W) \oplus \left( V \wedge W \right)$$ for all $K[u]$-modules $V$ and $W$. Since $V \wedge W \cong V \otimes W$ as $K[u]$-modules, this decomposition gives the following result.

\begin{lemma}\label{lemma:wedge2ofkumod}
Let $V$ be a $K[u]$-module such that $V \cong V_{d_1} \oplus \cdots \oplus V_{d_t}$, where $d_i > 0$ are integers. Then $$\wedge^2(V) \cong \bigoplus_{1 \leq i \leq t} \wedge^2(V_{d_i}) \oplus \bigoplus_{1 \leq i < j \leq t} V_{d_i} \otimes V_{d_j}.$$
\end{lemma}

\begin{remark}\label{remark:computewedgez}With Lemma \ref{lemma:wedge2ofkumod}, Theorem \ref{thm:GowLaffey}, and Theorem \ref{thm:tensordecompchar2}, we can compute the decomposition of $\wedge^2(V)$ for any $K[u]$-module $V$ efficiently.\end{remark}

Next we consider some results on the multiplicities of the Jordan block sizes in $\wedge^2(V_{2n})$.


\begin{lemma}\label{lemma:minblockinwedge2}
Let $0 < n \leq q/2$ and set $\alpha = \nu_2(n)$. Then the smallest Jordan block size in $\wedge^2(V_{2n})$ is $2^{\alpha}$, occurring with multiplicity one.
\end{lemma}

\begin{proof}
By induction on $n$. In the case $n = 1$, the claim holds since $\wedge^2(V_{2n}) = \wedge^2(V_2) = V_1$. Suppose then $n > 1$ and that the claim holds for all $0 < n' < n$. Without loss of generality, we can assume that $q/2 < 2n \leq q$. If $2n = q$, then the claim follows from Example \ref{esim:wedge2power2}. Suppose that $q/2 < 2n < q$. Then \begin{equation}\label{eq:wedgeof2n}\wedge^2(V_{2n}) \cong \wedge^2(V_{q-2n}) \oplus V_q^{2n-q/2-1} \oplus V_{3q/2 - 2n}\end{equation} by Theorem \ref{thm:GowLaffey}. Now $\nu_2((q-2n)/2) = \nu_2(n) = \alpha$ since $q > 2^{\alpha+1}$, so by induction the smallest Jordan block size in $\wedge^2(V_{q-2n})$ is $2^{\alpha}$, occurring with multiplicity one. Furthermore, we have $q > 3q/2 - 2n > q/2 \geq 2^{\alpha}$, so the result follows from~\eqref{eq:wedgeof2n}.
\end{proof}

%

\begin{lemma}\label{lemma:blockmultiplicitieswedge}
Let $0 < n \leq q/2$. Then every Jordan block size in $\wedge^2(V_{2n})$ has odd multiplicity.
\end{lemma}

\begin{proof}By induction on $n$. The steps of the proof are essentially the same as in the proof of Lemma \ref{lemma:minblockinwedge2}, so we omit the details.\end{proof}


\begin{lemma}\label{lemma:multiplicityatmosttwowedge}
Let $n > 0$ and suppose that all Jordan block sizes in $\wedge^2(V_{2n})$ have multiplicity at most $2$. Then $n \in \{1,2,3,5\}$.
\end{lemma}

\begin{proof}
For $n = 4$, an easy calculation with Theorem \ref{thm:GowLaffey} shows that $\wedge^2(V_4) \cong V_4 \oplus V_8^3$. Thus we may assume $n > 4$ for what follows. Let $q$ be a power of $2$ such that $q/2 < 2n \leq q$. Suppose that all Jordan block sizes in $\wedge^2(V_{2n})$ have multiplicity at most $2$. Then by Lemma \ref{lemma:blockmultiplicitieswedge} each Jordan block size in $\wedge^2(V_{2n})$ has multiplicity one. By Theorem \ref{thm:GowLaffey}, we have \begin{equation}\label{eq:anothergwform}\wedge^2(V_{2n}) \cong \wedge^2(V_{q-2n}) \oplus V_q^{2n-q/2-1} \oplus V_{3q/2 - 2n}\end{equation} so $2n-q/2-1 \leq 1$, which forces $2n = q/2+2$. Then~\eqref{eq:anothergwform} becomes $\wedge^2(V_{2n}) \cong \wedge^2(V_{q/2-2}) \oplus V_q \oplus V_{q-2}$. Now $q/4 < q/2-2 < q/2$, so applying Theorem \ref{thm:GowLaffey} we get $$\wedge^2(V_{q/2-2}) \cong V_1 \oplus V_{q/4+2} \oplus V_{q/2}^{q/4-3},$$ and therefore $q/4-3 \leq 1$, giving $q \leq 16$. Since $n > 4$ and $q/2 < 2n \leq q$, it follows that $q = 16$. In this case $2n = q/2 + 2 = 10$, so $n = 5$.\end{proof}

\section{Modules equipped with a bilinear form}\label{section:bilinearmodules}
Let $G$ be a group. It is an elementary fact in representation theory that the $\GL(V)$-conjugacy classes of homomorphisms $G \rightarrow \GL(V)$ are in bijection with the isomorphism classes of $K[G]$-module structures on $V$. Similarly, it is convenient to study the conjugacy classes of subgroups of $\Sp(V,b)$ in terms of modules equipped with a non-degenerate alternating bilinear form. In later sections of this paper, this will be useful for us when describing the conjugacy class of a unipotent element $u \in \Sp(V,b)$.

For some generalities on modules equipped with a bilinear form, see for example \cite{Willems}, \cite{Quebbemann}, and \cite{MurraySymmetricVertices}. We give the basic definitions and results needed in this paper in what follows.

\begin{maar}
A \emph{bilinear $K[G]$-module} $(V, b)$ is a $K[G]$-module $V$ with a $G$-invariant bilinear form $b: V \times V \rightarrow K$. A bilinear $K[G]$-module $(V, b)$ is said to be \emph{non-degenerate} if $b$ is non-degenerate, \emph{symmetric} if $b$ is symmetric, and \emph{alternating} if $b$ is alternating. 
\end{maar}

\begin{maar}An \emph{isomorphism} of bilinear $K[G]$-modules $(V, b)$ and $(V', b')$ is an isomorphism $\varphi: V \rightarrow V'$ of $K[G]$-modules such that $b'(\varphi(v), \varphi(w)) = b(v,w)$ for all $v, w \in V$.\end{maar}

\begin{maar}
Let $(V, b)$ be a bilinear $K[G]$-module and $W$ a $K[G]$-submodule of $V$. We denote $(W, b) := (W, b|_{W \times W})$. Furthermore, if $W$ is totally singular with respect to $b$, then we set $(W^\perp / W, b) := (W^\perp / W, b')$, where $b'(v_1+W,v_2+W) = b(v_1,v_2)$ for all $v_1, v_2 \in W^\perp$.
\end{maar}

\begin{maar}The \emph{orthogonal direct sum} of two bilinear $K[G]$-modules $(V, b)$ and $(V', b')$ is the bilinear $K[G]$-module $(V \oplus V', b \perp b')$, where $$(b \perp b')(v_1+v_1',v_2+v_2') = b(v_1, v_2) + b'(v_1', v_2')$$ for all $v_1, v_2 \in V$ and $v_1', v_2' \in V'$. We denote $(V \oplus V', b \perp b') := (V, b) \perp (V', b')$.\end{maar}

In the context of bilinear $K[G]$-modules, for $n \geq 0$ we will use $(V,b)^n$ to denote $(V,b) \perp \cdots \perp (V,b)$, where $(V,b)$ occurs $n$ times in the orthogonal direct sum. Note that $(V,b)^0 = 0$.

\begin{maar}We call a bilinear $K[G]$-module $(V, b)$ \emph{orthogonally indecomposable}, if $V \neq 0$ and whenever $V = V_1 \perp V_2$ for two $K[G]$-submodules $V_1$ and $V_2$, we have $V_1 = 0$ or $V_2 = 0$.\end{maar}


\begin{remark}It is clear that any bilinear $K[G]$-module decomposes into an orthogonal direct sum of orthogonally indecomposable bilinear $K[G]$-modules. However, there is no analogue of the Krull-Schmidt theorem in this setting, as noted in \cite[3.13]{WillemsThesis}. In fact, even the number of orthogonally indecomposable summands is not unique, see for example Lemma \ref{lemma:threeorthogonals} below.\end{remark}

\begin{maar}\label{def:tensorprod_bilinear}
The \emph{tensor product} of bilinear $K[G]$-modules $(V, b)$ and $(V', b')$ is the bilinear $K[G]$-module $(V \otimes V', b \otimes b')$, where $b \otimes b'$ is defined by $$(b \otimes b')(v_1 \otimes v_1', v_2 \otimes v_2') = b(v_1, v_2) b'(v_1', v_2')$$ for all $v_1, v_2 \in V$ and $v_1', v_2' \in V'$. We denote $(V \otimes V', b \otimes b') := (V, b) \otimes (V', b')$.
\end{maar}

We will also need to consider induction and restriction of bilinear $K[G]$-modules, as defined for example in \cite[Lemma, p. 1242]{GowWillemsGreenCorrespondence}, see also \cite[Section 4]{MurraySymmetricVertices}.

\begin{maar}
Let $H < G$ be a subgroup. For a bilinear $K[G]$-module $(V, b)$, its \emph{restriction} to $H$ is $\Res_H^G(V, b) := (\Res_H^G(V), b)$.
\end{maar}

\begin{maar}
Let $H < G$. For a bilinear $K[H]$-module $(L, b)$, the \emph{bilinear $K[G]$-module induced by $(L,b)$} is $\Ind_H^G(L, b) := (\Ind_H^G(L), a)$, where $\Ind_H^G(L) = K[G] \otimes_{K[H]} L$ is the $K[G]$-module induced by $L$ and $$a(g_1 \otimes \ell_1, g_2 \otimes \ell_2) = \begin{cases}b(g_2^{-1}g_1\ell_1,\ell_2), &\text{ if } g_1H = g_2H. \\ 0, &\text{ if } g_1H \neq g_2H. \end{cases}$$ for all $g_1,g_2 \in G$ and $\ell_1,\ell_2 \in L$.
\end{maar}

\begin{lemma}\label{lemma:bilinearindres}

Let $H < G$. Let $(L, b)$ be a bilinear $K[H]$-module and $(W, b')$ a bilinear $K[G]$-module. Then $$\Ind_H^G(L, b) \otimes (W, b') \cong \Ind_H^G(L \otimes \Res_H^G(W), b \otimes b')$$ as bilinear $K[G]$-modules. 

\end{lemma}

\begin{proof}
The corresponding result for $K[G]$-modules is a basic result \cite[Lemma 5 (5), p. 57]{Alperin}, and one can see that the map $\theta: \Ind_H^G(L) \otimes W \rightarrow \Ind_H^G(L \otimes \Res_H^G(W))$ defined by $(g \otimes \ell) \otimes w \mapsto g \otimes (\ell \otimes g^{-1}w)$ for all $g \in G$, $\ell \in L$, and $w \in W$, is an isomorphism of $K[G]$-modules. A straightforward check shows that $\theta$ is also an isometry with respect to the bilinear forms on $\Ind_H^G(L, b) \otimes (W, b')$ and $\Ind_H^G((L, b) \otimes \Res_H^G(W, b'))$.\end{proof}

\begin{maar}\label{def:pairedmodule}
Let $M$ be a $K[G]$-module. The \emph{paired module associated with $M$} is the bilinear $K[G]$-module $(M \oplus M^*, a)$, where $$a(v+f,v'+f') = f(v')+f'(v)$$ for all $v, v' \in M$ and $f,f' \in M^*$.
\end{maar}

Note that the paired module associated with a $K[G]$-module $M$ is always a non-degenerate alternating bilinear $K[G]$-module.

\begin{lemma}\label{lemma:basicpairedproperty}
Let $(V, b)$ be a non-degenerate alternating bilinear $K[G]$-module. Then $(V, b)$ is a paired module if and only if there exists a totally singular decomposition $V = W \oplus W'$, where $W$ and $W'$ are $K[G]$-submodules of $V$. Furthermore, in this case $(V, b)$ is the paired module associated with $W$.
\end{lemma}

\begin{proof}
If $(V, b) = (M \oplus M^*, a)$ is a paired module as in Definition \ref{def:pairedmodule}, then $V = M \oplus M^*$ is a totally singular decomposition with respect to $b$. Conversely, suppose that $V$ admits a totally singular decomposition $V = W \oplus W'$ into $K[G]$-submodules $W$ and $W'$. For $w' \in W'$, define $\varphi_{w'} \in W^*$ by $\varphi_{w'}(w) = b(w',w)$ for all $w \in W$. Then it is straightforward to see that the map $w+w' \mapsto w+\varphi_{w'}$ is an isomorphism $(V,b) \rightarrow (W \oplus W^*, a)$ of bilinear $K[G]$-modules, where $(W \oplus W^*, a)$ is the paired module associated with $W$.\end{proof}

The following two lemmas are easy consequences of Lemma \ref{lemma:basicpairedproperty}.

\begin{lemma}\label{lemma:tensorpaired}
Let $(V, b)$ be a paired $K[G]$-module. Then for any bilinear $K[G]$-module $(W, b')$, the tensor product $(V, b) \otimes (W, b')$ is a paired $K[G]$-module.
\end{lemma}

\begin{lemma}\label{lemma:inducedpaired}
Let $H < G$ and let $(W, b)$ be a paired $K[H]$-module. Then $\Ind_H^G(W, b)$ is a paired $K[G]$-module. 
\end{lemma}

We finish this section with a proof of the following lemma from \cite[Example 2.1]{MurraySymmetricVertices}.

\begin{lemma}\label{lemma:threeorthogonals}
Let $(W,b)$ be a bilinear $K[G]$-module. Then as bilinear $K[G]$-modules $$(W,b) \perp (W,b) \perp (W,b) \cong (W,b) \perp (W \oplus W^*, a),$$ where $(W \oplus W^*, a)$ is the paired module associated with $W$.
\end{lemma}

\begin{proof}Let $V = (W,b) \perp (W,b) \perp (W,b)$. It is straightforward to see that the diagonal subspace $Z = \{(w,w,w) : w \in W\}$ is non-degenerate, and that $Z \cong (W,b)$ as bilinear $K[G]$-modules. The orthogonal complement of $Z$ in $V$ is $Z^\perp = Z_1 \oplus Z_2$, where $Z_1 = \{(w,w,0) : w \in W \}$ and $Z_2 = \{(w,0,w) : w \in W \}$. Both $Z_1$ and $Z_2$ are totally singular and $Z_1 \cong W \cong Z_2$ and $K[G]$-modules. Thus by Lemma \ref{lemma:basicpairedproperty}, as a bilinear $K[G]$-module $Z^\perp$ is the paired module associated with $W$. Since $V = Z \perp Z^\perp$, the lemma follows.\end{proof}

\section{Unipotent classes in $\Sp(V)$}\label{section:unipclassesinfo}

Throughout this section, we denote by $u$ a generator of a cyclic $2$-group of order $q > 1$, and denote by $X$ the element $u-1$ of $K[u]$. Recall (Section \ref{section:notation}) that we denote the indecomposable $K[u]$-modules by $V_1$, $\ldots$, $V_q$, where $\dim V_i = i$ and $u$ acts on $V_i$ as a full $i \times i$ Jordan block.

For the symplectic groups $\Sp(V,b)$, the conjugacy classes of unipotent elements of order at most $q$ correspond to the isomorphism classes of non-degenerate alternating bilinear $K[u]$-modules. This is the basic approach taken in \cite{Hesselink}, where Hesselink classifies the unipotent conjugacy classes of $\Sp(V,b)$ in terms of orthogonally indecomposable bilinear $K[u]$-modules. We give an explicit construction of these modules in the following definitions.



\begin{maar}\label{def:wnmod}Let $d > 0$. We define $W(d)$ to be the paired module (Definition \ref{def:pairedmodule}) $(V_d \oplus V_d^*, a)$ associated with $V_d$.\end{maar}


\begin{maar}\label{def:vnmod}Let $d > 0$ be an even integer, say $d = 2k$. Fix a basis $e_1$, $\ldots$, $e_{d}$ of the $K[u]$-module $V_d$ such that \begin{align*}
ue_1 &= e_1, \\
ue_i &= e_i + e_{i-1} + \cdots + e_1 \text{ for all } 2 \leq i \leq k+1, \\
ue_i &= e_i + e_{i-1} \text{ for all } k+1 < i \leq d.
\end{align*} We define $V(d)$ to be the bilinear $K[u]$-module $(V_{d}, b)$ where $b(e_i, e_j) = 1$ if $i+j = d+1$ and $0$ otherwise.\end{maar}

Here $W(d)$ is orthogonally indecomposable by \cite[Section 2.3]{PforteMurray}, while $V(d)$ in Definition \ref{def:vnmod} is orthogonally indecomposable since it is indecomposable as a $K[u]$-module.

We note that Definition \ref{def:vnmod} is the same as \cite[Section 6.1, p. 91]{LiebeckSeitzClass}, and describes the action of a regular unipotent element of $\Sp(V, b)$ on the basis $(e_i)$ of $V$. More specifically, Definition \ref{def:vnmod} describes the action of the product $x_{\alpha_1}(1) \cdots x_{\alpha_k}(1)$ of fundamental root elements of $\Sp(V, b)$.


To describe the conjugacy classes in $\Sp(V,b)$, we will first need the following result from \cite{Hesselink}.

\begin{lause}[{\cite[Proposition 3.5]{Hesselink}}]\label{thm:orthogonallyindecomposables}
Up to isomorphism, the orthogonally indecomposable non-degenerate alternating bilinear $K[u]$-modules are $V(d)$ ($d$ even) and $W(d)$.
\end{lause}

As a consequence of Theorem \ref{thm:orthogonallyindecomposables}, each non-degenerate alternating bilinear $K[u]$-module has a certain normal form which is described in the next theorem. We will call this the \emph{Hesselink normal form}.

\begin{lause}\label{thm:hesselinkform}
Let $(V, b)$ a non-degenerate alternating bilinear $K[u]$-module. Let $0 < d_1 < \cdots < d_t$ be the Jordan block sizes of $u$ on $V$, and for $1 \leq i \leq t$ let $n_i > 0$ be the multiplicity of $d_i$ in $V$. 

There exists a unique sequence $W_1$, $\ldots$, $W_t$ of non-degenerate alternating bilinear $K[u]$-modules such that $V \cong W_1 \perp \cdots \perp W_t$ and the following hold for all $1 \leq i \leq t$:
	\begin{enumerate}[\normalfont (i)]
		\item If $d_i$ is odd, then $W_i = W(d_i)^{n_i/2}$.
		\item If $d_i$ is even, then $W_i = W(d_i)^{n_i/2}$ or $W_i = V(d_i)^{n_i}$.
	\end{enumerate}
\end{lause}

The normal form in Theorem \ref{thm:hesselinkform} is the same as that described by Hesselink in \cite[3.7]{Hesselink}. One can see this using \cite[3.7 -- 3.9]{Hesselink}, but we will give a proof later in this section to keep this paper more self-contained.

Note that the conjugacy class of a unipotent element $u \in \Sp(V,b)$ is determined by the Hesselink normal form of $u$ on $(V,b)$. 

\begin{remark}There is also a \emph{distinguished normal form} defined in \cite[p. 61]{LiebeckSeitzClass}, which is different from the Hesselink normal form and useful for describing centralizers of unipotent elements in $\Sp(V,b)$. Translating between these two normal forms is straightforward, using the fact that $V(2d)^3 \cong W(2d) \perp V(2d)$ as bilinear $K[u]$-modules (Lemma \ref{lemma:threeorthogonals}). In this paper, we will only use the Hesselink normal form.\end{remark}


Following \cite[2.6, p. 20]{Spaltenstein}, we make the following definition.

\begin{maar}
Let $(V, b)$ be a bilinear $K[u]$-module. We define $\varepsilon_{V, b}: \Z_{\geq 1} \rightarrow \{0,1\}$ by $\varepsilon_{V, b}(d) = 0$ if $b(X^{d-1}v,v) = 0$ for all $v \in V$ such that $X^dv = 0$, and $\varepsilon_{V, b}(d) = 1$ otherwise.
\end{maar}

It turns out that the Hesselink normal form of $u \in \Sp(V,b)$ (and hence its conjugacy class in $\Sp(V,b)$) is determined by the Jordan normal form of $u$ and the values of $\varepsilon_{V, b}$ on the Jordan block sizes of $u$. This is a well known result which is stated in \cite[2.6, p. 20]{Spaltenstein}.
	

\begin{lause}\label{thm:hesselinkepsilon}
Suppose that $u \in \Sp(V, b)$, and set $\varepsilon := \varepsilon_{V, b}$. Let $0 < d_1 < \cdots < d_t$ be the Jordan block sizes of $u$ on $V$, with block size $d_i$ having multiplicity $n_i > 0$. Let $(V,b) \cong W_1 \perp \cdots \perp W_t$ be the Hesselink normal form of $u$ on $(V, b)$ as in Theorem \ref{thm:hesselinkform}. Then for all $1 \leq i \leq t$, we have $W_i \cong W(d_i)^{n_i/2}$ if $\varepsilon_{V,b}(d_i) = 0$ and $W_i \cong V(d_i)^{n_i}$ if $\varepsilon_{V,b}(d_i) = 1$.

In particular, the Hesselink normal form of $u$ on $(V,b)$ is uniquely determined by the tuple $({d_1}_{\varepsilon(d_1)}^{n_1}, \ldots, {d_t}_{\varepsilon(d_t)}^{n_t})$.
\end{lause}


	%
	
Since our main results rely on Theorem \ref{thm:hesselinkepsilon}, we will give a proof in what follows. First we need a few lemmas which will also be useful later for the computation of $\varepsilon_{V, b}$ for various bilinear $K[u]$-modules $(V,b)$.
							
\begin{lemma}\label{lemma:basicXdlinear}Suppose that $u \in \SL(V)$, let $b$ be a $u$-invariant alternating bilinear form on $V$, not necessarily non-degenerate. Let $d > 0$ be an integer. Then:

\begin{enumerate}[\normalfont (i)]
\item For all $v, w \in \operatorname{Ker} X^d$ and $1 \leq i,j \leq d-1$ with $i+j = d$, we have $b(X^{i-1}v,X^jw) = b(X^iv,X^{j-1}w)$.
\item For all $v, w \in \operatorname{Ker} X^d$ and $1 \leq i,j \leq d$ with $i+j \geq d$, we have $b(X^iv,X^jw) = 0$.
\item For all $v,w \in \operatorname{Ker} X^d$, we have $b(X^{d-1}v,w) = b(v, X^{d-1}w)$.
\item The map $v \mapsto b(X^{d-1}v,v)$ is additive on $\operatorname{Ker} X^d$.
\end{enumerate}
\end{lemma}

\begin{proof}Set $(e_i)_{1 \leq i \leq d} = (X^{d-i} v)_{1 \leq i \leq d}$ and $(f_j)_{1 \leq j \leq d} = (X^{d-j} w)_{1 \leq j \leq d}$. For all $1 \leq i,j \leq d-1$ such that $i+j = d$, it follows from \cite[Lemme II.6.10 b), pg. 99]{Spaltenstein} that $b(e_{d-i+1}, f_{d-j}) + b(e_{d-i}, f_{d-j+1}) = 0$, which gives (i). For all $1 \leq i,j \leq d$ with $i+j \geq d$, we have $b(e_{d-i},f_{d-j}) = 0$ by \cite[Lemme II.6.10 a), pg. 99]{Spaltenstein}, which gives (ii).

For claim (iii), using (i) repeatedly we find that $$b(X^{d-1}v,w) = b(X^{d-2}v,Xw) = \cdots = b(v, X^{d-1}w)$$ for all $v,w \in \operatorname{Ker} X^d$. Claim (iv) is an easy consequence of (iii).\end{proof}

\begin{lemma}\label{lemma:epsilonequivalence}
Let $(V, b)$ be an alternating bilinear $K[u]$-module, not necessarily non-degenerate. The following statements are equivalent:
\begin{enumerate}[\normalfont (i)]
\item $b(X^{d-1}v,v) \neq 0$ for some $v \in V$ such that $X^d v = 0$.
\item $d$ is even, and $V(d)$ occurs as an orthogonal direct summand of $V$.
\item $d$ is even, and for any decomposition $V = W_1 \perp \cdots \perp W_t$ into orthogonally indecomposable $K[u]$-submodules, we have $(W_i,b) \cong V(d)$ for some $i$.
\end{enumerate}
\end{lemma}

\begin{proof}
We first show that (i) and (ii) are equivalent. Suppose that $b(X^{d-1}v,v) \neq 0$ for some $v \in V$ such that $X^d v = 0$. Let $W$ be the subspace of $V$ spanned by $v, Xv, \ldots, X^{d-1}v$, so now $W \cong V_d$ as $K[u]$-modules. It follows from Lemma \ref{lemma:basicXdlinear} (i) -- (ii) that the matrix of $b|_{W \times W}$ with respect to the basis $v, Xv, \ldots, X^{d-1}v$ is of the form \begin{equation}\label{eq:matrixofV(d)}\begin{pmatrix}* &  & \lambda \\  & \reflectbox{$\ddots$} & \\ \lambda & & 0 \end{pmatrix}\end{equation} where $\lambda = b(X^{d-1}v,v)$. Since $\lambda \neq 0$, it follows that $b|_{W \times W}$ is non-degenerate, so $d$ must be even since $b$ is alternating. Furthermore, we have $V = W \perp W^\perp$ and $W \cong V(d)$ since $W$ is non-degenerate, so (ii) holds. 

Conversely, suppose that $d$ is even and $V = W \perp W'$ with $W \cong V(d)$. Choose some $v \in W$ such that $X^{d-1}v \neq 0$. Then $v, Xv, \ldots, X^{d-1}v$ is a basis of $W$, and the matrix of $b|_{W \times W}$ with respect to this basis is as in~\eqref{eq:matrixofV(d)}, with $\lambda = b(X^{d-1}v,v)$. Thus we must have $b(X^{d-1}v,v) \neq 0$ since $W$ is a non-degenerate subspace. We conclude then that (i) and (ii) are equivalent.

It is obvious that (iii) implies (ii). Next we will show that (i) implies (iii), which will complete the proof. Suppose that (i) holds, and let $v \in V$ be such that $b(X^{d-1}v,v) \neq 0$ and $X^d v = 0$. Let $V = W_1 \perp \cdots \perp W_t$ be a decomposition into orthogonally indecomposable $K[u]$-submodules. We can write $v = w_1 + \cdots + w_t$ with $w_i \in W_i$ for all $1 \leq i \leq t$. Now $X^d w_i = 0$ for all $1 \leq i \leq t$, so by Lemma \ref{lemma:basicXdlinear} (iv) $$b(X^{d-1}v,v) = b(X^{d-1}w_1,w_1) + \cdots + b(X^{d-1}w_t,w_t).$$ Thus $b(X^{d-1}w_i,w_i) \neq 0$ for some $1 \leq i \leq t$. From the equivalence of (i) and (ii), it follows that $d$ is even and $(W_i,b)$ has $V(d)$ as an orthogonal direct summand. Since $(W_i, b)$ is orthogonally indecomposable, this proves that $(W_i, b) \cong V(d)$.\end{proof}

We can now prove Theorem \ref{thm:hesselinkform} and Theorem \ref{thm:hesselinkepsilon}.

\begin{proof}[Proof of Theorem {\ref{thm:hesselinkform}}]
Let $(V,b)$ be a non-degenerate alternating bilinear $K[u]$-module. One can write $V$ as an orthogonal direct sum $V = Z_1 \perp \cdots \perp Z_t$ of orthogonally indecomposable bilinear $K[u]$-modules, and by Theorem \ref{thm:orthogonallyindecomposables} each $Z_i$ is isomorphic to $V(d_i)$ ($d_i$ even) or $W(d_i)$ for some $d_i > 0$.

Note that $V(2d)^3 \cong W(2d) \perp V(2d)$ by Lemma \ref{lemma:threeorthogonals}. From this it follows that for $a, b \geq 0$, we have $$V(2d)^a \perp W(2d)^b \cong \begin{cases} V(2d)^{a+2b}, & \text{ if } a > 0. \\ W(2d)^b, & \text{ if } a = 0.\end{cases}$$ as bilinear $K[u]$-modules. Thus by collecting the orthogonal direct summands $Z_i$ with equal Jordan block sizes, we get the Hesselink normal form on $V$.

For uniqueness, let $0 < d_1 < \cdots < d_t$ be the Jordan block sizes of $u$ on $V$, with block size $d_i$ having multiplicity $n_i > 0$. Write $V = W_1 \perp \cdots \perp W_t$, where for all $i$ we have $W_i \cong W(d_i)^{n_i/2}$ or $W_i \cong V(d_i)^{n_i}$. By the equivalence of (i) and (iii) Lemma \ref{lemma:epsilonequivalence}, we see that $W_i \cong V(d_i)^{n_i}$ if and only if $d_i$ is even and $b(X^{d_i-1}v,v) \neq 0$ for some $v \in V$ such that $X^{d_i}v = 0$. From this we get the uniqueness of the Hesselink normal form.\end{proof}

\begin{proof}[Proof of Theorem {\ref{thm:hesselinkepsilon}}]The proof follows from the argument at the end of the previous proof. Indeed, by Lemma \ref{lemma:epsilonequivalence}, in Theorem \ref{thm:hesselinkform} we have $W_i \cong W(d_i)^{n_i/2}$ if $\varepsilon_{V,b}(d_i) = 0$ and $W_i \cong V(d_i)^{n_i}$ if $\varepsilon_{V,b}(d_i) = 1$.\end{proof}

\begin{lemma}\label{lemma:basicepsilonoddeven}
Let $(V, b)$ be a non-degenerate alternating bilinear $K[u]$-module. Then the following statements hold:
\begin{enumerate}[\normalfont (i)]
\item If $d$ is odd, then $\varepsilon_{V, b}(d) = 0$.
\item If $V_d$ occurs with odd multiplicity in $V$, then $d$ is even and $\varepsilon_{V, b}(d) = 1$.
\end{enumerate}
\end{lemma}

\begin{proof}Claim (i) is immediate from Lemma \ref{lemma:epsilonequivalence}. For (ii), suppose that $V_d$ occurs with odd multiplicity in $V$. Since $V_d$ has multiplicity $2$ in $W(d)$, it follows that $d$ is even and $(V,b)$ must have $V(d)$ as an orthogonal direct summand. Thus $\varepsilon_{V, b}(d) = 1$ by Lemma \ref{lemma:epsilonequivalence}.\end{proof}

Let $G = \SL(V)$ and set $n = \dim V$. In one of our main results, Theorem \ref{thm:MAINTHMA}, we describe the Hesselink normal form of any unipotent element $u \in G$ on the irreducible $K[G]$-module $L_{G}(\varpi_1 + \varpi_{n-1})$. In the proof, we make use of the fact that up to scalar multiples there is a unique non-zero alternating bilinear form on $V \otimes V^*$, and an isomorphism $L_G(\varpi_1 + \varpi_{n-1}) \cong \langle v \rangle^\perp / \langle v \rangle$ where $v \in V \otimes V^*$ is a $G$-fixed point --- see Section \ref{section:altformtq}. A natural approach then is to first consider the action of $u$ on $V \otimes V^*$ and use it to deduce information about the action of $u$ on $\langle v \rangle^\perp / \langle v \rangle$. For this we need the following general lemma, which will also be useful in the proof of our main result concerning Hesselink normal forms on $L_G(\varpi_2)$ for $G = \Sp(V,b)$ (Theorem \ref{thm:MAINTHMC}).

\begin{lemma}\label{lemma:vperpmodvdescription}
Let $(V, b)$ be a non-degenerate alternating bilinear $K[u]$-module and let $v \in V$ be a non-zero vector fixed by $u$. Write $V \cong \oplus_{d \geq 1} V_d^{\lambda(d)}$ and $\langle v \rangle^\perp / \langle v \rangle \cong \oplus_{d \geq 1} V_d^{\lambda'(d)}$ as $K[u]$-modules, where $\lambda(d), \lambda'(d) \geq 0$ for all $d \geq 1$. Set $\varepsilon := \varepsilon_{V, b}$ and $\varepsilon' := \varepsilon_{\langle v \rangle^\perp / \langle v \rangle, b}$.

Let $m \geq 1$ be such that $\operatorname{Ker} X_V^{m-1} \subseteq \langle v \rangle^\perp$ and $\operatorname{Ker} X_V^m \not\subseteq \langle v \rangle^\perp$. The following statements hold:

\begin{enumerate}[\normalfont (i)]
\item For all $d \geq 0$, we have $\operatorname{Ker} X_V^d \subseteq \langle v \rangle^\perp$ if and only if $\langle v \rangle \subseteq \operatorname{Im} X_V^d$. In particular, there exists $\delta \in V$ such that $X^{m-1} \delta = v$, and $v \not\in \operatorname{Im} X_V^d$ for $d \geq m$.

\item If $m = 1$, then:
	\begin{enumerate}[\normalfont (a)]
		\item $\lambda'(1) = \lambda(1) - 2$, and $\lambda'(d) = \lambda(d)$ for all $d > 1$.
		\item $\varepsilon'(d) = \varepsilon(d)$ for all $d \geq 1$.
	\end{enumerate}
	
\item If $m > 1$ and $\delta \in \langle v \rangle^\perp$, then: 
	\begin{enumerate}[\normalfont (a)]
		\item $\lambda'(m) = \lambda(m) - 2$, $\lambda'(m-1) = \lambda(m-1)+2$, and $\lambda'(d) = \lambda(d)$ for all $d \neq m-1,m$.
		\item $\varepsilon'(d) = \varepsilon(d)$ for all $d \neq m$.
	\end{enumerate}

\item If $m > 1$ and $\delta \not\in \langle v \rangle^\perp$, then: 
	\begin{enumerate}[\normalfont (a)]
		\item $\lambda'(m) = \lambda(m) - 1$, $\lambda'(m-2) = \lambda(m-2) + 1$ (if $m > 2$), and $\lambda'(d) = \lambda(d)$ for all $d \neq m-2,m$.
		\item $\varepsilon(d) = \varepsilon'(d)$ for all $d \neq m-2,m$. 
		\item $\varepsilon(m) = 1$.
		\item $\varepsilon'(m-2) = 1$ (if $m > 2$).		
	\end{enumerate}
\end{enumerate}
\end{lemma}

\begin{proof}
Since $b$ is non-degenerate, we have $\langle v \rangle^\perp \cong (V/\langle v \rangle)^*$ as $K[u]$-modules. Every $K[u]$-module is self-dual, so in fact $\langle v \rangle^\perp \cong V/\langle v \rangle$ as $K[u]$-modules. Then with Lemma \ref{jordanrestriction} and Lemma \ref{jordanquotient}, we conclude that for all $d \geq 0$, we have $\operatorname{Ker} X^d \subseteq \langle v \rangle^\perp$ if and only if $\langle v \rangle \subseteq \operatorname{Im} X^d$, which proves (i). Let $\delta \in V$ be such that $X^{m-1} \delta = v$.

For claims (ii) -- (iv), we first consider the description of $\lambda'$. By (i) and Lemma \ref{jordanquotient}, we find that $V/\langle v \rangle \cong \oplus_{d \geq 1}V_d^{\mu(d)}$ as $K[u]$-modules, where $\mu(d) \geq 0$ are given as follows:
	\begin{itemize}
		\item If $m = 1$, then $\mu(1) = \lambda(1) - 1$ and $\mu(d) = \lambda(d)$ for all $d > 1$.
		\item If $m > 1$, then $\mu(m) = \lambda(m) - 1$, $\mu(m-1) = \lambda(m-1)+1$, and $\mu(d) = \lambda(d)$ for all $d \neq m,m-1$.
	\end{itemize}
	
We shall apply Lemma \ref{jordanrestriction} to $V/\langle v \rangle$ and $\langle v \rangle^\perp / \langle v \rangle$ in order to describe $\lambda'$. First note that $v \not\in \operatorname{Im}X^d$ for all $d \geq m$ by (i), so \begin{equation}\label{eq:kernelXbigd}\operatorname{Ker} X_{V/\langle v \rangle}^d = \operatorname{Ker} X_{V}^d / \langle v \rangle\end{equation} for all $d \geq m$. If $0 \leq d \leq m-1$, then $X^d(X^{m-1-d} \delta) = v$ and any solution to $X^dv' = v$ is unique modulo $\operatorname{Ker} X^d$, so \begin{equation}\label{eq:kernelXsmalld}\operatorname{Ker} X_{V/\langle v \rangle}^d = \left(\operatorname{Ker} X_{V}^d \oplus \langle X^{m-1-d}\delta \rangle \right) / \langle v \rangle\end{equation} for all $0 \leq d \leq m-1$.

Note that $u$ acts trivially on the $1$-dimensional $K[u]$-module $V/\langle v \rangle^\perp$, so $\operatorname{Im} X \subseteq \langle v \rangle^\perp$. Thus~\eqref{eq:kernelXsmalld} implies that
\begin{align}\label{eq:statementeq1}\operatorname{Ker} X_{V/\langle v \rangle}^d \subseteq \langle v \rangle^\perp / \langle v \rangle & \text{ for all } 0 \leq d < m-1.\end{align}
Furthermore, by~\eqref{eq:kernelXsmalld} and~\eqref{eq:kernelXbigd} the following hold: \begin{align}\label{eq:statementeq2}\operatorname{Ker} X_{V/\langle v \rangle}^{m-1} \subseteq \langle v \rangle^\perp / \langle v \rangle & \text{ if and only if } \delta \in \langle v \rangle^\perp. \\
\label{eq:statementeq3}\operatorname{Ker} X_{V/\langle v \rangle}^{m} \not\subseteq \langle v \rangle^\perp / \langle v \rangle &.\end{align}

Now combining Lemma \ref{jordanrestriction}, statements~\eqref{eq:statementeq1} --~\eqref{eq:statementeq3}, and the description of $\mu(d)$ above, it follows easily that $\lambda'$ is given as described in (ii) -- (iv). This completes the proof of the claims for $\lambda'$.

For the rest of the proof we will consider the claims about $\varepsilon$ and $\varepsilon'$ in (ii) -- (iv). Let $\hat{\delta}$ be such that $X^m \hat{\delta} = 0$ and $\hat{\delta} \not\in \langle v \rangle^\perp$. Then \begin{equation}\label{eq:epsilonXsmalld}\operatorname{Ker} X_{V}^d = \operatorname{Ker} X_{\langle v \rangle^\perp}^d \oplus \langle \hat{\delta} \rangle\end{equation} for all $d \geq m$. If $d > m$, then $b(X^{d-1}\hat{\delta}, \hat{\delta}) = 0$ since $X^m\hat{\delta} = 0$. Thus it follows from~\eqref{eq:epsilonXsmalld} and Lemma \ref{lemma:basicXdlinear} (iv) that $\varepsilon'(d) = \varepsilon(d)$ for all $d > m$.

We always have $\varepsilon'(1) = \varepsilon(1) = 0$, so if $m = 1$, then $\varepsilon'(d) = \varepsilon(d)$ for all $d \geq 1$. This proves (ii), so we will assume for the rest of the proof that $m > 1$. 

We have $\operatorname{Im} X \subseteq \langle v \rangle^\perp$, and $\operatorname{Ker} X_V^d \subseteq \langle v \rangle^\perp$ for all $1 \leq d \leq m-1$, so \begin{equation}\label{eq:vperpvsmalldKERX}\operatorname{Ker} X_{\langle v \rangle^\perp / \langle v \rangle}^d = \left( \operatorname{Ker} X_{V}^d \oplus \langle X^{m-d-1} \delta \rangle \right) / \langle v \rangle \end{equation} for all $1 \leq d \leq m-2$. 

Note that $X^m \delta = 0$ since $v$ is fixed by $u$. Thus if $1 \leq d < m-2$, then $$b(X^{d-1}(X^{m-d-1}\delta), X^{m-d-1}\delta) = b(X^{m-2}\delta, X^{m-d-1}\delta) = 0$$ by Lemma \ref{lemma:basicXdlinear} (ii). It follows then from~\eqref{eq:vperpvsmalldKERX} and Lemma \ref{lemma:basicXdlinear} (iv) that $\varepsilon'(d) = \varepsilon(d)$ for all $1 \leq d < m-2$.

So far we have shown that $\varepsilon'(d) = \varepsilon(d)$ for all $d \neq m-2,m-1,m$, as claimed by (iii) and (iv). For $d = m-2,m-1,m$, we will consider the two cases (iii) and (iv) separately.\newline

\noindent \emph{Case (iii): $\delta \in \langle v \rangle^\perp$.}\newline

\noindent In this case $b(X^{m-1}\delta, \delta) = b(v, \delta) = 0$. Since $$\operatorname{Ker} X_{\langle v \rangle^\perp / \langle v \rangle}^{m-1} = \left( \operatorname{Ker} X_{V}^{m-1} \oplus \langle \delta \rangle \right) / \langle v \rangle$$ by~\eqref{eq:vperpvsmalldKERX}, it follows from Lemma \ref{lemma:basicXdlinear} (iv) that $\varepsilon'(m-1) = \varepsilon(m-1)$.

If $m > 2$, with Lemma \ref{lemma:basicXdlinear} (i) we get \begin{equation}\label{eq:eqofdeltav}b(X^{m-3}(X\delta), X\delta) = b(X^{m-2}\delta, X\delta) = b(X^{m-1}\delta, \delta) = b(v, \delta).\end{equation} Thus $b(X^{m-3}(X\delta), X\delta) = 0$. Since $$\operatorname{Ker} X_{\langle v \rangle^\perp / \langle v \rangle}^{m-2} = \left( \operatorname{Ker} X_{V}^{m-2} \oplus \langle X \delta \rangle \right) / \langle v \rangle$$ by~\eqref{eq:vperpvsmalldKERX}, it follows from Lemma \ref{lemma:basicXdlinear} (iv) that $\varepsilon'(m-2) = \varepsilon(m-2)$. Hence $\varepsilon'(d) = \varepsilon(d)$ for all $d \neq m$, as claimed in (iii).\newline

\noindent \emph{Case (iv): $\delta \not\in \langle v \rangle^\perp$.}\newline 

\noindent In this case $b(X^{m-1}\delta, \delta) = b(v, \delta) \neq 0$, so $\varepsilon(m) = 1$. Thus $m$ must be even and $\varepsilon'(m-1) = \varepsilon(m-1) = 0$ by Lemma \ref{lemma:basicepsilonoddeven}. 

If $m > 2$, then we see from~\eqref{eq:eqofdeltav} that $b(X^{m-3}(X\delta), X\delta) \neq 0$. Thus $\varepsilon'(m-2) = 1$, since $X\delta + \langle v \rangle \in \operatorname{Ker} X_{\langle v \rangle^\perp / \langle v \rangle}^{m-2}$. This completes the proof of (iv) and the lemma.\end{proof}

We finish this section by describing the induction and restriction of orthogonally indecomposable bilinear $K[u]$-modules. First we need a small lemma, which will also be useful later.

\begin{lemma}\label{lemma:pairedjordan}
Let $(Z, b)$ be a non-degenerate alternating bilinear $K[u]$-module. Then the following statements are equivalent:

\begin{enumerate}[\normalfont (i)]
\item $(Z,b)$ is a paired module.
\item There exists a totally singular decomposition $Z = W \oplus W'$ where $W$ and $W'$ are $K[u]$-submodules of $Z$.
\item For all $d > 0$ and $v \in Z$ such that $X^d v = 0$, we have $b(X^{d-1}v,v) = 0$. 
\end{enumerate}

Furthermore, if $Z = W \oplus W'$ as in (ii) and $W \cong V_{d_1} \oplus \cdots \oplus V_{d_t}$ as $K[u]$-modules, then $$(Z, b) \cong W(d_1) \perp \cdots \perp W(d_t)$$ as bilinear $K[u]$-modules.\end{lemma}

\begin{proof}
The equivalence of (i) and (ii) is given by Lemma \ref{lemma:basicpairedproperty}. We show that (ii) implies (iii). Let $Z = W \oplus W'$ be a totally singular decomposition, where $W$ and $W'$ are $K[u]$-submodules of $Z$. For any $v \in \operatorname{Ker} X^d$, we can write $v = w + w'$ where $w \in \operatorname{Ker} X_W^d$ and $w' \in \operatorname{Ker} X_{W'}^d$. Then $b(X^{d-1}v,v) = b(X^{d-1}w,w) + b(X^{d-1}w',w')$ by Lemma \ref{lemma:basicXdlinear} (iv). Since $W$ and $W'$ are totally singular, it follows that $b(X^{d-1}v,v) = 0$.

Next we show that (iii) implies (i). If (iii) holds, then $(Z,b)$ does not have any orthogonal direct summands of the form $V(m)$ by Lemma \ref{lemma:epsilonequivalence}. It follows from Theorem \ref{thm:orthogonallyindecomposables} that $(Z, b) \cong W(d_1) \perp \cdots \perp W(d_t)$ for some integers $d_i > 0$, and consequently $(Z,b)$ is a paired module since each $W(d_i)$ is.

For the last statement of the lemma, suppose that $Z = W \oplus W'$ as in (ii) and $W \cong V_{d_1} \oplus \cdots \oplus V_{d_t}$ as $K[u]$-modules. Then $W = W^\perp$ since $Z = W \oplus W'$ is a totally singular decomposition, so $W' \cong V/W = V/W^\perp \cong W^*$. Hence $$Z \cong W \oplus W^* \cong V_{d_1}^2 \oplus \cdots \oplus V_{d_t}^2.$$ As in the previous paragraph, as a bilinear $K[u]$-module $(Z,b)$ decomposes into an orthogonal direct sum involving only summands of the form $W(d)$, so we must have $$(Z, b) \cong W(d_1) \perp \cdots \perp W(d_t)$$ as bilinear $K[u]$-modules.\end{proof}

\begin{lemma}\label{lemma:indresforcyclicALPHA} Let $\alpha > 0$ be such that $2^{\alpha} \leq q$ and let $0 < d \leq q/2^{\alpha}$. Then:
\begin{enumerate}[\normalfont (i)]
\item If $d$ is even, then $\Ind_{\langle u^{2^{\alpha}} \rangle}^{\langle u \rangle}(V(d)) \cong V(2^{\alpha}d)$.
\item $\Ind_{\langle u^{2^{\alpha}} \rangle}^{\langle u \rangle}(W(d)) \cong W(2^{\alpha}d)$.
\item $\Res_{\langle u^2 \rangle}^{\langle u \rangle}(V(2d)) \cong \begin{cases} V(d)^2, &\text{ if } d \text{ is even}.\\ W(d), &\text{ if } d \text{ is odd}.\end{cases}$
\item Write $d = a2^{\alpha-1} + r$ for $0 \leq r < 2^{\alpha-1}$. Then $$\Res_{\langle u^{2^{\alpha}} \rangle}^{\langle u \rangle}(V(2d)) \cong \begin{cases} V(d/2^{\alpha-1})^{2^{\alpha}}, &\text{ if } 2^{\alpha} \mid d.\\ W(a+1)^r \perp W(a)^{2^{\alpha-1}-r}, &\text{ if } 2^{\alpha} \nmid d.\end{cases}$$ where we define $W(0) = 0$.
\end{enumerate}
\end{lemma}

\begin{proof}For (i), note that by Lemma \ref{lemma:inductionpowerp} we have $\Ind_{\langle u^{2^{\alpha}} \rangle}^{\langle u \rangle}(V_d) \cong V_{2^{\alpha}d}$ as $K[u]$-modules. Thus from Theorem \ref{thm:orthogonallyindecomposables} it is clear that $\Ind_{\langle u^{2^{\alpha}} \rangle}^{\langle u \rangle}(V(d)) \cong V(2^{\alpha}d)$. 

By Lemma \ref{lemma:inducedpaired} the induced module $\Ind_{\langle u^{2^{\alpha}} \rangle}^{\langle u \rangle}(W(d))$ is a paired $K[u]$-module, so (ii) follows from Lemma \ref{lemma:pairedjordan} and the fact that $\Ind_{\langle u^{2^{\alpha}} \rangle}^{\langle u \rangle}(V_d) \cong V_{2^{\alpha}d}$ as $K[u]$-modules.

Claim (iii) is \cite[Lemma 4.1]{LawtherOuter}. For claim (iv), note that the case $\alpha = 1$ is the same as (iii). For $\alpha > 1$, we prove by induction on $\alpha$ that $\Res_{\langle u^{2^{\alpha}} \rangle}^{\langle u \rangle}(V(2d)) \cong V(d/2^{\alpha-1})^{2^{\alpha}}$ if $2^{\alpha} \mid d$. If $\alpha > 1$ and $2^{\alpha} \mid d$, then $\Res_{\langle u^{2^{\alpha-1}} \rangle}^{\langle u \rangle}(V(2d)) \cong V(d/2^{\alpha-2})^{2^{\alpha-1}}$ by induction. On the other hand $\Res_{\langle u^{2^{\alpha}} \rangle}^{\langle u^{2^{\alpha-1}} \rangle}(V(d/2^{\alpha-2})) \cong V(d/2^{\alpha-1})^2$ by (iii), so we conclude that $\Res_{\langle u^{2^{\alpha}} \rangle}^{\langle u \rangle}(V(2d)) \cong V(d/2^{\alpha-1})^{2^{\alpha}}$.

Next consider the case where $2^{\alpha} \nmid d$. We show first that $\Res_{\langle u^{2^{\alpha}} \rangle}^{\langle u \rangle}(V(2d))$ is a paired module. To this end, let $0 \leq \beta < \alpha$ be such that $2^{\beta} \mid d$ and $2^{\beta+1} \nmid d$. Then we have already shown that $\Res_{\langle u^{2^{\beta}} \rangle}^{\langle u \rangle}(V(2d)) \cong V(d/2^{\beta-1})^{2^{\beta}}$. Since $d/2^{\beta}$ is odd, it follows from (iii) that $\Res_{\langle u^{2^{\beta+1}} \rangle}^{\langle u^{2^{\beta}} \rangle} V(d/2^{\beta-1}) \cong W(d/2^{\beta})$ and so $\Res_{\langle u^{2^{\beta+1}} \rangle}^{\langle u \rangle}(V(2d)) \cong W(d/2^{\beta})^{2^{\beta}}$. Thus $\Res_{\langle u^{2^{\alpha}} \rangle}^{\langle u \rangle}(V(2d))$ is a paired module since $\Res_{\langle u^{2^{\beta+1}} \rangle}^{\langle u \rangle}(V(2d))$ is. Now $2d = a2^{\alpha} + 2r$, so by Lemma \ref{lemma:restrictionpowerp} $$\Res_{\langle u^{2^{\alpha}} \rangle}^{\langle u \rangle}(V_{2d}) \cong V_{a+1}^{2r} \oplus V_a^{2^{\alpha}-2r}$$ and (iv) follows from Lemma \ref{lemma:pairedjordan}.\end{proof}

\section{Tensor products of bilinear $K[u]$-modules}\label{section:tensorproductbilinear}

We keep the setup of the previous section, so let $u$ be a generator of a cyclic $2$-group of order $q > 1$ and denote $X = u-1$. 

In this section, we describe how to decompose tensor products of non-degenerate alternating bilinear $K[u]$-modules into orthogonally indecomposable summands. Clearly, it suffices to do this for the orthogonally indecomposable bilinear $K[u]$-modules, which (up to isomorphism) are of the form $V(2d)$ or $W(d)$ for some integer $d > 0$ (Theorem \ref{thm:orthogonallyindecomposables}). For tensor products with $W(d)$, the following proposition is an easy consequence of Lemma \ref{lemma:tensorpaired} and Lemma \ref{lemma:pairedjordan}.

\begin{prop}\label{prop:pairedtensorcase}
Let $0 < d,d' \leq q$. Suppose that $V_d \otimes V_{d'} \cong V_{d_1} \oplus \cdots \oplus V_{d_t}$. Then:
\begin{enumerate}[\normalfont (i)]
\item If $d'$ is even, then $W(d) \otimes V(d') \cong W(d_1) \perp \cdots \perp W(d_t)$.
\item $W(d) \otimes W(d') \cong W(d_1)^2 \perp \cdots \perp W(d_t)^2$.
\end{enumerate}
\end{prop}

Here Proposition \ref{prop:pairedtensorcase} describes the tensor products $W(d) \otimes V(d')$ and $W(d) \otimes W(d')$ in terms of indecomposable summands of $V_d \otimes V_{d'}$, which can be calculated with Theorem \ref{thm:tensordecompchar2}. 

For the rest of this section, we will consider the decomposition of $V(2d) \otimes V(2d')$ into orthogonally indecomposable summands, and Theorem \ref{thm:formtensorproducts} gives a complete answer in terms of the decomposition of $V_{2d} \otimes V_{2d'}$. We begin with a series of lemmas that deal with the case where $d$ and $d'$ are odd.

\begin{lemma}\label{lemma:inducedsingulardecomp}
Let $0 < \ell \leq q$ be an odd integer. Then there exists a non-degenerate alternating $u$-invariant bilinear form $a$ on $\Ind_{\langle u^2 \rangle}^{\langle u \rangle}(V_\ell)$ such that $(\Ind_{\langle u^2 \rangle}^{\langle u \rangle}(V_\ell), a) \cong V(2\ell)$ and $$\Ind_{\langle u^2 \rangle}^{\langle u \rangle}(V_\ell) = (1 \otimes V_\ell) \oplus (u \otimes V_\ell)$$ is a totally singular decomposition with respect to $a$.
\end{lemma}

\begin{proof}
Consider first $V = (V(2\ell), b)$ with a basis $e_1$, $\ldots$, $e_{2\ell}$ such that $b(e_i,e_j) = 1$ if $i+j = 2\ell+1$ and $0$ otherwise, and with the action of $u$ on the $e_i$ as in Definition \ref{def:vnmod}. 

Let $W = \langle e_2, e_4, \ldots, e_{2\ell} \rangle$. Then $W$ is $u^2$-invariant, since $u^2e_{2i} = e_{2i} + e_{2i-2} + \cdots + e_{2}$ for $2 \leq 2i \leq \ell+1$ and $u^2e_{2i} = e_{2i} + e_{2i-2}$ for $\ell+3 \leq 2i \leq 2\ell$. Furthermore, we claim that $W \cap u(W) = 0$. To this end, note that $W \cap u(W)$ is $u$-invariant since $u^2(W) \subseteq W$. On the other hand, there are no non-zero $u$-fixed points in $W \cap u(W)$ since there are none in $W$, so we must have $W \cap u(W) = 0$.

Therefore $V = W \oplus u(W)$ and this is a totally singular decomposition. We have $W \cong V_\ell$ as $\langle u^2 \rangle$-modules, so it follows from a basic property of induced modules \cite[Proof of Lemma 4, pp. 56-57]{Alperin} that there exists an isomorphism $\varphi: \Ind_{\langle u^2 \rangle}^{\langle u \rangle}(V_\ell) \rightarrow V$ of $K[u]$-modules with $\varphi(1 \otimes V_\ell) = W$ and $\varphi(u \otimes V_\ell) = u(W)$. Now we can define a non-degenerate $u$-invariant alternating bilinear form $a$ on $\Ind_{\langle u^2 \rangle}^{\langle u \rangle}(V_\ell)$ via $a(x,y) = b(\varphi(x), \varphi(y))$ for all $x,y \in \Ind_{\langle u^2 \rangle}^{\langle u \rangle}(V_\ell)$. It is clear that $(\Ind_{\langle u^2 \rangle}^{\langle u \rangle}(V_\ell), a) \cong V(2\ell)$, and furthermore $\Ind_{\langle u^2 \rangle}^{\langle u \rangle}(V_\ell) = (1 \otimes V_\ell) \oplus (u \otimes V_\ell)$ is a totally singular decomposition since $V = W \oplus u(W)$ is.\end{proof}

\begin{lemma}\label{lemma:keylemmaoddind}
Let $0 < \ell,k \leq q/2$ be odd integers. Then we have an orthogonal decomposition $V(2\ell) \otimes V(2k) = W \perp W'$, where $W$ and $W'$ are $K[u]$-submodules of $V(2\ell) \otimes V(2k)$ such that $$W \cong \Ind_{\langle u^2 \rangle}^{\langle u \rangle}(V_\ell \otimes V_k) \cong W'$$ as $K[u]$-modules.
\end{lemma}

\begin{proof}
It will suffice to prove the lemma for $(\Ind_{\langle u^2 \rangle}^{\langle u \rangle}(V_\ell), a) \otimes (\Ind_{\langle u^2 \rangle}^{\langle u \rangle}(V_{k}), a')$ where $a$ and $a'$ are as in Lemma \ref{lemma:inducedsingulardecomp}. Now $\Ind_{\langle u^2 \rangle}^{\langle u \rangle}(V_\ell) = (1 \otimes V_\ell) \oplus (u \otimes V_\ell)$ is a totally singular decomposition, so there exist bases $e_1$, $\ldots$, $e_\ell$ and $f_1$, $\ldots$, $f_\ell$ of $V_\ell$ such that $a(1 \otimes e_i, u \otimes f_j) = \delta_{i,j}$ for all $1 \leq i,j \leq \ell$. Similarly, one finds bases $e_1'$, $\ldots$, $e_{k}'$ and $f_1'$, $\ldots$, $f_{k}'$ of $V_{k}$ such that $a(1 \otimes e_i', u \otimes f_j') = \delta_{i,j}$ for all $1 \leq i,j \leq k$.

Consider the map $\theta: \Ind_{\langle u^2 \rangle}^{\langle u \rangle}(V_\ell \otimes V_{k}) \rightarrow \Ind_{\langle u^2 \rangle}^{\langle u \rangle}(V_\ell) \otimes \Ind_{\langle u^2 \rangle}^{\langle u \rangle}(V_{k})$ defined by $\theta(g \otimes (x \otimes y)) = (g \otimes x) \otimes (g \otimes y)$ for all $g \in \langle u \rangle$, $x \in V_\ell$, and $y \in V_{k}$. It is straightforward to see that $\theta$ is an injective map of $K[u]$-modules. 

We claim that $W = \operatorname{Im} \theta$ is a non-degenerate subspace of $\Ind_{\langle u^2 \rangle}^{\langle u \rangle}(V_\ell) \otimes \Ind_{\langle u^2 \rangle}^{\langle u \rangle}(V_{k})$ with respect to the tensor product form $b = a \otimes a'$. For this, first note that $W$ has as a basis the elements $v_{i,i_0} = (1 \otimes e_i) \otimes (1 \otimes e_{i_0}')$ and $w_{j,j_0} = (u \otimes f_j) \otimes (u \otimes f_{j_0}')$ for $1 \leq i,j \leq \ell$ and $1 \leq i_0,j_0 \leq k$. We have $b(v_{i,i_0}, w_{j,j_0}) = \delta_{i,j} \delta_{i_0,j_0}$ for all $1 \leq i,j \leq \ell$ and $1 \leq i_0,j_0 \leq k$, so it follows that $W$ is non-degenerate. Therefore $$\Ind_{\langle u^2 \rangle}^{\langle u \rangle}(V_\ell) \otimes \Ind_{\langle u^2 \rangle}^{\langle u \rangle}(V_{k}) = W \perp W',$$ where $W'$ is the orthogonal complement to $W$ with respect to $b$.


Now $W \cong \Ind_{\langle u^2 \rangle}^{\langle u \rangle}(V_\ell \otimes V_{k})$, and by \cite[Lemma 5 (5), p. 57]{Alperin} \begin{align*}\Ind_{\langle u^2 \rangle}^{\langle u \rangle}(V_\ell) \otimes \Ind_{\langle u^2 \rangle}^{\langle u \rangle}(V_{k}) &\cong \Ind_{\langle u^2 \rangle}^{\langle u \rangle}(V_\ell \otimes \Res_{\langle u^2 \rangle}^{\langle u \rangle}\Ind_{\langle u^2 \rangle}^{\langle u \rangle}(V_{k}))\\&\cong \Ind_{\langle u^2 \rangle}^{\langle u \rangle}(V_\ell \otimes V_{k}^2)\\&\cong W \oplus W\end{align*} as $K[u]$-modules. From the Krull-Schmidt theorem for $K[u]$-modules, we conclude that $W' \cong W$.\end{proof}

With the lemmas above, we can now prove the main result of this section.

\begin{lause}\label{thm:formtensorproducts}
Let $0 < \ell,k \leq q/2$. Then: 
\begin{enumerate}[\normalfont (i)]
\item $V_{2\ell} \otimes V_{2k} \cong V_{2d_1}^{2k_1} \oplus \cdots \oplus V_{2d_t}^{2k_t}$ for some integers $0 < d_1 < \cdots < d_t$ and $k_i > 0$.
\item If $\nu_2(\ell) \neq \nu_2(k)$, then as a bilinear $K[u]$-module $V(2\ell) \otimes V(2k)$ is isomorphic to $$\bigperp_{1 \leq i \leq t} W(2d_i)^{k_i}.$$
\item If $\nu_2(\ell) = \nu_2(k) = \alpha$, then there a unique $j$ such that $\nu_2(d_j) = \alpha$, and $d_j/2^{\alpha}$ is the unique odd Jordan block size in $V_{\ell/2^{\alpha}} \otimes V_{k/2^{\alpha}}$. Furthermore $k_j = 2^{\alpha}$, and as a bilinear $K[u]$-module $V(2\ell) \otimes V(2k)$ is isomorphic to $$V(2d_j)^{2^{\alpha+1}} \perp \bigperp_{\substack{1 \leq i \leq t \\ i \neq j}} W(2d_i)^{k_i}.$$
\end{enumerate}
\end{lause}

\begin{proof}We have $V_{2\ell} \otimes V_{2k} \cong \Ind_{\langle u^2 \rangle}^{\langle u \rangle}(V_\ell \otimes V_{2k}) \cong \Ind_{\langle u^2 \rangle}^{\langle u \rangle}(V_\ell \otimes \Res_{\langle u^2 \rangle}^{\langle u \rangle}(V_{2k}))$ by \cite[Lemma 5 (5), p. 57]{Alperin}, so \begin{equation}\label{eq:indreseq}V_{2\ell} \otimes V_{2k} \cong \Ind_{\langle u^2 \rangle}^{\langle u \rangle}(V_\ell \otimes V_k)^2\end{equation} as $K[u]$-modules. Thus $V_{2\ell} \otimes V_{2k} \cong V_{2d_1}^{2k_1} \oplus \cdots \oplus V_{2d_t}^{2k_t}$, where $V_\ell \otimes V_k \cong V_{d_1}^{k_1} \oplus \cdots \oplus V_{d_t}^{k_t}$ for some integers $0 < d_1 < \cdots < d_t$ and $k_i > 0$, which proves (i). We note here that (i) follows also from \cite[Theorem 5]{GlasbyPraegerXia}.

For (ii), we assume without loss of generality that $\alpha = \nu_2(\ell) > \nu_2(k)$. Write $\ell = 2^{\alpha}\ell'$. We have $V(2\ell) \cong \Ind_{\langle u^{2^{\alpha}} \rangle}^{\langle u \rangle}(V(2\ell'))$ by Lemma \ref{lemma:indresforcyclicALPHA} (i), so it follows with Lemma \ref{lemma:bilinearindres} that \begin{align}V(2\ell) \otimes V(2k) &\cong \Ind_{\langle u^{2^{\alpha}} \rangle}^{\langle u \rangle}(V(2\ell')) \otimes V(2k) \nonumber \\ \label{eq:indreseq2} &\cong \Ind_{\langle u^{2^{\alpha}} \rangle}^{\langle u \rangle}(V(2\ell') \otimes \Res_{\langle u^{2^{\alpha}} \rangle}^{\langle u \rangle}(V(2k)))\end{align} as bilinear $K[u]$-modules.

Write $k = 2^{\alpha-1}k' + r$ for $0 \leq r < 2^{\alpha-1}$. Since $2^{\alpha} \nmid k$, by Lemma \ref{lemma:indresforcyclicALPHA} (iv) $$\Res_{\langle u^{2^{\alpha}} \rangle}^{\langle u \rangle}(V(2k))) \cong W(2k'+1)^{r} \perp W(2k')^{2^{\alpha-1}-r}$$ as bilinear $K[u]$-modules. Thus~\eqref{eq:indreseq2} is a paired module by Lemma \ref{lemma:tensorpaired} and Lemma \ref{lemma:inducedpaired}, which combined with Lemma \ref{lemma:pairedjordan} gives (ii).

Next we consider (iii), so suppose that $\alpha = \nu_2(\ell) = \nu_2(k)$, and write $\ell = 2^{\alpha}\ell'$, $k = 2^{\alpha}k'$, where $\ell',k'$ are odd integers. Similarly to~\eqref{eq:indreseq}, we see that \begin{equation}\label{eq:indreseq0}V_{2\ell} \otimes V_{2k} \cong \Ind_{\langle u^{2^{\alpha+1}} \rangle}^{\langle u \rangle}(V_{\ell'} \otimes V_{k'})^{2^{\alpha+1}}\end{equation} as $K[u]$-modules. By Lemma \ref{lemma:uniqueoddchar2} the tensor product $V_{\ell'} \otimes V_{k'}$ has a unique Jordan block of odd size $d'$, occurring with multiplicity $1$. Hence we conclude from~\eqref{eq:indreseq0} that $\nu_2(d_j) = \alpha$ for a unique $j$, and $k_j = 2^{\alpha}$. Note that $d_j = 2^{\alpha}d'$ and $\nu_2(d_i) > \alpha$ for all $i \neq j$.

We will now proceed to show that $V(2d_i)$ occurs as an orthogonal direct summand of $V(2\ell) \otimes V(2k)$ if and only if $i = j$, which will complete the proof of (iii) and the theorem. First note that $$\Res_{\langle u^{2^{\alpha+1}} \rangle}^{\langle u \rangle}(V(2\ell) \otimes V(2k)) \cong W(\ell') \otimes W(k')$$ is paired module by Lemma \ref{lemma:indresforcyclicALPHA} (iv) and Lemma \ref{lemma:tensorpaired}. On the other hand, by Lemma \ref{lemma:indresforcyclicALPHA} (iv) we have $\Res_{\langle u^{2^{\alpha+1}} \rangle}^{\langle u \rangle} V(2d_i) \cong V(d_i/2^{\alpha})$ for $i \neq j$ (since $\nu_2(d_i) > \alpha$). Thus we conclude that if $i \neq j$, then $V(2d_i)$ cannot be an orthogonal direct summand of $V(2\ell) \otimes V(2k)$.

What remains is to show that $V(2d_j)$ occurs as an orthogonal direct summand of $V(2\ell) \otimes V(2k)$. For this, first note that $\Res_{\langle u^{2^{\alpha}} \rangle}^{\langle u \rangle}(V(2k))) \cong V(2k')^{2^{\alpha}}$ by Lemma \ref{lemma:indresforcyclicALPHA} (iv). Thus by Lemma \ref{lemma:bilinearindres} \begin{equation}\label{eqref:case3formula}V(2\ell) \otimes V(2k) \cong \Ind_{\langle u^{2^{\alpha}} \rangle}^{\langle u \rangle}(V(2\ell') \otimes V(2k'))^{2^{\alpha}}\end{equation} as bilinear $K[u]$-modules, as in~\eqref{eq:indreseq2}. 

By Lemma \ref{lemma:keylemmaoddind}, we have $V(2\ell') \otimes V(2k') = W \perp W'$, where $$W \cong \Ind_{\langle u^2 \rangle}^{\langle u \rangle}(V_{\ell'} \otimes V_{k'})$$ as $K[u]$-modules. Since $V_{d'}$ occurs with multiplicity $1$ in $V_{\ell'} \otimes V_{k'}$, we conclude that $V_{2d'}$ occurs with multiplicity $1$ in $W$. In this case $V(2d')$ must occur as an orthogonal direct summand of $W$ by Lemma \ref{lemma:basicepsilonoddeven} (ii) and Lemma \ref{lemma:epsilonequivalence}. Now it follows from~\eqref{eqref:case3formula} and Lemma \ref{lemma:indresforcyclicALPHA} that $V(2^{\alpha+1}d') = V(2d_j)$ occurs as an orthogonal direct summand of $V(2\ell) \otimes V(2k)$.\end{proof}

We finish this section by giving some examples that illustrate Theorem \ref{thm:formtensorproducts}.

\begin{esim}\label{example:v2tensor}
For any $0 < k \leq q/2$, it follows from~\eqref{eq:indreseq} that $V_2 \otimes V_{2k} \cong V_{2k}^2$. Thus we conclude from Theorem \ref{thm:formtensorproducts} that $$V(2) \otimes V(2k) \cong \begin{cases}W(2k),&\text{ if } k \equiv 0 \mod{2}. \\ V(2k)^2,&\text{ if } k \equiv 1 \mod{2}.\end{cases}$$
\end{esim}

\begin{esim}\label{example:v4tensor}
For $1 < k \leq q/2$, it is well known that $V_2 \otimes V_k \cong V_k^2$ if $k \equiv 0 \mod{2}$ and $V_2 \otimes V_k \cong V_{k-1} \oplus V_{k+1}$ if $k \equiv 1 \mod{2}$. It follows then from~\eqref{eq:indreseq} that $V_4 \otimes V_{2k} \cong V_{2k}^4$ if $k \equiv 0 \mod{2}$ and $V_4 \otimes V_{2k} \cong V_{2k-2}^2 \oplus V_{2k+2}^2$ if $k \equiv 1 \mod{2}$. Hence $$V(4) \otimes V(2k) \cong 
\begin{cases} W(2k)^2,&\text{ if } k \equiv 0 \mod{4}. \\
W(2k-2) \perp W(2k+2),&\text{ if } k \equiv 1 \mod{4}. \\
V(2k)^4,&\text{ if } k \equiv 2 \mod{4}. \\
W(2k-2) \perp W(2k+2),&\text{ if } k \equiv 3 \mod{4}. \\\end{cases}$$ by Theorem \ref{thm:formtensorproducts}.\end{esim}

\begin{esim}\label{example:v6tensor}
For $2 < k \leq q/2$, similarly to Examples \ref{example:v2tensor} and Example \ref{example:v4tensor}, from~\eqref{eq:indreseq} and the decomposition of $V_3 \otimes V_{k}$ (Example \ref{example:tensor3n}) one finds using Theorem \ref{thm:formtensorproducts} that $$V(6) \otimes V(2k) \cong 
\begin{cases} W(2k)^3,&\text{ if } k \equiv 0 \mod{4}. \\
W(2k-2)^2 \perp V(2k+4)^2,&\text{ if } k \equiv 1 \mod{4}. \\
W(2k-4) \perp W(2k) \perp W(2k+4),&\text{ if } k \equiv 2 \mod{4}. \\
V(2k-4)^2 \perp W(2k+2)^2,&\text{ if } k \equiv 3 \mod{4}. \\\end{cases}$$ 
\end{esim}

\section{An alternating bilinear form on $V \otimes V^*$}\label{section:altformtq}

Let $V$ be a finite-dimensional vector space over $K$ with $n = \dim V$ and set $G = \SL(V)$. The purpose of this section is to describe a non-zero alternating $G$-invariant bilinear form on $V \otimes V^*$ explicitly, and to give some of its basic properties.

Fix a basis $e_1$, $\ldots$, $e_n$ of $V$ and the corresponding dual basis $e_1^*$, $\ldots$, $e_n^*$ of $V^*$, so $e_i^*(e_j) = \delta_{i,j}$ for all $1 \leq i,j \leq n$. 

There is a natural bilinear form $\hat{b}_V$ on $V \otimes V^*$ defined by $$\hat{b}_V(v \otimes f, v' \otimes f') = f(v')f'(v)$$ for all $v, v' \in V$ and $f, f' \in V^*$. A straightforward calculation shows that $\hat{b}_V$ is a non-degenerate $G$-invariant symmetric bilinear form. Note that $\hat{b}_V(e_i \otimes e_i^*, e_i \otimes e_i^*) = 1$ for all $1 \leq i \leq n$, so $\hat{b}_V$ is not alternating. However, as in \cite[Lemma 4.1]{KorhonenUP}, one can use $\hat{b}_V$ to define an alternating bilinear form on $V \otimes V^*$.

Let $$\gamma_V = \sum_{1 \leq i \leq n} e_i \otimes e_i^*.$$ It is well known that the choice of $\gamma_V$ does not depend on the choice of the basis $(e_i)$, and furthermore $\gamma_V$ spans the fixed point space of $G$ on $V \otimes V^*$, see for example \cite[Lemma 3.7]{KorhonenJordanGood}.

We have a morphism of $G$-modules $\psi_V: V \otimes V^* \rightarrow K$ defined by $\psi_V(v \otimes f) = f(v)$ for all $v \in V$ and $f \in V^*$. By calculating $\psi_V(x)$ on basis elements $e_i \otimes e_j^*$ of $V \otimes V^*$, one finds that $\psi_V(x) = \hat{b}_V(x, \gamma_V)$ for all $x \in V \otimes V^*$. 

We define $\tau: V \otimes V^* \rightarrow V \otimes V^*$ by $\tau(x) = x + \psi_V(x)\gamma_V$ for all $x \in V \otimes V^*$. Then $\tau$ is a morphism of $G$-modules, so $b_V$ defined by \begin{equation}\label{eq:betaVdef}b_V(x,y) = \hat{b}_V(\tau(x), y) = \hat{b}_V(x,y) + \psi_V(x)\psi_V(y)\end{equation} for all $x, y \in V \otimes V^*$ is a $G$-invariant bilinear form on $V \otimes V^*$. For calculations, it is useful to note that $$b_V(e_i \otimes e_i^*, e_j \otimes e_j^*) = \begin{cases} 1, &\text{ if } i \neq j. \\ 0, &\text{ if } i = j.\end{cases}$$ for all $1 \leq i,j \leq n$. Furthermore, for $i \neq j$ $$b_V(e_i \otimes e_j^*, e_{i_0} \otimes e_{j_0}^*) = \begin{cases} 1, &\text{ if } i = j_0 \text { and } j = i_0. \\ 0, &\text{ otherwise.}\end{cases}$$

We will now make some basic observations about $b_V$.

\begin{lemma}\label{lemma:basiclemmaonBV}The following statements hold:

\begin{enumerate}[\normalfont (i)]
\item The bilinear form $b_V$ is alternating.
\item If $n$ is even, then $\operatorname{rad} b_V = 0$ and $\operatorname{Ker} \psi_V = \langle \gamma_V \rangle^\perp$.
\item If $n$ is odd, then $\operatorname{rad} b_V = \langle \gamma_V \rangle$ and $V \otimes V^* = \operatorname{Ker} \psi_V \oplus \langle \gamma_V \rangle$.
\item We have $\langle \gamma_V \rangle^\perp / \langle \gamma_V \rangle \cong L_G(\varpi_1 + \varpi_{n-1})$ as $G$-modules.
\end{enumerate}

\end{lemma}

\begin{proof}
For (i), first note that $b_V$ is symmetric since $\hat{b}_V$ is. It is easy to verify that $b_V(x,x) = 0$ for all basis elements $x = e_i \otimes e_j^*$, so $b_V$ is alternating.

The bilinear form $\hat{b}_V$ is non-degenerate, so $\operatorname{rad} b_V = \operatorname{Ker} \tau$. Clearly $\operatorname{Ker} \tau \subseteq \langle \gamma_V \rangle$, and $\tau(\gamma_V) = \gamma_V + \psi_V(\gamma_V)\gamma_V = (n+1) \gamma_V$. Thus $\gamma_V \in \operatorname{Ker} \tau$ if and only if $n$ is odd, from which the claims about $\operatorname{rad} b_V$ in (ii) and (iii) follow. For other claim in (ii), note that $\operatorname{Ker} \psi_V \subseteq \langle \gamma_V \rangle^\perp$. If $n$ is even, then $\langle \gamma_V \rangle^\perp \neq V \otimes V^*$ and so equality holds since both subspaces have codimension one. The other claim in (iii) follows since $\gamma_V \not\in \operatorname{Ker} \psi_V$ when $n$ is odd.

Since $\gamma_V$ spans the unique $1$-dimensional $G$-submodule of $V \otimes V^*$, claim (iv) follows easily from (ii), (iii), and Lemma \ref{lemma:typeA_VxV}.\end{proof}

\begin{lemma}\label{lemma:formonVxVunique}
Every $G$-invariant alternating bilinear form on $V \otimes V^*$ is a scalar multiple of $b_V$.
\end{lemma}

\begin{proof}
It will suffice to show that $V \otimes V^*$ has a unique $G$-invariant alternating bilinear form up to a scalar multiple. If $n$ is even, this follows from \cite[Lemma 4.2]{KorhonenUP}. If $n$ is odd, it follows from Lemma \ref{lemma:typeA_VxV} that $V \otimes V^* = W \oplus \langle \gamma_V \rangle$, where $W \cong L_G(\varpi_1 + \varpi_{n-1})$. 

Let $b$ be a $G$-invariant alternating bilinear form on $V \otimes V^*$. Then the map $f: V \otimes V^* \rightarrow K$ defined by $f(v) = b(v, \gamma_V)$ is a morphism of $G$-modules, where $G$ acts trivially on $K$. The map $f$ must vanish on $W$ since $W$ is a non-trivial irreducible $K[G]$-module, and furthermore $f$ vanishes on $\gamma_V$ since $b$ is alternating. Thus $f$ is zero, which means that $\gamma_V \in \operatorname{rad} b$. Now the claim follows, since $W$ is irreducible and thus has a unique $G$-invariant bilinear form up to a scalar multiple.\end{proof}

\begin{remark}An alternative point of view that could have been used in this section is the following. Recall that there is a natural isomorphism $V \otimes V^* \rightarrow \operatorname{End}(V)$ of $G$-modules, where for $v \in V$ and $f \in V^*$ the image of $v \otimes f$ is the linear map $V \rightarrow V$ defined by $w \mapsto f(w)v$. Here $G$ acts on $\operatorname{End}(V)$ by conjugation. 

Under this isomorphism, the element $\gamma_V$ corresponds to the identity map $\operatorname{Id}_V$ on $V$, and the map $\psi_V$ corresponds to the trace map $\operatorname{End}(V) \rightarrow K$. The bilinear form $\hat{b}_V$ corresponds to the bilinear form on $\operatorname{End}(V)$ defined by $(A,B) \mapsto \operatorname{Tr}(AB)$. The map $\tau$ corresponds to $A \mapsto A + \operatorname{Tr}(A)\operatorname{Id}_V$, so the alternating bilinear form $b_V$ corresponds to the bilinear form defined by $(A,B) \mapsto \operatorname{Tr}(AB) - \operatorname{Tr}(A)\operatorname{Tr}(B)$.\end{remark}

\section{An alternating bilinear form on $\wedge^2(V)$}\label{section:altformaq}

Let $b$ be a non-degenerate alternating bilinear form on a vector space $V$ over $K$ with $\dim V = 2n$. Set $G = \Sp(V,b)$. This section is analogous to the previous one, and we will be concerned with a non-zero alternating $G$-invariant bilinear form on $\wedge^2(V)$ and its basic properties.

The $G$-invariant bilinear form $b$ on $V$ induces a $G$-invariant bilinear form $\hat{a}_V$ on $\wedge^2(V)$ via $\hat{a}_V(v_1 \wedge v_2, w_1 \wedge w_2) = \det(b(v_i, w_j))_{1 \leq i,j \leq 2}$ for all $v_i,w_j \in V$. That is, $$\hat{a}_V(v_1 \wedge v_2, w_1 \wedge w_2) = b(v_1,w_1)b(v_2,w_2) + b(v_1,w_2)b(v_2,w_1)$$ for all $v_1, v_2, w_1,w_2 \in V$. The bilinear form $\hat{a}_V$ is a non-degenerate $G$-invariant symmetric bilinear form. Now $\hat{a}_V$ is not alternating, but as in Section \ref{section:altformtq}, with a small modification we can construct a $G$-invariant alternating bilinear form.

Fix a basis $e_1$, $\ldots$, $e_{2n}$ of $V$ such that $b(e_i, e_j) = 1$ if $i+j = 2n+1$ and $b(e_i,e_j) = 0$ otherwise. Define $$\beta_V = \sum_{1 \leq i \leq n} e_i \wedge e_{2n+1-i}.$$ It follows from \cite[3.4]{DeBruyn} that $\beta_V$ does not depend on the choice of the basis $(e_i)$, and thus it is fixed by the action of $\Sp(V,b)$ on $\wedge^2(V)$. Furthermore, it is clear from Lemma \ref{lemma:typeComega} that $\beta_V$ is the unique $\Sp(V,b)$-fixed point in $\wedge^2(V)$, up to scalar multiples.

We have a morphism of $G$-modules $\varphi_V: \wedge^2(V) \rightarrow K$ defined by $\varphi_V(v \wedge v') = b(v,v')$ for all $v,v' \in V$. Similarly to $\psi_V$ in Section \ref{section:altformtq}, we see that $\varphi_V(x) = \hat{a}_V(x, \beta_V)$ for all $x \in \wedge^2(V)$. 

Define $\sigma: \wedge^2(V) \rightarrow \wedge^2(V)$ by $\sigma(x) = x + \varphi_V(x)\beta_V$ for all $x \in \wedge^2(V)$. Then $\sigma$ is a morphism of $G$-modules, and so $a_V$ defined by $$a_V(x,y) = \hat{a}_V(\sigma(x),y) = \hat{a}_V(x,y) + \varphi_V(x)\varphi_V(y)$$ for all $x, y \in \wedge^2(V)$ is a $G$-invariant bilinear form on $\wedge^2(V)$.

\begin{lemma}\label{lemma:basiclemmaonAV}The following statements hold:

\begin{enumerate}[\normalfont (i)]
\item The bilinear form $a_V$ is alternating.
\item If $n$ is even, then $\operatorname{rad} a_V = 0$ and $\operatorname{Ker} \varphi_V = \langle \beta_V \rangle^\perp$.
\item If $n$ is odd, then $\operatorname{rad} a_V = \langle \beta_V \rangle$ and $\wedge^2(V) = \operatorname{Ker} \varphi_V \oplus \langle \beta_V \rangle$.
\item We have $\langle \beta_V \rangle^\perp / \langle \beta_V \rangle \cong L_G(\varpi_2)$ as $G$-modules.
\end{enumerate}

\end{lemma}

\begin{proof}Same as Lemma \ref{lemma:basiclemmaonBV}.\end{proof}

\begin{lemma}\label{lemma:formonwedgeVunique}
Every $G$-invariant alternating bilinear form on $\wedge^2(V)$ is a scalar multiple of $a_V$.
\end{lemma}

\begin{proof}
If $n$ is even, the claim follows from \cite[Lemma 4.2]{KorhonenUP}. If $n$ is odd, the lemma follows with the same proof as Lemma \ref{lemma:formonVxVunique}.
\end{proof}

\begin{lemma}\label{lemma:pairedaltsquare}
Let $H < G$, and let $(V, b) = (W \oplus W^*, b)$ be the paired module associated with some $K[H]$-module $W$. Then $$\wedge^2(W \oplus W^*) = \wedge^2(W) \oplus \wedge^2(W^*) \oplus \left( W \wedge W^* \right),$$ where $(\wedge^2(W) \oplus \wedge^2(W^*), a_V)$ is the paired module associated with $\wedge^2(W)$, and $(W \wedge W^*, a_V) \cong (W \otimes W^*, b_W)$ as bilinear $K[H]$-modules.
\end{lemma}

\begin{proof}
The restriction of $a_V$ to $\wedge^2(W) \oplus \wedge^2(W^*)$ is non-degenerate, and furthermore $\wedge^2(W) \oplus \wedge^2(W^*)$ is a totally singular decomposition with respect to $a_V$. Thus $(\wedge^2(W) \oplus \wedge^2(W^*), a_V)$ is the paired module associated with $\wedge^2(W)$ by Lemma \ref{lemma:basicpairedproperty}.

For $W \wedge W^*$, a straightforward verification shows that $w \wedge f \mapsto w \otimes f$ defines an isomorphism $(W \wedge W^*, a_V) \rightarrow (W \otimes W^*, b_W)$ of bilinear $K[H]$-modules.\end{proof}

\section{Hesselink normal forms on $V \otimes V^*$}\label{section:mainA}

In this section, we will prove Theorem \ref{thm:MAINTHMA}, one of the main results of this paper. At the end of this section, we will also give some examples which illustrate how Theorem \ref{thm:MAINTHMA} can be applied.

Let $V$ be a vector space over $K$ with $n = \dim V$. Set $G = \SL(V)$. Recall (Lemma \ref{lemma:formonVxVunique}) that we have an alternating $G$-invariant bilinear form $b_V$ on $V \otimes V^*$ which is unique up to scalar multiples. By Lemma \ref{lemma:basiclemmaonBV}, the bilinear form $b_V$ induces a non-degenerate $G$-invariant bilinear form on $\langle \gamma_V \rangle^\perp / \langle \gamma_V \rangle \cong L_G(\varpi_1 + \varpi_{n-1})$, giving us a representation $f: G \rightarrow \Sp(L_G(\varpi_1 + \varpi_{n-1}), b_V)$. Furthermore, the bilinear form $b_V$ is non-degenerate if and only if $n$ is even (Lemma \ref{lemma:basiclemmaonBV}), in which case we also get a representation $f': G \rightarrow \Sp(V \otimes V^*, b_V)$. 

For each unipotent element $u \in G$, Theorem \ref{thm:MAINTHMA} describes the Hesselink normal form of $f(u)$. Furthermore when $n$ is even, Theorem \ref{thm:MAINTHMA} also gives the Hesselink normal form of $f'(u)$. We state the Hesselink normal forms in terms of the Jordan normal form of $u$ on $V \otimes V^*$, which one can calculate using Theorem \ref{thm:tensordecompchar2}. 

We will first need two lemmas, and to setup their statements we fix a basis $e_1$, $\ldots$, $e_n$ of $V$ and the corresponding dual basis $e_1^*$, $\ldots$, $e_n^*$ of $V^*$. For convenience of notation, we set $e_i = 0$ and $e_i^* = 0$ for all $i \leq 0$ and $i > n$. Let $u \in G$ be a unipotent Jordan block with respect to the basis $(e_i)$, that is, $$ue_i = e_i + e_{i-1}$$ for all $1 \leq i \leq n$. As usual, we denote by $X$ the element $u-1$ of $K[u]$.

Let $\alpha > 0$ be such that $2^\alpha \leq n < 2^{\alpha+1}$. For all $1 \leq \beta \leq \alpha+1$ and $2^{\beta-1}+1 \leq i \leq 2^{\beta}$, we define \begin{equation}\label{eq:defofvi}v_i^{(\beta)} = \sum_{j \geq 0} e_{i+j2^{\beta}} \otimes e_{i-2^{\beta-1}+j2^{\beta}}^*.\end{equation}

The key lemma in this section is the following.

\begin{lemma}\label{lemma:vi_values}
Let $1 \leq \beta \leq \alpha+1$ and write $n = k2^{\beta}+r$, where $0 \leq r < 2^\beta$. Then for all $2^{\beta-1}+1 \leq i \leq 2^{\beta}$:

\begin{enumerate}[\normalfont (i)]
\item We have $X^{2^\beta} \cdot v_i^{(\beta)} = 0$ if and only if $0 \leq r < i-2^{\beta-1}$ or $r \geq i$.
\item If $r < i$, then $b_V(X^{2^\beta-1}v_i^{(\beta)},v_i^{(\beta)}) = k$.
\item If $r \geq i$, then $b_V(X^{2^\beta-1}v_i^{(\beta)},v_i^{(\beta)}) = k+1$.
\end{enumerate}
\end{lemma}

\begin{proof}For (i), note first that $X^{2^{\beta}} = u^{2^{\beta}}-1$, so $X^{2^{\beta}} \cdot (e_{i+j2^{\beta}} \otimes e_{i-2^{\beta-1}+j2^{\beta}}^*)$ equals \begin{equation}\label{eq:firstj}(X^{2^{\beta}}e_{i+j2^{\beta}} \otimes u^{2^{\beta}}e_{i-2^{\beta-1}+j2^{\beta}}^*) + (e_{i+j2^{\beta}} \otimes X^{2^{\beta}}e_{i-2^{\beta-1}+j2^{\beta}}^*)\end{equation} for all $j \geq 0$. Now by \cite[(5.1)]{KorhonenJordanGood} we have $X^{2^{\beta}}e_{i'} = e_{i'-2^{\beta}}$ and $X^{2^{\beta}}e_{i'}^* = \sum_{j' > 0} e_{i'+j2^{\beta}}^*$ for all $i'$, so~\eqref{eq:firstj} equals $q_{j-1} + q_j$, where we define $$q_j = e_{i+j2^{\beta}} \otimes \sum_{j' > j} e_{i-2^{\beta-1}+j'2^{\beta}}^*$$ for all $j \geq -1$. Note that $q_{-1} = 0$ since $e_{i-2^{\beta}} = 0$. Therefore \begin{equation}\label{eq:qmXkernel}X^{2^{\beta}} \cdot v_i^{(\beta)} = \sum_{j \geq 0} (q_{j-1} + q_j) = q_{-1} + q_m = q_m,\end{equation} where $m \geq 0$ is maximal such that $i+m2^{\beta} \leq n$. 

It is clear that $q_m = 0$ if and only if $i-2^{\beta-1}+(m+1)2^{\beta} > n$. If $r < i$, then $m = k-1$ and so $q_m = 0$ if and only if $i - 2^{\beta-1} + k2^{\beta} > n$, which is equivalent to $r < i-2^{\beta-1}$. Similarly if $r \geq i$, then $m = k$ and so $q_m = 0$ if and only if $i - 2^{\beta-1} + (k+1)2^{\beta} > n$, which is equivalent to $r < i+2^{\beta-1}$, and this always holds since $i > 2^{\beta-1}$. We have shown that $q_m = 0$ if and only if $0 \leq r < i-2^{\beta-1}$ or $r \geq i$, which together with~\eqref{eq:qmXkernel} completes the proof of (i).

For (ii) and (iii), we proceed to calculate $b_V(X^{2^\beta-1}v_i^{(\beta)},v_i^{(\beta)})$. By \cite[Lemma 5.2]{KorhonenJordanGood}, for all $j \geq 0$ we have $X^{2^{\beta}-1} \cdot (e_{i+j2^{\beta}} \otimes e_{i-2^{\beta-1}+j2^{\beta}}^*)$ equal to \begin{align}&\sum_{\substack{0 \leq t \leq 2^{\beta}-1 \\ 0 \leq s \leq t}} \binom{2^{\beta}-1}{t} \binom{t}{s} X^t e_{i+j2^{\beta}} \otimes X^{2^{\beta}-1-s} e_{i-2^{\beta-1}+j2^{\beta}}^* \nonumber \\ =\ & \sum_{\substack{0 \leq t \leq 2^{\beta}-1 \\ 0 \leq s \leq t}} \binom{t}{s} e_{i-t+j2^{\beta}} \otimes X^{2^{\beta}-1-s} e_{i-2^{\beta-1}+j2^{\beta}}^* \label{eq:tensoract1}\end{align} where~\eqref{eq:tensoract1} holds since $\binom{2^{\beta}-1}{t} \equiv 1 \mod{2}$ for all $0 \leq t \leq 2^{\beta}-1$ by Lucas' theorem. 

The summands in~\eqref{eq:tensoract1} that have non-zero product with $v_i^{(\beta)}$ with respect to $b_V$ occur only for $0 \leq t \leq 2^{\beta}-1$ such that $i-t+j2^{\beta} = i-2^{\beta-1}+j'2^{\beta}$, which is only possible for $t = 2^{\beta-1}$. Furthermore, by Lucas' theorem for $0 \leq s \leq 2^{\beta-1}$ we have $\binom{2^{\beta-1}}{s} \equiv 0 \mod{2}$ except for $s = 0$ and $s = 2^{\beta-1}$. Thus~\eqref{eq:tensoract1} equals \begin{equation}\label{eq:tensoractmoduloperp1}e_{i-2^{\beta-1}+j2^{\beta}} \otimes (X^{2^{\beta}-1}e_{i-2^{\beta-1}+j2^{\beta}}^* + X^{2^{\beta-1}-1}e_{i-2^{\beta-1}+j2^{\beta}}^*) \mod{\langle v_i^{(\beta)} \rangle^\perp}.\end{equation} We show next that~\eqref{eq:tensoractmoduloperp1} equals \begin{equation}\label{eq:tensoractmoduloperp2}e_{i-2^{\beta-1}+j2^{\beta}} \otimes e_{i+j2^{\beta}}^* \mod{\langle v_i^{(\beta)} \rangle^\perp}.\end{equation}

We divide the proof of~\eqref{eq:tensoractmoduloperp2} into two cases.\newline

\noindent \emph{Case 1:} Suppose that $\beta = 1$.\newline

\noindent In this case~\eqref{eq:tensoractmoduloperp1} equals \begin{align}&e_{i-1+2j} \otimes (Xe_{i-1+2j}^* + e_{i-1+2j}^*) \nonumber \\ =\ &e_{i-1+2j} \otimes \sum_{j' \geq i-1+2j} e_{j'}^* \label{eq:tensoractmoduloperp3} \\ \equiv\ & e_{i-1+2j} \otimes e_{i+2j}^* \mod{\langle v_i^{(\beta)} \rangle^\perp} \nonumber \end{align} where~\eqref{eq:tensoractmoduloperp3} is given by \cite[Lemma 5.1 (ii)]{KorhonenJordanGood}.\newline

\noindent \emph{Case 2:} Suppose that $\beta > 1$.\newline

\noindent In this case, by \cite[Lemma 5.1 (ii)]{KorhonenJordanGood} we have $e_{i-2^{\beta-1}+j2^{\beta}} \otimes X^{2^{\beta}-1} e_{i-2^{\beta-1}+j2^{\beta}}^*$ equal to a sum of some basis elements $e_{i-2^{\beta-1}+j2^{\beta}} \otimes e_{j'}^*$ such that $j' \geq i+2^{\beta-1}-1+j2^{\beta}$. Here $j' > i+j2^{\beta}$ since $\beta > 1$, so we conclude that $e_{i-2^{\beta-1}+j2^{\beta}} \otimes X^{2^{\beta}-1} e_{i-2^{\beta-1}+j2^{\beta}}^*$ has zero product with $v_i^{(\beta)}$ with respect to $b_V$. Thus~\eqref{eq:tensoractmoduloperp1} equals \begin{align}&e_{i-2^{\beta-1}+j2^{\beta}} \otimes X^{2^{\beta-1}-1}e_{i-2^{\beta-1}+j2^{\beta}}^* \mod{\langle v_i^{(\beta)} \rangle^\perp} \nonumber \\ \label{eq:tensoractmoduloperp4} =\ &e_{i-2^{\beta-1}+j2^{\beta}} \otimes \sum_{j' \geq i-1+j2^{\beta}} \binom{j'-i+2^{\beta-1}-j2^{\beta}-1}{2^{\beta-1}-2} e_{j'}^* \mod{\langle v_i^{(\beta)} \rangle^\perp} \\ =\ &e_{i-2^{\beta-1}+j2^{\beta}} \otimes \binom{2^{\beta-1}-1}{2^{\beta-1}-2} e_{i+j2^{\beta}}^* \mod{\langle v_i^{(\beta)} \rangle^\perp} \nonumber \\ =\ &e_{i-2^{\beta-1}+j2^{\beta}} \otimes e_{i+j2^{\beta}}^* \mod{\langle v_i^{(\beta)} \rangle^\perp} \nonumber\end{align} where~\eqref{eq:tensoractmoduloperp4} is given by \cite[Lemma 5.1 (ii)]{KorhonenJordanGood}. This completes the proof of~\eqref{eq:tensoractmoduloperp2}.\newline

\noindent We have shown that $X^{2^{\beta}-1} \cdot (e_{i+j2^{\beta}} \otimes e_{i-2^{\beta-1}+j2^{\beta}}^*)$ equals~\eqref{eq:tensoractmoduloperp2} modulo $\langle v_i^{(\beta)} \rangle^\perp$. From this we conclude that \begin{align*}b_V(X^{2^{\beta}-1}v_i^{(\beta)}, v_i^{(\beta)}) &=\ \sum_{j \geq 0} b_V(e_{i-2^{\beta-1}+j2^{\beta}} \otimes e_{i+j2^{\beta}}^*, v_i^{(\beta)})\\
 &=\ \sum_{\substack{j \geq 0 \\ i+j2^{\beta} \leq n}} 1\end{align*} which equals $k$ if $r < i$ and $k+1$ if $r \geq i$, proving (ii) and (iii).\end{proof}

\begin{lemma}\label{lemma:nonsingularsinVV}
Let $1 \leq \beta \leq \alpha+1$. If $V_n \otimes V_n$ has a Jordan block of size $2^{\beta}$, then $X^{2^\beta} \cdot v_i^{(\beta)} = 0$ and $b_V(X^{2^\beta-1}v_i^{(\beta)},v_i^{(\beta)}) \neq 0$ for some $2^{\beta-1}+1 \leq i \leq 2^{\beta}$.
\end{lemma}

\begin{proof}
Suppose that $V_n \otimes V_n$ has a Jordan block of size $2^{\beta}$. By Theorem \ref{thm:GPXtensorsquare}, this means that $2^{\beta}$ occurs in the consecutive-ones binary-expansion of $n$. Equivalently, either (a) $2^{\beta}$ occurs in the binary expansion of $n$ and $2^{\beta-1}$ does not; or (b) $2^{\beta-1}$ occurs in the binary expansion of $n$ and $2^{\beta}$ does not.

If (a) holds, then $n = k2^{\beta} + r$, where $0 \leq r < 2^{\beta-1}$ and $k$ is odd. By Lemma \ref{lemma:vi_values} (i) and (ii), for any $r+2^{\beta-1}+1 \leq i \leq 2^{\beta}$ we have $X^{2^\beta} \cdot v_i^{(\beta)} = 0$ and $b_V(X^{2^\beta-1}v_i^{(\beta)},v_i^{(\beta)}) = k \neq 0$. For example, one can choose $i = 2^{\beta}$.

If (b) holds, then $n = k2^{\beta} + r$ where $2^{\beta-1} \leq r < 2^{\beta}$ and $k$ is even. By Lemma \ref{lemma:vi_values} (i) and (ii), for any $2^{\beta-1}+1 \leq i \leq r$ we have $X^{2^\beta} \cdot v_i^{(\beta)} = 0$ and $b_V(X^{2^\beta-1}v_i^{(\beta)},v_i^{(\beta)}) = k+1 \neq 0$. For example, we could choose $i = r$.\end{proof}

With Lemma \ref{lemma:nonsingularsinVV}, we will be able to prove Theorem \ref{thm:MAINTHMA}, our first main result. We refer the reader to the introduction for the statement of the theorem.

\begin{proof}[Proof of Theorem \ref{thm:MAINTHMA}] We first recall the setup of the theorem. Let $G = \SL(V)$, where $\dim V = n$ for some $n \geq 2$. Let $u \in G$ be unipotent and $V \cong V_{d_1} \oplus \cdots \oplus V_{d_t}$ as $K[u]$-modules, where $t \geq 1$ and $d_r \geq 1$ for all $1 \leq r \leq t$. Set $\alpha = \nu_2(\gcd(d_1, \ldots, d_t))$. Suppose that $V \otimes V^* \cong \oplus_{d \geq 1} V_d^{\lambda(d)}$ and $L_G(\varpi_1 + \varpi_{n-1}) \cong \oplus_{d \geq 1} V_d^{\lambda'(d)}$ as $K[u]$-modules, where $\lambda(d), \lambda'(d) \geq 0$ for all $d \geq 1$. We identify $L_G(\varpi_1 + \varpi_{n-1})$ as the subquotient $(\langle \beta_V \rangle^\perp / \langle \beta_V \rangle, b_V)$ of $(V \otimes V^*, b_V)$ --- see Lemma \ref{lemma:basiclemmaonBV} (iv). Set $\varepsilon := \varepsilon_{V \otimes V^*, b_V}$ and $\varepsilon' := \varepsilon_{L_G(\varpi_1 + \varpi_{n-1}), b_V}$.

For statements (i) -- (iii) of the theorem, the description of $\lambda'$ is just \cite[Theorem 6.1]{KorhonenJordanGood} in characteristic $p = 2$.

For the proof of claims (iv) -- (vi) concerning $\varepsilon$ and $\varepsilon'$, we first setup some more notation. Let $V = W_1 \oplus \cdots \oplus W_t$, where $W_r$ are $u$-invariant subspaces and $W_r \cong V_{d_r}$ for all $1 \leq r \leq t$. For each $r$, choose a basis $(e_j^{(r)})_{1 \leq j \leq d_{r}}$ of $W_r$ such that $ue_j^{(r)} = e_j^{(r)} + e_{j-1}^{(r)}$ for all $1 \leq j \leq d_r$, where we set $e_j^{(r)} = 0$ for all $j \leq 0$. For the basis $(e_j^{(r)})$ of $V$, we let $(e_j^{(r)^*})$ be the corresponding dual basis of $V^*$. 



We consider (iv). Suppose first that $\varepsilon(d) = 1$, so now $d$ is even by Lemma \ref{lemma:epsilonequivalence}. For each $1 \leq r \leq t$, we identify $W_r^*$ with the subspace spanned by the $(e_j^{(r)^*})$. Then $$V \otimes V^* = \bigoplus_{1 \leq r \leq t} \left( W_r \otimes W_r^* \right) \oplus \bigoplus_{1 \leq r < s \leq t} Z_{rs}$$ where $Z_{rs} = \left( W_r \otimes W_s^* \right) \oplus \left( W_s \otimes W_r^* \right)$ for all $1 \leq r < s \leq t$. The restriction of $b_V$ to $Z_{rs}$ is non-degenerate, and furthermore $W_r \otimes W_s^*$ and $W_s \otimes W_r^*$ are totally singular subspaces. Thus it follows from Lemma \ref{lemma:basicpairedproperty} that $(Z_{rs}, b_V)$ is the paired module associated with $W_r \otimes W_s^*$. Consequently by Lemma \ref{lemma:pairedjordan} and Lemma \ref{lemma:basicXdlinear} (iv), there exists $1 \leq r \leq t$ and $v \in \operatorname{Ker} X_{W_r \otimes W_r^*}^d$ such that $b_V(X^{d-1}v,v) \neq 0$. Then by Lemma \ref{lemma:epsilonequivalence} there is a Jordan block of size $d$ in $W_r \otimes W_r^*$, so by Theorem \ref{thm:GPXtensorsquare} we have $d = 2^{\beta}$ for some $2^{\beta} > 1$ occurring in the consecutive-ones binary expansion of $d_r$.


Conversely, suppose that $d = 2^{\beta} > 1$ occurs in the consecutive-ones binary expansion of $d_r$ for some $1 \leq r \leq t$. By Lemma \ref{lemma:nonsingularsinVV} and Theorem \ref{thm:GPXtensorsquare}, for a suitable choice of $2^{\beta-1}+1 \leq i \leq 2^{\beta}$ the element \begin{equation}\label{eq:vivectorfortsq}v_i^{(\beta)} = \sum_{j \geq 0} e_{i+j2^{\beta}}^{(r)} \otimes e_{i-2^{\beta-1}+j2^{\beta}}^{(r)^*}\end{equation} is such that $X^{2^{\beta}} v_i^{(\beta)} = 0$ and $b_V(X^{2^{\beta}-1}v_i^{(\beta)},v_i^{(\beta)}) \neq 0$. Thus $\varepsilon(d) = 1$, which completes the proof of (iv). Note that here $v_i^{(\beta)} \in \langle \gamma_V \rangle^\perp$, so this also shows that $\varepsilon'(d) = 1$ in this case.

For (v) and (vi), we proceed to calculate $\varepsilon'$. Suppose first that $2 \nmid n$. Then $\gamma_V \in \operatorname{rad} b_V$ by Lemma \ref{lemma:basiclemmaonBV} (iii), so $\langle \gamma_V \rangle^\perp / \langle \gamma_V \rangle = V / \langle \gamma_V \rangle$. In this case it is clear that $\varepsilon'(d) = \varepsilon(d)$ for all $d \geq 1$, as claimed.

We assume next that $2 \mid n$, so now $b_V$ is non-degenerate and $\langle \gamma_V \rangle^\perp = \operatorname{Ker} \psi_V$ by Lemma \ref{lemma:basiclemmaonBV} (ii). By \cite[(6.2)]{KorhonenJordanGood}, there exists a $\hat{\delta} \in V \otimes V^*$ such that $X^{2^{\alpha}}\hat{\delta} = 0$ and $\hat{\delta} \not\in \langle \gamma_V \rangle^\perp$. In the case where $\alpha = 0$, it follows from Lemma \ref{lemma:vperpmodvdescription} (ii) that $\varepsilon'(d) = \varepsilon(d)$ for all $d \geq 1$, proving the theorem in this case. Assume then for the rest of the proof that $\alpha > 0$. 

Note that $\varepsilon'(d) = \varepsilon(d)$ for all $d \neq 2^{\alpha}-2, 2^{\alpha}$ by Lemma \ref{lemma:vperpmodvdescription} (iii) -- (iv). We show that $\varepsilon(2^{\alpha}) = \varepsilon'(2^{\alpha})$. To this end, pick some $1 \leq r \leq t$ such that $\nu_2(d_r) = \alpha$. Then $2^{\alpha}$ occurs in the consecutive-ones binary expansion of $d_r$, so $\varepsilon(2^{\alpha}) = \varepsilon'(2^{\alpha}) = 1$ as shown at the end of the proof of (iv). In the case where $\alpha = 1$, it follows that $\varepsilon(d) = \varepsilon'(d)$ for all $d \geq 1$, as claimed. Thus we can assume $\alpha > 1$ for the rest of the proof.

So far we have shown that $\varepsilon(d) = \varepsilon'(d)$ for all $d \neq 2^{\alpha}-2$, as is claimed by (v) and (vi). For $\varepsilon(2^{\alpha}-2)$ and $\varepsilon'(2^{\alpha}-2)$, note that by \cite[Lemma 4.3]{KorhonenJordanGood} the smallest Jordan block size of $u$ in $V \otimes V^*$ is $2^{\alpha}$. Thus $u$ has no Jordan blocks of size $2^{\alpha}-2$ in $V \otimes V^*$, and so $\varepsilon(2^{\alpha}-2) = 0$ by Lemma \ref{lemma:epsilonequivalence}. If (iii)(b) holds, then $2^{\alpha}-2$ occurs in $L_G(\varpi_1 + \varpi_{n-1})$ with multiplicity one, so $\varepsilon'(2^{\alpha}-2) = 1$ by Lemma \ref{lemma:basicepsilonoddeven} (ii). If (iii)(b) does not hold, then $u$ has no Jordan blocks of size $2^{\alpha}-2$ on $L_G(\varpi_1 + \varpi_{n-1})$, and so $\varepsilon'(2^{\alpha}-2) = 0$.\end{proof}

In the following we show with small examples how Theorem \ref{thm:MAINTHMA} is applied. Let $G = \SL(V)$ and $n = \dim V$, and let $u \in G$ be a unipotent element. Let $\alpha = \nu_2(\gcd(d_1, \ldots, d_t))$ as in Theorem \ref{thm:MAINTHMA}, so $2^{\alpha}$ is the largest power of two dividing every Jordan block size of $u$ on $V$. Set $\varepsilon := \varepsilon_{V \otimes V^*, b_V}$ and $\varepsilon' := \varepsilon_{L_G(\varpi_1 + \varpi_{n-1}), b_V}$. In what follows we shall use Theorem \ref{thm:hesselinkepsilon} to describe the decomposition of $(V \otimes V^*, b_V)$ and $(L_G(\varpi_1 + \varpi_{n-1}), b_V)$ into bilinear $K[u]$-modules.

\begin{esim}When $n = 2$ and $V \cong V_2$ as $K[u]$-modules, we have $\alpha = 1$ and by Theorem \ref{thm:tensordecompchar2} we get $V \otimes V^* \cong V_2^2$ as $K[u]$-modules. We have a consecutive-ones binary expansion $n = 2^1$, so $\varepsilon(2) = 1$ and $\varepsilon(d) = 0$ for all $d \neq 2$ by Theorem \ref{thm:MAINTHMA} (iv). Thus $(V \otimes V^*, b_V) \cong V(2)^2$ as bilinear $K[u]$-modules. In this case Theorem \ref{thm:MAINTHMA} (iii)(c) and (vi) apply, giving $L_G(\varpi_1 + \varpi_{n-1}) \cong V_2$ and $\varepsilon'(d) = \varepsilon(d)$ for all $d \geq 1$. Hence $(L_G(\varpi_1 + \varpi_{n-1}), b_V) \cong V(2)$ as bilinear $K[u]$-modules.\end{esim}

\begin{esim}When $n = 4$ and $V \cong V_4$ as $K[u]$-modules, we have $\alpha = 2$ and by Theorem \ref{thm:tensordecompchar2} we get $V \otimes V^* \cong V_4^4$ as $K[u]$-modules. We have a consecutive-ones binary expansion $n = 2^2$, so $\varepsilon(4) = 1$ and $\varepsilon(d) = 0$ for all $d \neq 4$ by Theorem \ref{thm:MAINTHMA} (iv). Thus $(V \otimes V^*, b_V) \cong V(4)^4$ as bilinear $K[u]$-modules. In this case Theorem \ref{thm:MAINTHMA} (iii)(b) and (vi) apply, giving $L_G(\varpi_1 + \varpi_{n-1}) \cong V_2 \oplus V_4^3$, $\varepsilon'(2) = 1$, and $\varepsilon'(d) = \varepsilon(d)$ for all $d \neq 2$. Hence $(L_G(\varpi_1 + \varpi_{n-1}), b_V) \cong V(2) \perp V(4)^3$ as bilinear $K[u]$-modules.\end{esim}

\begin{esim}For $n = 6$ and $V \cong V_1 \oplus V_5$ as $K[u]$-modules, we have $\alpha = 0$ and by Theorem \ref{thm:tensordecompchar2} we get $V \otimes V^* \cong V_1^2 \oplus V_4^2 \oplus V_5^2 \oplus V_8^2$. We have a consecutive-ones binary expansion $5 = 2^3 - 2^2 + 2^0$, so $\varepsilon(8) = 1$, $\varepsilon(4) = 1$ and $\varepsilon(d) = 0$ for all $d \neq 4,8$ by Theorem \ref{thm:MAINTHMA} (iv). Hence $(V \otimes V^*, b_V) \cong W(1) \perp V(4)^2 \perp W(5) \perp V(8)^2$ as bilinear $K[u]$-modules. In this case Theorem \ref{thm:MAINTHMA} (ii) and (vi) apply, giving $(L_G(\varpi_1 + \varpi_{n-1}), b_V) \cong V(4)^2 \perp W(5) \perp V(8)^2$ as bilinear $K[u]$-modules.\end{esim}

\begin{esim}\label{esim:EXTHMA}In Table \ref{table:examplesofTHMA}, we illustrate Theorem \ref{thm:MAINTHMA} for all $2 \leq n \leq 7$. In the first column we use notation $(d_1^{n_1}, \ldots, d_t^{n_t})$ to denote that $V \cong V_{d_1}^{n_1} \oplus \cdots \oplus V_{d_t}^{n_t}$ as $K[u]$-modules. In the third and second columns, we use notation as in Theorem \ref{thm:hesselinkepsilon}. That is, for an alternating bilinear $K[u]$-module $(W,b)$, we use $({d_1}_{\varepsilon_1}^{n_1}, \ldots, {d_t}_{\varepsilon_t}^{n_t})$ to denote that $W \cong V_{d_1}^{n_1} \oplus \cdots \oplus V_{d_t}^{n_t}$ as $K[u]$-modules and $\varepsilon_{W, b}(d_i) = \varepsilon_i$ for $1 \leq i \leq t$.\end{esim}

\begin{table}[!htbp]
\centering
\caption{}\label{table:examplesofTHMA}
\footnotesize
\begin{tabular}{| c | l | l | l |}
\hline
$n$              & $V \downarrow K[u]$       & $(V \otimes V^*, b_V)$         & $(L_G(\varpi_1 + \varpi_{n-1}), b_V)$ \\ \hline
$n = 2$ & $(2)$ & $(2_1^{2})$ & $(2_1)$ \\
& & & \\
$n = 3$ & $(3)$ & $(1_0, 4_1^{2})$ & $(4_1^{2})$ \\
        & $(1, 2)$ & $(1_0, 2_1^{4})$ & $(2_1^{4})$ \\
& & & \\
$n = 4$ & $(4)$ & $(4_1^{4})$ & $(2_1, 4_1^{3})$ \\
        & $(1, 3)$ & $(1_0^{2}, 3_0^{2}, 4_1^{2})$ & $(3_0^{2}, 4_1^{2})$ \\
        & $(2^{2})$ & $(2_1^{8})$ & $(1_0^{2}, 2_1^{6})$ \\
        & $(1^{2}, 2)$ & $(1_0^{4}, 2_1^{6})$ & $(1_0^{2}, 2_1^{6})$ \\
& & & \\
$n = 5$ & $(5)$ & $(1_0, 4_1^{2}, 8_1^{2})$ & $(4_1^{2}, 8_1^{2})$ \\
        & $(1, 4)$ & $(1_0, 4_1^{6})$ & $(4_1^{6})$ \\
        & $(2, 3)$ & $(1_0, 2_1^{4}, 4_1^{4})$ & $(2_1^{4}, 4_1^{4})$ \\
        & $(1^{2}, 3)$ & $(1_0^{5}, 3_0^{4}, 4_1^{2})$ & $(1_0^{4}, 3_0^{4}, 4_1^{2})$ \\
        & $(1, 2^{2})$ & $(1_0, 2_1^{12})$ & $(2_1^{12})$ \\
        & $(1^{3}, 2)$ & $(1_0^{9}, 2_1^{8})$ & $(1_0^{8}, 2_1^{8})$ \\
& & & \\
$n = 6$ & $(6)$ & $(2_1^{2}, 8_1^{4})$ & $(2_1, 8_1^{4})$ \\
        & $(1, 5)$ & $(1_0^{2}, 4_1^{2}, 5_0^{2}, 8_1^{2})$ & $(4_1^{2}, 5_0^{2}, 8_1^{2})$ \\
        & $(2, 4)$ & $(2_1^{2}, 4_1^{8})$ & $(2_1, 4_1^{8})$ \\
        & $(1^{2}, 4)$ & $(1_0^{4}, 4_1^{8})$ & $(1_0^{2}, 4_1^{8})$ \\
        & $(3^{2})$ & $(1_0^{4}, 4_1^{8})$ & $(1_0^{2}, 4_1^{8})$ \\
        & $(1, 2, 3)$ & $(1_0^{2}, 2_1^{6}, 3_0^{2}, 4_1^{4})$ & $(2_1^{6}, 3_0^{2}, 4_1^{4})$ \\
        & $(1^{3}, 3)$ & $(1_0^{10}, 3_0^{6}, 4_1^{2})$ & $(1_0^{8}, 3_0^{6}, 4_1^{2})$ \\
        & $(2^{3})$ & $(2_1^{18})$ & $(2_1^{17})$ \\
        & $(1^{2}, 2^{2})$ & $(1_0^{4}, 2_1^{16})$ & $(1_0^{2}, 2_1^{16})$ \\
        & $(1^{4}, 2)$ & $(1_0^{16}, 2_1^{10})$ & $(1_0^{14}, 2_1^{10})$ \\
& & & \\				
$n = 7$ & $(7)$ & $(1_0, 8_1^{6})$ & $(8_1^{6})$ \\
        & $(1, 6)$ & $(1_0, 2_1^{2}, 6_0^{2}, 8_1^{4})$ & $(2_1^{2}, 6_0^{2}, 8_1^{4})$ \\
        & $(2, 5)$ & $(1_0, 2_1^{2}, 4_1^{4}, 6_0^{2}, 8_1^{2})$ & $(2_1^{2}, 4_1^{4}, 6_0^{2}, 8_1^{2})$ \\
        & $(1^{2}, 5)$ & $(1_0^{5}, 4_1^{2}, 5_0^{4}, 8_1^{2})$ & $(1_0^{4}, 4_1^{2}, 5_0^{4}, 8_1^{2})$ \\
        & $(3, 4)$ & $(1_0, 4_1^{12})$ & $(4_1^{12})$ \\
        & $(1, 2, 4)$ & $(1_0, 2_1^{4}, 4_1^{10})$ & $(2_1^{4}, 4_1^{10})$ \\
        & $(1^{3}, 4)$ & $(1_0^{9}, 4_1^{10})$ & $(1_0^{8}, 4_1^{10})$ \\
        & $(1, 3^{2})$ & $(1_0^{5}, 3_0^{4}, 4_1^{8})$ & $(1_0^{4}, 3_0^{4}, 4_1^{8})$ \\
        & $(2^{2}, 3)$ & $(1_0, 2_1^{12}, 4_1^{6})$ & $(2_1^{12}, 4_1^{6})$ \\
        & $(1^{2}, 2, 3)$ & $(1_0^{5}, 2_1^{8}, 3_0^{4}, 4_1^{4})$ & $(1_0^{4}, 2_1^{8}, 3_0^{4}, 4_1^{4})$ \\
        & $(1^{4}, 3)$ & $(1_0^{17}, 3_0^{8}, 4_1^{2})$ & $(1_0^{16}, 3_0^{8}, 4_1^{2})$ \\
        & $(1, 2^{3})$ & $(1_0, 2_1^{24})$ & $(2_1^{24})$ \\
        & $(1^{3}, 2^{2})$ & $(1_0^{9}, 2_1^{20})$ & $(1_0^{8}, 2_1^{20})$ \\
        & $(1^{5}, 2)$ & $(1_0^{25}, 2_1^{12})$ & $(1_0^{24}, 2_1^{12})$ \\
\hline		
\end{tabular}
\end{table}

\section{Hesselink normal forms on $\wedge^2(V)$}\label{section:mainB}

In this section, we will give a proof of Theorem \ref{thm:MAINTHMC}. At the end of this section we have included a number of examples to illustrate how Theorem \ref{thm:MAINTHMC} can be used.

Let $G = \Sp(V,b)$, where $V$ is a $K$-vector space of dimension $2n$ and $b$ is a non-degenerate alternating bilinear form on $V$. We recall the following which is analogous to the setup of the previous section. By Lemma \ref{lemma:formonwedgeVunique}, we have an alternating $G$-invariant bilinear form $a_V$ on $\wedge^2(V)$ which is unique up to scalar multiples. By Lemma \ref{lemma:basiclemmaonAV}, the bilinear form $a_V$ induces a non-degenerate $G$-invariant bilinear form on $\langle \beta_V \rangle^\perp / \langle \beta_V \rangle \cong L_G(\varpi_2)$, which gives us a representation $f: G \rightarrow \Sp(L_G(\varpi_2), a_V)$. Furthermore, by Lemma \ref{lemma:basiclemmaonAV} the bilinear form $a_V$ is non-degenerate if and only if $n$ is even, in which case we get a representation $f': G \rightarrow \Sp(\wedge^2(V), a_V)$. 

In Theorem \ref{thm:MAINTHMC}, we describe the Hesselink normal form of $f(u)$ for each unipotent element $u \in G$. When $n$ is even, Theorem \ref{thm:MAINTHMC} also describes the Hesselink normal form of $f'(u)$.

We begin with some observations in case where $u \in \Sp(V,b)$ and $(V,b) \cong V(2n)$ or $(V,b) \cong W(n)$ as bilinear $K[u]$-modules. Let $n \geq 2$ and consider $(V, b) = V(2n)$ with a basis $e_1$, $\ldots$, $e_{2n}$ as in Definition \ref{def:vnmod}. That is, we define $b(e_i, e_j) = 1$ if $i+j = 2n+1$ and $0$ otherwise. Furthermore, the action of $u$ on the $e_i$ is given by \begin{align*}
ue_1 &= e_1, \\
ue_i &= e_i + e_{i-1} + \cdots + e_1 \text{ for all } 2 \leq i \leq n+1, \\
ue_i &= e_i + e_{i-1} \text{ for all } n+1 < i \leq 2n.
\end{align*}Throughout we will denote $e_j = 0$ for all $j \leq 0$ and $j > n$.

Suppose that $2^{\alpha} \mid n$, where $\alpha \geq 0$. We define \begin{equation}\label{eq:deltaalphadef}\delta_\alpha = \sum_{1 \leq i \leq n} e_i \wedge e_{2n-2^{\alpha}t_i^{(\alpha)}} \end{equation} where $t_i^{(\alpha)} = \left \lfloor{\frac{i-1}{2^{\alpha}}} \right\rfloor$ for all $1 \leq i \leq n$, cf. \cite[(5.6)]{KorhonenJordanGood}. Note that $\delta_0 = \beta_V$.

\begin{prop}\label{prop:mainpropaltdelta}
Suppose that $2^{\alpha} \mid n$, where $\alpha > 0$. Then $X^{2^{\alpha-1}} \delta_{\alpha} = \delta_{\alpha-1}$.
\end{prop}

\begin{proof}We first note that $X^{2^{\alpha-1}} = u^{2^{\alpha-1}} - 1$, so it follows that \begin{equation}\label{eq:wedgeformulaalpha}X^{2^{\alpha-1}} \cdot (v \wedge w) = (X^{2^{\alpha-1}} v) \wedge (X^{2^{\alpha-1}} w) + v \wedge (X^{2^{\alpha-1}} w) + (X^{2^{\alpha-1}} v) \wedge w\end{equation} for all $v, w \in V$. 

It is clear from the definition that $(u-1)e_i = ue_{i-1}$ for all $1 \leq i \leq n+1$. Thus $(u-1)^ke_i = u^ke_{i-k}$ for all $k \geq 1$ and $1 \leq i \leq n+1$. With $k = 2^{\alpha-1}$, we see that $X^{2^{\alpha-1}}e_i = e_{i-2^{\alpha-1}} + X^{2^{\alpha-1}}e_{i-2^{\alpha-1}}$ for all $1 \leq i \leq n+1$. Consequently \begin{equation}\label{eq:eqalpha1Xfirst}X^{2^{\alpha-1}} e_i = \sum_{j \geq 1} e_{i-2^{\alpha-1}j}\end{equation} for all $1 \leq i \leq n+1$. 

It is clear that $X^{2^{\alpha-1}}e_i = e_{i-2^{\alpha-1}}$ for all $i > n+2^{\alpha-1}$. Since $2n-2^{\alpha}t_i^{(\alpha)} > n + 2^{\alpha-1}$, it follows that \begin{equation}\label{eq:eqalpha1Xsecond}X^{2^{\alpha-1}} e_{2n-2^{\alpha}t_i^{(\alpha)}} = e_{2n-2^{\alpha}t_i^{(\alpha)} - 2^{\alpha-1}}\end{equation} for all $1 \leq i \leq n$. 

Applying~\eqref{eq:eqalpha1Xfirst} and~\eqref{eq:eqalpha1Xsecond} on~\eqref{eq:wedgeformulaalpha}, we see that $X^{2^{\alpha-1}} \delta_{\alpha}$ equals $$\sum_{1 \leq i \leq n} e_i \wedge e_{2n-2^{\alpha}t_i^{(\alpha)} - 2^{\alpha-1}} + \sum_{\substack{1 \leq i \leq n \\ j \geq 1}} e_{i-2^{\alpha-1}j} \wedge \left( e_{2n-2^{\alpha}t_i^{(\alpha)}} + e_{2n-2^{\alpha}t_i^{(\alpha)}-2^{\alpha-1}} \right),$$ where collecting the terms of the form $e_i \wedge v$, we get \begin{equation}\label{eq:firsteqwithsj}\sum_{1 \leq i \leq n} e_i \wedge \left( e_{2n-2^{\alpha}t_{i}^{(\alpha)}-2^{\alpha-1}} + \sum_{j \geq 1} s_j \right),\end{equation} with $s_j = e_{2n-2^{\alpha}t_{i+2^{\alpha-1}j}^{(\alpha)}} + e_{2n-2^{\alpha}t_{i+2^{\alpha-1}j}^{(\alpha)} - 2^{\alpha-1}}$ for all $j \geq 1$.


For $1 \leq i \leq n$, we have $i-1 = t_i^{(\alpha)}2^{\alpha} + r_i$, for some $0 \leq r_i < 2^{\alpha}$. Consider first the case where $0 \leq r_i < 2^{\alpha-1}$. Then $t_{i+2^{\alpha-1}j}^{(\alpha)} = t_i^{(\alpha)} + \frac{j}{2}$ if $j$ is even, and $t_{i+2^{\alpha-1}j}^{(\alpha)} = t_i^{(\alpha)} + \frac{j-1}{2}$ if $j$ is odd. Hence if $j > 1$ is even, then $s_{j} = s_{j+1}$, which implies that $\sum_{j \geq 1} s_j = s_1$. Now $t_{i+2^{\alpha-1}}^{(\alpha)} = t_i^{(\alpha)}$, so $e_{2n-2^{\alpha}t_{i}^{(\alpha)}-2^{\alpha-1}} + \sum_{j \geq 1} s_j = e_{2n-2^{\alpha}t_i^{(\alpha)}}$. Because $0 \leq r_i < 2^{\alpha-1}$, we have $t_i^{(\alpha-1)} = 2t_i^{(\alpha)}$, so the summand in~\eqref{eq:firsteqwithsj} corresponding to $i$ is $e_i \wedge e_{2n - 2^{\alpha-1}t_i^{(\alpha-1)}}$.

If $2^{\alpha-1} \leq r_i < 2^{\alpha}$, a similar calculation shows that $s_j = s_{j+1}$ for all $j \geq 1$ odd, and thus $\sum_{j \geq 1} s_j = 0$. In this case the summand in~\eqref{eq:firsteqwithsj} corresponding to $i$ is $e_i \wedge e_{2n - 2^{\alpha}t_i^{(\alpha)} - 2^{\alpha-1}} = e_i \wedge e_{2n - 2^{\alpha-1}t_i^{(\alpha-1)}}$, since $t_i^{(\alpha-1)} = 2t_i^{(\alpha)} + 1$ when $2^{\alpha-1} \leq r_i < 2^{\alpha}$. Thus we conclude that~\eqref{eq:firsteqwithsj} equals $$\sum_{1 \leq i \leq n} e_i \wedge e_{2n - 2^{\alpha-1}t_i^{(\alpha-1)}} = \delta_{\alpha-1}$$ as claimed.\end{proof}




\begin{seur}\label{cor:deltaimg}
Suppose that $2^{\alpha} \mid n$, where $\alpha \geq 0$. Then $X^{2^{\alpha}-1} \delta_{\alpha} = \delta_0 = \beta_V$.
\end{seur}

\begin{proof}If $\alpha = 0$, the claim follows since $\delta_{\alpha} = \delta_0 = \beta_V$. If $\alpha > 0$, then $2^{\alpha}-1 = \sum_{0 \leq \beta \leq \alpha-1} 2^{\beta}$, so the claim follows using Proposition \ref{prop:mainpropaltdelta}.\end{proof}

\begin{lemma}\label{lemma:hesselinkmainlemma}
Let $(V, b)$ be a bilinear $K[u]$-module such that $(V, b) \cong V(2n)$ or $(V, b) \cong W(n)$, where $n > 0$. Assume that $2^{\alpha} \mid n$, where $\alpha \geq 0$. Then: 
\begin{enumerate}[\normalfont (i)]
\item There exists $\delta \in \wedge^2(V)$ such that $X^{2^{\alpha}-1} \delta = \beta_V$ and $\varphi_V(\delta) = n/2^{\alpha}$.
\item We have $\operatorname{Ker} X_{\wedge^2(V)}^{2^{\alpha}-1} \subseteq \operatorname{Ker} \varphi_V$.
\item If $\alpha = \nu_2(n)$, then $\operatorname{Ker} X_{\wedge^2(V)}^{2^{\alpha}} \not\subseteq \operatorname{Ker} \varphi_V$.
\item If $\alpha < \nu_2(n)$, then $a_V(X^{2^{\alpha}-1}v,v) = 0$ for all $v \in \operatorname{Ker} X_{\wedge^2(V)}^{2^{\alpha}}$.
\item If $\alpha = \nu_2(n)$ and $(V, b) \cong V(2n)$, then $a_V(X^{2^{\alpha}-1}v,v) = 0$ for all $v \in \operatorname{Ker} X_{\operatorname{Ker} \varphi_V}^{2^{\alpha}}$.
\item If $\alpha = \nu_2(n) > 0$ and $(V, b) \cong W(n)$, then $a_V(X^{2^{\alpha}-1}v,v) \neq 0$ for some $v \in \operatorname{Ker} X_{\operatorname{Ker} \varphi_V}^{2^{\alpha}}$.
\end{enumerate}

\end{lemma}

\begin{proof}

For (i), we first consider the case where $(V, b) \cong V(2n)$. Take a basis $e_1$, $\ldots$, $e_{2n}$ as above and $\delta_\alpha$ as in~\eqref{eq:deltaalphadef}. Then $X^{2^{\alpha}-1} \delta_\alpha = \beta_V$ by Corollary \ref{cor:deltaimg}. For the summands in~\eqref{eq:deltaalphadef}, for $1 \leq i \leq n$ we have $\varphi_V(e_i \wedge e_{2n-2^{\alpha}t_i^{(\alpha)}}) = 1$ if $2n-2^{\alpha}t_i^{(\alpha)}+i = 2n+1$ and $0$ otherwise. Now $2n-2^{\alpha}t_i^{(\alpha)}+i = 2n+1$ if and only if $2^{\alpha} \mid i-1$, so we deduce that $\varphi_V(\delta_\alpha) = n/2^{\alpha}$.

Next consider (i) in the case where $(V, b) \cong W(n)$. Fix a basis $e_1$, $\ldots$, $e_n$, $f_1$, $\ldots$, $f_n$ of $V$ such that the subspaces $A = \langle e_1, \ldots, e_n \rangle$ and $B = \langle f_1, \ldots, f_n \rangle$ are $u$-invariant and totally singular, and $b(e_i, f_j) = \delta_{i,j}$. With a suitable choice of basis, we also arrange $ue_1 = e_1$ and $ue_i = e_{i} + e_{i-1}$ for all $1 < i \leq n$. Then $B \cong A^*$, with an isomorphism of $K[u]$-modules defined by $f_i \mapsto e_i^*$. Recall that $\wedge^2(V)$ decomposes as $$\wedge^2(V) = \wedge^2(A) \oplus \wedge^2(B) \oplus A \wedge B.$$ By the proof of Lemma \ref{lemma:pairedaltsquare}, we have $(A \wedge B, a_V) \cong (A \otimes A^*, b_A)$, with an isomorphism of bilinear $K[u]$-modules given by $e_i \wedge f_j \mapsto e_i \otimes e_j^*$. Thus by \cite[Corollary 5.5]{KorhonenJordanGood}, with $$\delta = \sum_{1 \leq i \leq 2^{\alpha}} \sum_{0 \leq j \leq n/2^{\alpha} - 1} e_{j2^{\alpha}+i} \wedge f_{j2^{\alpha}+1}$$ we get $X^{2^{\alpha}-1} \delta = \sum_{1 \leq i \leq n} e_i \wedge f_i = \beta_V$. Since $\varphi_V(\delta) = n/2^{\alpha}$, this completes the proof of (i).

Claim (ii) is obvious in the case where $\alpha = 0$. If $\alpha > 0$, then $n$ is even and $\operatorname{Ker} \varphi_V = \langle \beta_V \rangle^\perp$. In this case (ii) follows from (i) and Lemma \ref{lemma:vperpmodvdescription} (i).

For (iii), suppose that $\alpha = \nu_2(n)$. If $\delta$ is as in (i), then $\delta \in \operatorname{Ker} X_{\wedge^2(V)}^{2^{\alpha}}$ since $\beta_V$ is a fixed point. Since $\varphi_V(\delta) = n/2^{\alpha} \neq 0$, claim (iii) follows.

In the case where $V \cong V(2n)$, claim (iv) follows from Lemma \ref{lemma:epsilonequivalence}, since the smallest Jordan block size of $u$ on $\wedge^2(V)$ is $2^{\nu_2(n)}$ (Lemma \ref{lemma:minblockinwedge2}). Consider then (iv) in the case where $V \cong W(n)$, and let $V = A \oplus B$ as in the proof of (i) above. For $v \in \operatorname{Ker} X_{\wedge^2(V)}^{2^{\alpha}}$, we can write $v = v' + v''$, where $v' \in \operatorname{Ker} X_{\wedge^2(A) \oplus \wedge^2(B)}^{2^{\alpha}}$ and $v'' \in \operatorname{Ker} X_{A \wedge B}^{2^{\alpha}}$. We have $a_V(X^{2^{\alpha}-1}v',v') = 0$ since $(\wedge^2(A) \oplus \wedge^2(B), a_V)$ is a paired module (Lemma \ref{lemma:pairedaltsquare}). Furthermore, by Lemma \ref{lemma:epsilonequivalence} we get $a_V(X^{2^{\alpha}-1}v'',v'') = 0$, since the smallest Jordan block size of $u$ in $A \wedge B \cong A \otimes A^*$ is $2^{\nu_2(n)} > 2^{\alpha}$ \cite[Lemma 4.3]{KorhonenJordanGood}. Thus $a_V(X^{2^{\alpha}-1}v,v) = 0$ by Lemma \ref{lemma:basicXdlinear} (iv).


Claim (v) is clear if $\alpha = 0$, so suppose that $\alpha > 0$ and let $\delta \in \wedge^2(V)$ be as in (i). Now $2 \nmid \frac{n}{2^{\alpha}}$, so $\delta \not\in \langle \beta_V \rangle^\perp$ and from Lemma \ref{lemma:vperpmodvdescription} (iv) and Lemma \ref{lemma:minblockinwedge2} we conclude that $u$ has no Jordan blocks of size $2^{\alpha}$ on $\langle \beta_V \rangle^\perp / \langle \beta_V \rangle = \operatorname{Ker} \varphi_V / \langle \beta_V \rangle$. Thus $a_V(X^{2^{\alpha}-1}v,v) = 0$ for all $v \in \operatorname{Ker} X_{\operatorname{Ker} \varphi_V}^{2^{\alpha}}$ by Lemma \ref{lemma:epsilonequivalence}.

For (vi), suppose that $(V, b) \cong W(n)$ and let $V = A \oplus B$ as in the proof of (i) above. For $2^{\alpha-1} + 1 \leq i \leq 2^{\alpha}$, similarly to~\eqref{eq:defofvi} we define $$w_i = \sum_{j \geq 0} e_{i+j2^{\alpha}} \wedge f_{i-2^{\alpha-1}+j2^{\alpha}}.$$ Now $2^{\alpha}$ occurs as a Jordan block size in $V_n \otimes V_n$ for example by \cite[Lemma 4.3]{KorhonenJordanGood}, so from $(A \wedge B, a_V) \cong (A \otimes A^*, b_A)$ and Lemma \ref{lemma:nonsingularsinVV} we conclude that $X^{2^\alpha} \cdot w_i = 0$ and $a_V(X^{2^\alpha-1}w_i,w_i) \neq 0$ for a suitable choice of $i$. Since $w_i \in \operatorname{Ker} \varphi_V$, this proves (vi).\end{proof}


We will now be able to prove Theorem \ref{thm:MAINTHMC}, our second main result. We refer the reader to the introduction for the statement of the theorem.

\begin{proof}[Proof of Theorem \ref{thm:MAINTHMC}]
We first recall the setup of the theorem, as stated in the introduction. Let $G = \Sp(V, b)$, where $\dim V = 2n$ for some $n \geq 2$. Let $u \in G$ be unipotent. For $t \geq 0$, let $d_1$, $\ldots$, $d_t$ be the Jordan block sizes $d$ of $u$ such that $\varepsilon_{V,b}(d) = 0$, and for $s \geq 0$ let $2d_{t+1}$, $\ldots$, $2d_{t+s}$ be the Jordan block sizes $d$ of $u$ such that $\varepsilon_{V,b}(d) = 1$. Write $V \cong V_{d_1}^{n_1} \oplus \cdots \oplus V_{d_t}^{n_t} \oplus V_{2d_{t+1}}^{n_{t+1}} \oplus \cdots \oplus V_{2d_{t+s}}^{n_{t+s}}$ as $K[u]$-modules, where $n_r > 0$ for all $1 \leq r \leq t+s$.

Set $\alpha = \nu_2(\gcd(d_1, \ldots, d_{t+s}))$. Suppose that $\wedge^2(V) \cong \oplus_{d \geq 1} V_d^{\lambda(d)}$ and $L_G(\varpi_2) \cong \oplus_{d \geq 1} V_d^{\lambda'(d)}$ as $K[u]$-modules, where $\lambda(d), \lambda'(d) \geq 0$ for all $d \geq 1$. We identify $L_G(\varpi_2)$ as the subquotient $(\langle \beta_V \rangle^\perp / \langle \beta_V \rangle, a_V)$ of $(\wedge^2(V), a_V)$, which is justified by Lemma \ref{lemma:basiclemmaonAV} (iv). Set $\varepsilon := \varepsilon_{\wedge^2(V), a_V}$ and $\varepsilon' := \varepsilon_{L_G(\varpi_2), a_V}$.

Before the beginning the actual proof, we still need some additional notation. By Theorem \ref{thm:hesselinkepsilon}, we have $$(V,b) \cong W(d_1)^{n_1/2} \perp \cdots \perp W(d_t)^{n_t/2} \perp V(2d_{t+1})^{n_{t+1}} \perp \cdots \perp V(2d_{t+s})^{n_{t+s}}$$ as bilinear $K[u]$-modules. Then writing $V$ as an orthogonal direct sum of indecomposables, we have $$V = W_1 \perp \cdots \perp W_{t_0} \perp W_{t_0+1} \perp \cdots \perp W_{t_0+s_0}$$ where $W_r$ is $u$-invariant for all $1 \leq r \leq t_0+s_0$, and furthermore for all $1 \leq r \leq t_0$ we have $(W_r, b) \cong W(d_{\pi(r)})$ for some $1 \leq \pi(r) \leq t$, and for all $t_0+1 \leq r \leq t_0+s_0$ we have $(W_r, b) \cong V(2d_{\pi(r)})$ for some $t+1 \leq \pi(r) \leq t+s$. For $1 \leq r \leq t_0$, let $W_r = A_r \oplus B_r$ be a totally singular decomposition into two $K[u]$-submodules $A_r \cong B_r \cong V_{d_{\pi(r)}}$.

It is easy to see that \begin{equation}\label{eq:thmCdirsum}\wedge^2(V) = \bigoplus_{1 \leq r \leq t_0+s_0} \wedge^2(W_r) \oplus \bigoplus_{1 \leq r < r' \leq t_0+s_0} W_r \wedge W_{r'}\end{equation} where as bilinear $K[u]$-modules $(\wedge^2(W_r), a_V) \cong (\wedge^2(W_r), a_{W_r})$ for all $1 \leq r \leq t_0+s_0$, and $(W_r \wedge W_{r'}, a_V) \cong (W_r, b) \otimes (W_{r'}, b)$ for all $1 \leq r < r' \leq t_0+s_0$.

If $2 \nmid n$, then it follows from Lemma \ref{lemma:basiclemmaonAV} that $\wedge^2(V) = \operatorname{Ker} \varphi_V \oplus \langle \beta_V \rangle$, where $\operatorname{Ker} \varphi_V \cong L_G(\varpi_2)$. It is clear in this case that $\lambda'(1) = \lambda(1) - 1$ and $\lambda'(d) = \lambda(d)$ for all $d > 1$, as in Theorem \ref{thm:MAINTHMA} (i). 

We will next describe the values of $\lambda'$ in the case where $2 \mid n$, using Lemma \ref{lemma:vperpmodvdescription}. Note that in this case $\langle \beta_V \rangle^\perp = \operatorname{Ker} \varphi_V$ (Lemma \ref{lemma:basiclemmaonAV} (ii)). We will first show that \begin{equation}\label{eq:KERXAlPHAm1}\operatorname{Ker} X_{\wedge^2(V)}^{2^{\alpha}-1} \subseteq \langle \beta_V \rangle^\perp\end{equation} and \begin{equation}\label{eq:KERXAlPHA}\operatorname{Ker} X_{\wedge^2(V)}^{2^{\alpha}} \not\subseteq \langle \beta_V \rangle^\perp.\end{equation} 

Clearly $W_r \wedge W_{r'} \subseteq \langle \beta_V \rangle^\perp$ for all $1 \leq r < r' \leq t_0+s_0$, so $\operatorname{Ker} X_{W_r \wedge W_{r'}}^{2^{\alpha}-1} \subseteq \langle \beta_V \rangle^\perp$. Furthermore $\operatorname{Ker} X_{\wedge^2(W_r)}^{2^{\alpha}-1} \subseteq \operatorname{Ker} \varphi_{W_r} = \wedge^2(W_r) \cap \operatorname{Ker} \varphi_V$ for all $1 \leq r \leq t_0+s_0$ by Lemma \ref{lemma:hesselinkmainlemma} (ii), so~\eqref{eq:KERXAlPHAm1} follows. 

Let $1 \leq r \leq t_0+s_0$ be such that $\nu_2(d_{\pi(r)}) = \alpha$. Then $\operatorname{Ker} X_{\wedge^2(W_r)}^{2^{\alpha}} \not\subseteq \operatorname{Ker} \varphi_{W_r} = \wedge^2(W_r) \cap \operatorname{Ker} \varphi_V$ by Lemma \ref{lemma:hesselinkmainlemma} (iii), which proves~\eqref{eq:KERXAlPHA}.

For all $1 \leq r \leq t_0+s_0$, let $\delta_r \in \wedge^2(W_r)$ be as in Lemma \ref{lemma:hesselinkmainlemma} (i), so $X^{2^{\alpha}-1} \delta_r = \beta_{W_r}$ and $\varphi_{W_r}(\delta_r) = \varphi_V(\delta_r) = d_{\pi(r)}/2^{\alpha}$. Set $\delta = \delta_1 + \cdots + \delta_{t_0+s_0}$. Then $$X^{2^{\alpha}-1} \delta = \beta_{W_1} + \cdots + \beta_{W_{t_0+s_0}} = \beta_V$$ and furthermore $\varphi_V(\delta) = d_{\pi(1)}/2^{\alpha} + \cdots + d_{\pi(t_0+s_0)}/2^{\alpha} = n/2^{\alpha}$. Thus \begin{equation}\label{eq:statementondelta}\delta \in \langle \beta_V \rangle^\perp \text{ if and only if } 2 \mid \frac{n}{2^{\alpha}}.\end{equation} With~\eqref{eq:KERXAlPHAm1} --~\eqref{eq:statementondelta} and Lemma \ref{lemma:vperpmodvdescription}, we conclude that the values of $\lambda'$ are given in terms of $\lambda$ as in (i) -- (iii) of Theorem \ref{thm:MAINTHMA}.

Next we consider the values of $\varepsilon$ and prove (iv). Let $d \geq 1$ be such that $\varepsilon(d) = 1$. Then $d > 1$, and there exists $v \in \operatorname{Ker} X_{\wedge^2(V)}^{d}$ such that $a_V(X^{d-1}v,v) \neq 0$. Furthermore by~\eqref{eq:thmCdirsum} and Lemma \ref{lemma:basicXdlinear} (iv), we can choose $v$ such that $v \in \wedge^2(W_r)$ for some $1 \leq r \leq t_0+s_0$, or $v \in W_r \wedge W_{r'}$ for some $1 \leq r < r' \leq t_0+s_0$. 

Suppose first that $v \in \wedge^2(W_r)$ for some $1 \leq r \leq t_0$. Then $$\wedge^2(W_r) = \wedge^2(A_r) \oplus \wedge^2(B_r) \oplus A_r \wedge B_r,$$ so $v = v' + v''$, where $v' \in \operatorname{Ker} X_{\wedge^2(A_r) \oplus \wedge^2(B_r)}^{d}$ and $v'' \in \operatorname{Ker} X_{A_r \wedge B_r}^{d}$. The bilinear $K[u]$-module $(\wedge^2(A_r) \oplus \wedge^2(B_r), a_V)$ is a paired module (Lemma \ref{lemma:pairedaltsquare}). Hence $a_V(X^{d-1}v',v') = 0$ by Lemma \ref{lemma:pairedjordan} and then $a_V(X^{d-1}v,v) = a_V(X^{d-1}v'',v'') \neq 0$ by Lemma \ref{lemma:basicXdlinear} (iv). Thus by Lemma \ref{lemma:epsilonequivalence}, there is a Jordan block of size $d$ in $A_r \wedge B_r \cong V_{d_{\pi(r)}} \otimes V_{d_{\pi(r)}}$. Now it follows from Theorem \ref{thm:GPXtensorsquare} that $d = 2^{\beta}$ for some $2^{\beta} > 1$ occurring in the consecutive-ones binary expansion of $d_{\pi(r)}$, so (iv)(a) holds. 

If $v \in \wedge^2(W_r)$ for some $t_0+1 \leq r \leq t_0+s_0$, then by Lemma \ref{lemma:epsilonequivalence} there is a Jordan block of size $d$ in $\wedge^2(W_r) \cong \wedge^2(V_{2d_{\pi(r)}})$. In other words, case (iv)(b) holds.

Suppose that $v \in W_r \wedge W_{r'}$ for some $1 \leq r < r' \leq t_0+s_0$. If $1 \leq r \leq t_0$, then $W_r \cong W(d_{\pi(r)})$ is a paired module, and thus so is $(W_r \wedge W_{r'}, a_V) \cong W(d_{\pi(r)}) \otimes (W_{r'}, a_V)$ (Lemma \ref{lemma:tensorpaired}). But in that case $a_V(X^{d-1}v,v) = 0$ by Lemma \ref{lemma:pairedjordan}, contradiction. Thus we must have $t_0+1 \leq r \leq t_0+s_0$, and so $(W_r \wedge W_{r'}, a_V) \cong V(2d_{\pi(r)}) \otimes V(2d_{\pi(r')})$. By Theorem \ref{thm:formtensorproducts} and Lemma \ref{lemma:epsilonequivalence}, the fact that $a_V(X^{d-1}v,v) \neq 0$ implies that $\beta = \nu_2(d_{\pi(r)}) = \nu_2(d_{{\pi(r')}})$, and furthermore $d = 2^{\beta+1}d'$, where $d'$ is the unique odd Jordan block size in $V_{d_{\pi(r)}/2^{\beta}} \otimes V_{d_{\pi(r')} / 2^{\beta}}$. In other words, we are in case (iv)(c).

For the converse of (iv), we consider (iv)(a) -- (iv)(c). In case (iv)(a), we have $d = 2^{\beta}$ for some $2^{\beta} > 1$ occurring in the consecutive-ones binary expansion of $d_{\pi(r)}$ for some $1 \leq r \leq t_0$. Since $(A_r \wedge B_r, a_V) \cong (A_r \otimes A_r^*, b_{A_r})$ as bilinear $K[u]$-modules by Lemma \ref{lemma:pairedaltsquare}, it follows from Theorem \ref{thm:MAINTHMA} that $a_V(X^{d-1}v,v) \neq 0$ for some $v \in \operatorname{Ker} X_{A_r \wedge B_r}^{d}$. Thus $\varepsilon(d) = 1$ in this case. 

For statement (iv)(b), let $d > 1$ be a Jordan block size of $u$ in $\wedge^2(V_{2d_{\pi(r)}})$ for some $t_0+1 \leq r \leq t_0+s_0$. Note that $d$ occurs in $\wedge^2(V_{2d_{\pi(r)}})$ with odd multiplicity by Lemma \ref{lemma:blockmultiplicitieswedge}. If $d_{\pi(r)}$ is even, then $(\wedge^2(W_r), a_V)$ is non-degenerate (Lemma \ref{lemma:basiclemmaonAV} (ii)), and thus $a_V(X^{d-1}v,v) \neq 0$ for some $v \in \operatorname{Ker} X_{\wedge^2(W_r)}^{d}$ by Lemma \ref{lemma:basicepsilonoddeven} (i). In the case where $d_{\pi(r)}$ is odd, it follows from Lemma \ref{lemma:basiclemmaonAV} that $\wedge^2(W_r) = \operatorname{Ker} \varphi_{W_r} \oplus \langle \beta_{W_r} \rangle$, where $(\operatorname{Ker} \varphi_{W_r}, a_V)$ is non-degenerate. The multiplicity of $d$ in $\operatorname{Ker} \varphi_{W_r}$ is the same as in $\wedge^2(W_r)$, in particular the multiplicity is odd. Hence $a_V(X^{d-1}v,v) \neq 0$ for some $v \in \operatorname{Ker} X_{\operatorname{Ker} \varphi_{W_r}}^{d}$ by Lemma \ref{lemma:basicepsilonoddeven} (i). We conclude then that $\varepsilon(d) = 1$.

In case (iv)(c), we have $d = d'2^{\beta+1}$, where $\beta = \nu_2(d_{\pi(r)}) = \nu_2(d_{{\pi(r')}})$ for some $t_0+1 \leq r < r' \leq t_0+s_0$ and $d'$ is the unique odd Jordan block size in $V_{d_{\pi(r)}/2^{\beta}} \otimes V_{d_{\pi(r')}/2^{\beta}}$. We have $(W_r \wedge W_{r'}, a_V) \cong (W_r,b) \otimes (W_{r'},b) \cong V(2d_{\pi(r)}) \otimes V(2d_{\pi(r')})$ as bilinear $K[u]$-modules. Hence $V(2d)$ is an orthogonal direct summand of $(W_r \wedge W_{r'}, a_V)$ by Theorem \ref{thm:formtensorproducts}, and so it follows from Lemma \ref{lemma:epsilonequivalence} that $a_V(X^{d-1}v,v) \neq 0$ for some $v \in \operatorname{Ker} X_{W_r \wedge W_{r'}}^{d}$. Thus $\varepsilon(d) = 1$, which completes the proof of (iv).

Next we calculate $\varepsilon'$ and prove claims (v) and (vi). If $2 \nmid n$, then $\wedge^2(V) = \operatorname{Ker} \varphi_V \oplus \langle \beta_V \rangle$, where $\operatorname{Ker} \varphi_V \cong L_G(\varpi_2)$ and $\beta_V \in \operatorname{rad} a_V$ (Lemma \ref{lemma:basiclemmaonAV}). In this case it is clear that $\varepsilon'(d) = \varepsilon(d)$ for all $d \geq 1$. If $2 \mid n$ and $\alpha = 0$, then $\varepsilon'(d) = \varepsilon(d)$ for all $d \geq 1$ by Lemma \ref{lemma:vperpmodvdescription} (ii). This completes the proof of (v), so we will consider (vi) and suppose for the rest of the proof that $2 \mid n$ and $\alpha > 0$. 

Combining~\eqref{eq:KERXAlPHAm1},~\eqref{eq:KERXAlPHA}, and Lemma \ref{lemma:vperpmodvdescription} (iii) -- (iv), we see that $\varepsilon'(d) = \varepsilon(d)$ for all $d \neq 2^{\alpha}, 2^{\alpha}-2$, as is claimed by (vi). We prove (vi)(a) next, that is, we show that $\varepsilon(2^{\alpha}) = 1$. Let $1 \leq r \leq t_0+s_0$ be such that $\nu_2(d_{\pi(r)}) = \alpha$. If $1 \leq r \leq t_0$, then $2^{\alpha} > 1$ occurs in the consecutive-ones binary expansion of $d_{\pi(r)}$, and thus $\varepsilon(2^{\alpha}) = 1$ by (iv)(a). If $t_0+1 \leq r \leq t_0+s_0$, then $2^{\alpha} > 1$ occurs as the smallest Jordan block size in $\wedge^2(V_{2d_{\pi(r)}})$ (Lemma \ref{lemma:minblockinwedge2}), and so $\varepsilon(2^{\alpha}) = 1$ by (iv)(b).

For (vi)(b), suppose first that $\nu_2(d_{\pi(r)}) = \alpha$ for some $1 \leq r \leq t_0$. By Lemma \ref{lemma:hesselinkmainlemma}, there exists $v \in \operatorname{Ker} \varphi_{W_r}$ such that $X^{2^{\alpha}}v = 0$ and $a_V(X^{2^{\alpha}-1}v,v) \neq 0$. Then $v \in \operatorname{Ker} \varphi_V = \langle \beta_V \rangle^\perp$, so we conclude that $\varepsilon'(2^{\alpha}) = 1$.

For the other direction of (vi)(b), suppose that $\nu_2(d_{\pi(r)}) > \alpha$ for all $1 \leq r \leq t_0$. We will show that $\varepsilon'(2^{\alpha}) = 0$, which is equivalent to the claim that $a_V(X^{2^{\alpha}-1}v,v) = 0$ for all $v \in \operatorname{Ker} X_{\operatorname{Ker} \varphi_V}^{2^{\alpha}}$. First we describe $\operatorname{Ker} X_{\operatorname{Ker} \varphi_V}^{2^{\alpha}}$. If $\nu_2(d_{\pi(r)}) > \alpha$, then $\operatorname{Ker} X_{\wedge^2(W_r)}^{2^{\alpha}} \subseteq \operatorname{Ker} \varphi_{W_r} \subseteq \operatorname{Ker} \varphi_V$ by Lemma \ref{lemma:hesselinkmainlemma} (ii). If $\nu_2(d_{\pi(r)}) = \alpha$, then $\varphi_V(\delta_r) = d_{\pi(r)}/2^{\alpha} = 1$, so $\operatorname{Ker} X_{\wedge^2(W_r)}^{2^{\alpha}} = \operatorname{Ker} X_{\operatorname{Ker} \varphi_{W_r}}^{2^{\alpha}} \oplus \langle \delta_r \rangle$. Furthermore, for all $1 \leq r < r' \leq t_0+s_0$ we have $W_r \wedge W_{r'} \subseteq \operatorname{Ker} \varphi_V$. 

Thus any $v \in \operatorname{Ker} X_{\operatorname{Ker} \varphi_V}^{2^{\alpha}}$ can be written in the form $$v = \sum_{\substack{1 \leq r \leq t_0+s_0 \\ \nu_2(d_{\pi(r)}) > \alpha}} z_r + \sum_{\substack{1 \leq r \leq t_0+s_0 \\ \nu_2(d_{\pi(r)}) = \alpha}} \left( z_r + \mu_r \delta_r \right) + \sum_{1 \leq r < r' \leq t_0+s_0} z_{r,r'},$$  where $z_r \in \operatorname{Ker} X_{\operatorname{Ker} \varphi_{W_r}}^{2^{\alpha}}$ for all $1 \leq r \leq t_0+s_0$, $z_{r,r'} \in \operatorname{Ker} X_{W_r \wedge W_{r'}}^{2^{\alpha}}$ for all $1 \leq r < r' \leq t_0+s_0$, and \begin{equation}\label{eq:sumofcoefs0}\sum_{\substack{1 \leq r \leq t_0+s_0 \\ \nu_2(d_{\pi(r)}) = \alpha}} \mu_r = 0.\end{equation} 

Now $W_r \cong V(2d_{\pi(r)})$ for all $1 \leq r \leq t_0+s_0$ such that $\nu_2(d_{\pi(r)}) = \alpha$, so we conclude that $a_V(X^{2^{\alpha}-1}z_r,z_r) = 0$ for all $1 \leq r \leq t_0+s_0$ by Lemma \ref{lemma:hesselinkmainlemma} (iv) -- (v). We will show next that $a_V(X^{2^{\alpha}-1}z_{r,r'},z_{r,r'}) = 0$ for all $1 \leq r < r' \leq t_0+s_0$. If $1 \leq r \leq t_0$, then $(W_r \wedge W_{r'}, a_V) \cong W(d_{\pi(r)}) \otimes (W_{r'}, b)$ is a paired module (Lemma \ref{lemma:tensorpaired}), and so $a_V(X^{2^{\alpha}-1}z_{r,r'},z_{r,r'}) = 0$ by Lemma \ref{lemma:pairedjordan}. If $t_0+1 \leq r \leq t_0+s_0$, then $(W_r \wedge W_{r'}, a_V) \cong V(2d_{\pi(r)}) \otimes V(2d_{{\pi(r')}})$. By~\eqref{eq:indreseq0} the smallest Jordan block size in $V_{2d_{\pi(r)}} \otimes V_{2d_{{\pi(r')}}}$ is $\geq 2^{\alpha+1}$, so $a_V(X^{2^{\alpha}-1}z_{r,r'},z_{r,r'}) = 0$ by Lemma \ref{lemma:epsilonequivalence}. Hence by Lemma \ref{lemma:basicXdlinear} (iv) we get \begin{align*}a_V(X^{2^{\alpha}-1}v,v) &= \sum_{\substack{1 \leq r \leq t_0+s_0 \\ \nu_2(d_{\pi(r)}) = \alpha}} \mu_r^2 a_V(X^{2^{\alpha}-1}\delta_r, \delta_r) \\ &= \sum_{\substack{1 \leq r \leq t_0+s_0 \\ \nu_2(d_{\pi(r)}) = \alpha}} \mu_r^2 a_V(\beta_{W_r}, \delta_r) \\ &= \sum_{\substack{1 \leq r \leq t_0+s_0 \\ \nu_2(d_{\pi(r)}) = \alpha}} \mu_r^2\end{align*} which equals zero by~\eqref{eq:sumofcoefs0}. This completes the proof of (vi)(b).

What remains is to prove (vi)(c) and (vi)(d), so suppose that $\alpha > 1$. For (vi)(d), if $2 \mid \frac{n}{2^{\alpha}}$, then $\delta \in \langle \beta_V \rangle^\perp$ by~\eqref{eq:statementondelta} and thus $\varepsilon'(2^{\alpha}-2) = \varepsilon(2^{\alpha}-2) = 0$ by Lemma \ref{lemma:vperpmodvdescription} (iii). Similarly, if $2 \nmid \frac{n}{2^{\alpha}}$, then $\delta \not\in \langle \beta_V \rangle^\perp$ by~\eqref{eq:statementondelta} and thus $\varepsilon'(2^{\alpha}-2) = 1$ by Lemma \ref{lemma:vperpmodvdescription} (iv). This completes the proof of (vi) and the theorem.\end{proof}

In the following we show with small examples how Theorem \ref{thm:MAINTHMC} is applied. Let $G = \Sp(V, b)$, where $\dim V = 2n$ for some $n \geq 2$. Let $u \in G$ be a unipotent element and $V \cong V_{d_1}^{n_1} \oplus \cdots \oplus V_{d_t}^{n_t} \oplus V_{2d_{t+1}}^{n_{t+1}} \oplus \cdots \oplus V_{2d_{t+s}}^{n_{t+s}}$ as $K[u]$-modules, where $d_i$ and $n_i$ are as in Theorem \ref{thm:MAINTHMC}. Equivalently by Theorem \ref{thm:hesselinkepsilon}, we have $V \cong W(d_1)^{n_1/2} \perp \cdots \perp W(d_t)^{n_t/2} \perp V(2d_{t+1})^{n_{t+1}} \perp \cdots \perp V(2d_{t+s})^{n_{t+s}}$ as bilinear $K[u]$-modules. Let $\alpha = \nu_2(\gcd(d_1, \ldots, d_{t+s}))$ as in Theorem \ref{thm:MAINTHMC}. Set $\varepsilon := \varepsilon_{\wedge^2(V), a_V}$ and $\varepsilon' := \varepsilon_{L_G(\varpi_2), a_V}$.



\begin{esim}If $n = 2$ and $V \cong V(4)$, then $\wedge^2(V) \cong V_2 \oplus V_4$ by Theorem \ref{thm:GowLaffey}. We have $\varepsilon(2) = 1$, $\varepsilon(4) = 1$ by Theorem \ref{thm:GowLaffey} (iv)(b), so $(\wedge^2(V), a_V) \cong V(2) \perp V(4)$ as bilinear $K[u]$-modules. Now $\alpha = 1$, so $L_G(\varpi_2) \cong V_4$ by Theorem \ref{thm:MAINTHMC} (applying rule (iii)(c) from Theorem \ref{thm:MAINTHMA}). Hence $(L_G(\varpi_2), a_V) \cong V(4)$ as bilinear $K[u]$-modules.\end{esim}

\begin{esim}If $n = 4$ and $V \cong W(4)$, then $V \cong V_4 \oplus V_4$ and so $\wedge^2(V) \cong V_2^2 \oplus V_4^6$ by Theorem \ref{thm:GowLaffey}. In this case we have a consecutive-ones binary expansion $4 = 2^2$, so by Theorem \ref{thm:MAINTHMC} we find that $\varepsilon(4) = 1$ and $\varepsilon(2) = 0$. Hence $(\wedge^2(V), a_V) \cong W(2) \perp V(4)^6$ as bilinear $K[u]$-modules. Now $\alpha = 1$, so $L_G(\varpi_2) \cong V_2^3 \oplus V_4^5$ by Theorem \ref{thm:MAINTHMC} (applying rule (iii)(a) from Theorem \ref{thm:MAINTHMA}). From Theorem \ref{thm:MAINTHMC} (vi) we conclude that $\varepsilon'(2) = 1$ and $\varepsilon'(4) = 1$, so $(L_G(\varpi_2), a_V) \cong V(2)^3 \perp V(4)^5$ as bilinear $K[u]$-modules.\end{esim}

\begin{esim}If $n = 6$ and $V \cong V(4)^3$, then $\wedge^2(V) \cong V_2^3 \oplus V_4^{15}$ by Theorem \ref{thm:GowLaffey}. We have $\varepsilon(2) = 1$, $\varepsilon(4) = 1$ by Theorem \ref{thm:GowLaffey} (iv)(b), so $(\wedge^2(V), a_V) \cong V(2)^3 \perp V(4)^{15}$ as bilinear $K[u]$-modules. Now $\alpha = 1$, so $L_G(\varpi_2) \cong V_2^2 \oplus V_4^{15}$ by Theorem \ref{thm:MAINTHMC} (applying rule (iii)(c) from Theorem \ref{thm:MAINTHMA}). From Theorem \ref{thm:MAINTHMC} (vi), we conclude that $\varepsilon'(2) = 0$ and $\varepsilon'(4) = 1$, so $(L_G(\varpi_2), a_V) \cong W(2) \perp V(4)^{15}$ as bilinear $K[u]$-modules.\end{esim}

\begin{esim}\label{examples:thmC}In Table \ref{table:examplesofTHMC3}, we illustrate Theorem \ref{thm:MAINTHMC} for $2 \leq n \leq 8$. For $n > 3$, we have only included the cases where $\alpha > 0$. As in Example \ref{esim:EXTHMA}, for a bilinear $K[u]$-module $(W,b)$, we use $({d_1}_{\varepsilon_1}^{n_1}, \ldots, {d_t}_{\varepsilon_t}^{n_t})$ to denote that $W \cong V_{d_1}^{n_1} \oplus \cdots \oplus V_{d_t}^{n_t}$ as $K[u]$-modules and $\varepsilon_{W, b}(d_i) = \varepsilon_i$ for $1 \leq i \leq t$. It is straightforward to see from the results of Section \ref{section:unipclassesinfo} that if $(V, b)$ corresponds to $({d_1}_{\varepsilon_1}^{n_1}, \ldots, {d_t}_{\varepsilon_t}^{n_t})$, then $\alpha = \nu_2(\gcd(d_1', \ldots, d_t'))$, where $d_i' = d_i$ if $\varepsilon_i = 0$ and $d_i' = d_i/2$ if $\varepsilon_i = 1$. Note that the examples in Table \ref{table:examplesofTHMC3} illustrate all possible cases of (iv) -- (vi) in Theorem \ref{thm:MAINTHMC}.\end{esim}


\begin{table}[!htbp]
\centering
\caption{Example cases of Theorem \ref{thm:MAINTHMC} for $2 \leq n \leq 8$, see Example \ref{examples:thmC}.}\label{table:examplesofTHMC3}
\footnotesize
\begin{tabular}{| c | l | l | l | l |}
\hline
$n$              & $(V,b)$       & $(\wedge^2(V), a_V)$         & $(L_G(\varpi_2), a_V)$ & $\alpha$ \\ \hline
				
$n = 2$ & $(4_1)$ & $(2_1, 4_1)$ & $(4_1)$ & $1$ \\
        & $(2_1^2)$ & $(1_0^{2}, 2_1^{2})$ & $(2_1^{2})$ & $0$ \\
        & $(2_0^2)$ & $(1_0^{2}, 2_1^{2})$ & $(1_0^{2}, 2_1)$ & $1$ \\
        & $(1_0^2, 2_1)$ & $(1_0^{2}, 2_0^{2})$ & $(2_0^{2})$ & $0$ \\			
& & & & \\
	$n = 3$ & $(6_1)$ & $(1_0, 6_1, 8_1)$ & $(6_1, 8_1)$ & $0$  \\
        & $(2_1, 4_1)$ & $(1_0, 2_1, 4_1^{3})$ & $(2_1, 4_1^{3})$ & $0$  \\
        & $(1_0^2, 4_1)$ & $(1_0, 2_1, 4_1^{3})$ & $(2_1, 4_1^{3})$ & $0$  \\
        & $(3_0^2)$ & $(1_0, 3_0^{2}, 4_1^{2})$ & $(3_0^{2}, 4_1^{2})$ & $0$  \\
        & $(2_1^3)$ & $(1_0^{3}, 2_1^{6})$ & $(1_0^{2}, 2_1^{6})$ & $0$  \\
        & $(1_0^2, 2_1^2)$ & $(1_0^{3}, 2_1^{6})$ & $(1_0^{2}, 2_1^{6})$ & $0$  \\
        & $(1_0^2, 2_0^2)$ & $(1_0^{3}, 2_1^{6})$ & $(1_0^{2}, 2_1^{6})$ & $0$  \\
        & $(1_0^4, 2_1)$ & $(1_0^{7}, 2_0^{4})$ & $(1_0^{6}, 2_0^{4})$ & $0$  \\
& & & & \\				
$n = 4$ & $(8_1)$ & $(4_1, 8_1^{3})$ & $(2_1, 8_1^{3})$ & $2$ \\
        & $(4_1^2)$ & $(2_1^{2}, 4_1^{6})$ & $(1_0^{2}, 4_1^{6})$ & $1$ \\
        & $(4_0^2)$ & $(2_0^{2}, 4_1^{6})$ & $(2_1^{3}, 4_1^{5})$ & $2$ \\
        & $(2_0^2, 4_1)$ & $(1_0^{2}, 2_1^{3}, 4_1^{5})$ & $(1_0^{4}, 2_1, 4_1^{5})$ & $1$ \\
        & $(2_0^4)$ & $(1_0^{4}, 2_1^{12})$ & $(1_0^{6}, 2_1^{10})$ & $1$ \\
& & & & \\				
$n = 6$ & $(12_1)$ & $(2_1, 4_1, 12_1, 16_1^{3})$ & $(4_1, 12_1, 16_1^{3})$ & $1$ \\
        & $(4_1, 8_1)$ & $(2_1, 4_1^{2}, 8_1^{7})$ & $(4_1^{2}, 8_1^{7})$ & $1$ \\
        & $(2_0^2, 8_1)$ & $(1_0^{2}, 2_1^{2}, 4_1, 8_1^{7})$ & $(1_0^{2}, 2_1, 4_1, 8_1^{7})$ & $1$ \\
        & $(6_0^2)$ & $(1_0^{2}, 2_1^{2}, 6_0^{2}, 8_1^{6})$ & $(1_0^{2}, 2_1, 6_0^{2}, 8_1^{6})$ & $1$ \\
        & $(4_1^3)$ & $(2_1^{3}, 4_1^{15})$ & $(2_0^{2}, 4_1^{15})$ & $1$ \\
        & $(2_0^2, 4_1^2)$ & $(1_0^{2}, 2_1^{4}, 4_1^{14})$ & $(1_0^{2}, 2_1^{3}, 4_1^{14})$ & $1$ \\
        & $(2_0^2, 4_0^2)$ & $(1_0^{2}, 2_1^{4}, 4_1^{14})$ & $(1_0^{2}, 2_1^{3}, 4_1^{14})$ & $1$ \\
        & $(2_0^4, 4_1)$ & $(1_0^{4}, 2_1^{13}, 4_1^{9})$ & $(1_0^{4}, 2_1^{12}, 4_1^{9})$ & $1$ \\
        & $(2_0^6)$ & $(1_0^{6}, 2_1^{30})$ & $(1_0^{6}, 2_1^{29})$ & $1$ \\
& & & & \\	
$n = 8$ & $(16_1)$ & $(8_1, 16_1^{7})$ & $(6_1, 16_1^{7})$ & $3$ \\
        & $(4_1, 12_1)$ & $(2_1^{2}, 4_1^{2}, 12_1^{5}, 16_1^{3})$ & $(1_0^{2}, 4_1^{2}, 12_1^{5}, 16_1^{3})$ & $1$ \\
        & $(2_0^2, 12_1)$ & $(1_0^{2}, 2_1^{3}, 4_1, 12_1^{5}, 16_1^{3})$ & $(1_0^{4}, 2_1, 4_1, 12_1^{5}, 16_1^{3})$ & $1$ \\
        & $(8_1^2)$ & $(4_1^{2}, 8_1^{14})$ & $(3_0^{2}, 8_1^{14})$ & $2$ \\
        & $(8_0^2)$ & $(4_0^{2}, 8_1^{14})$ & $(4_0^{2}, 6_1, 8_1^{13})$ & $3$ \\
        & $(4_1^2, 8_1)$ & $(2_1^{2}, 4_1^{7}, 8_1^{11})$ & $(1_0^{2}, 4_1^{7}, 8_1^{11})$ & $1$ \\
        & $(4_0^2, 8_1)$ & $(2_0^{2}, 4_1^{7}, 8_1^{11})$ & $(2_0^{2}, 3_0^{2}, 4_1^{5}, 8_1^{11})$ & $2$ \\
        & $(2_0^2, 4_1, 8_1)$ & $(1_0^{2}, 2_1^{3}, 4_1^{6}, 8_1^{11})$ & $(1_0^{4}, 2_1, 4_1^{6}, 8_1^{11})$ & $1$ \\
        & $(2_0^4, 8_1)$ & $(1_0^{4}, 2_1^{12}, 4_1, 8_1^{11})$ & $(1_0^{6}, 2_1^{10}, 4_1, 8_1^{11})$ & $1$ \\
        & $(4_1, 6_0^2)$ & $(1_0^{2}, 2_1^{3}, 4_1^{5}, 6_0^{2}, 8_1^{10})$ & $(1_0^{4}, 2_1, 4_1^{5}, 6_0^{2}, 8_1^{10})$ & $1$ \\
        & $(2_0^2, 6_0^2)$ & $(1_0^{4}, 2_1^{4}, 6_0^{10}, 8_1^{6})$ & $(1_0^{6}, 2_1^{2}, 6_0^{10}, 8_1^{6})$ & $1$ \\
        & $(4_1^4)$ & $(2_1^{4}, 4_1^{28})$ & $(1_0^{2}, 2_0^{2}, 4_1^{28})$ & $1$ \\
        & $(4_0^4)$ & $(2_0^{4}, 4_1^{28})$ & $(2_0^{4}, 3_0^{2}, 4_1^{26})$ & $2$ \\
        & $(2_0^2, 4_1^3)$ & $(1_0^{2}, 2_1^{5}, 4_1^{27})$ & $(1_0^{4}, 2_1^{3}, 4_1^{27})$ & $1$ \\
        & $(2_0^4, 4_1^2)$ & $(1_0^{4}, 2_1^{14}, 4_1^{22})$ & $(1_0^{6}, 2_1^{12}, 4_1^{22})$ & $1$ \\
        & $(2_0^4, 4_0^2)$ & $(1_0^{4}, 2_1^{14}, 4_1^{22})$ & $(1_0^{6}, 2_1^{12}, 4_1^{22})$ & $1$ \\
        & $(2_0^6, 4_1)$ & $(1_0^{6}, 2_1^{31}, 4_1^{13})$ & $(1_0^{8}, 2_1^{29}, 4_1^{13})$ & $1$ \\
        & $(2_0^8)$ & $(1_0^{8}, 2_1^{56})$ & $(1_0^{10}, 2_1^{54})$ & $1$ \\	
\hline		
\end{tabular}
\end{table}

\section{Overgroups of distinguished unipotent elements}\label{section:applications}

Let $G$ be a simple algebraic group over $K$. One approach towards understanding the subgroup structure of $G$ is to classify subgroups by the elements that they contain. See for example the survey \cite{SaxlSpecial} for some results in this direction and their applications. To give a specific example, all connected reductive subgroups containing a \emph{regular unipotent element} of $G$ are known by the results in \cite{SaxlSeitz, TestermanZalesski}. Overgroups of regular unipotent elements were studied further in \cite[Section 3]{GuralnickMalle}, motivated by an application to the inverse Galois problem. 

In the PhD thesis of the present author, the main result classifies all maximal closed connected subgroups $G$ that contain a \emph{distinguished unipotent element}, in any characteristic $p > 0$. Recall that a unipotent element of $G$ is \emph{distinguished}, if its centralizer in $G$ does not contain a non-trivial torus.

In this section, we keep our assumption that $\operatorname{char} K = 2$, and apply our main results to classify some subgroups of $\Sp(V,b)$ that contain distinguished unipotent elements. The results of Proposition \ref{prop:applicationAirr}, Proposition \ref{prop:applicationtensor}, and Proposition \ref{prop:applicationCirr} below appeared first in the PhD thesis of the present author. However, using our results we are able to give proofs which are shorter and do not rely on many case-by-case calculations. 

The following definition is convenient for describing distinguished unipotent elements in $\Sp(V,b)$.

\begin{maar}\label{def:distinguishedunip}
Let $u$ be a generator of a cyclic $2$-group and let $(V,b)$ be a bilinear $K[u]$-module. We say that \emph{$u$ acts on $(V,b)$ as a distinguished unipotent element}, if one of the following equivalent conditions hold:
	\begin{enumerate}[\normalfont (i)]
		\item The image of $u$ in $\Sp(V,b)$ is a distinguished unipotent element of $\Sp(V,b)$.
		\item $(V,b) \cong V(2k_1)^{d_1} \perp \cdots \perp V(2k_t)^{d_t}$ as bilinear $K[u]$-modules where $0 < k_1 < \cdots < k_t$ and $d_i \leq 2$ for all $1 \leq i \leq t$.
		\item Every Jordan block size $d$ of $u$ on $V$ is even, has multiplicity at most two, and $\varepsilon_{V,b}(d) = 1$.
		\item The bilinear $K[u]$-module $(V,b)$ does not have any orthogonal direct summands of the form $W(m)$ for $m > 0$.
	\end{enumerate}
\end{maar}

The equivalence of the conditions in Definition \ref{def:distinguishedunip} is seen as follows. The equivalence of (i) and (ii) is given by \cite[Proposition 6.1]{LiebeckSeitzClass}, while the equivalence of (ii) and (iii) follows from Lemma \ref{lemma:epsilonequivalence}. The fact that (ii) implies (iv) is immediate from the Hesselink normal form (Theorem \ref{thm:hesselinkform}). If (iv) holds, then it follows from Theorem \ref{thm:hesselinkform} that $(V,b) \cong V(2k_1)^{d_1} \perp \cdots \perp V(2k_t)^{d_t}$ as bilinear $K[u]$-modules for some integers $0 < k_1 < \cdots < k_t$ and $d_i > 0$. If $d_i > 2$, then from the isomorphism $V(2k_i)^3 \cong W(2k_i) \perp V(2k_i)$ of bilinear $K[u]$-modules (Lemma \ref{lemma:threeorthogonals}) we see that $W(2k_i)$ occurs as an orthogonal direct summand of $(V,b)$, contradiction. Thus $d_i \leq 2$ for all $1 \leq i \leq t$, which proves that (iv) implies (ii).


We shall need the following easy lemma, after which we will be able to prove the main results of this section.

\begin{lemma}\label{lemma:easiestlemma}
Let $G$ be a simple algebraic group and let $f: G \rightarrow \Sp(V,b)$ be a non-trivial representation of $G$. If $u \in G$ is a unipotent element that acts on $(V,b)$ as a distinguished unipotent element, then $u$ is a distinguished unipotent element of $G$.
\end{lemma}

\begin{proof}If $u$ is not a distinguished unipotent element of $G$, then $u$ is centralized by some non-trivial torus $S < G$. In this case $f(S)$ is a non-trivial torus centralizing $f(u)$, so $f(u)$ is not a distinguished unipotent element of $\Sp(V,b)$.\end{proof}

\begin{prop}\label{prop:applicationAtilt}
Let $G = \SL(V)$, where $n = \dim V$ is even. A unipotent element $u \in G$ acts on $(V \otimes V^*, b_V)$ as a distinguished unipotent element if and only if $n = 2$ and $V \cong V_2$ as $K[u]$-modules.
\end{prop}

\begin{prop}\label{prop:applicationAirr}
Let $G = \SL(V)$ and set $n = \dim V$, where $n > 1$. A unipotent element $u \in G$ acts on $(L_G(\varpi_1 + \varpi_{n-1}), b_V)$ as a distinguished unipotent element if and only if $V \cong V_n$ as $K[u]$-modules and $n \in \{2,3,5\}$.
\end{prop}

\begin{proof}[Proof of Proposition \ref{prop:applicationAtilt} and Proposition \ref{prop:applicationAirr}]
The only distinguished unipotent elements in $G$ are the regular unipotent elements \cite[Proposition 3.5]{LiebeckSeitzClass}, so by Lemma \ref{lemma:easiestlemma} we may assume that $V \cong V_n$ as $K[u]$-modules for some $n > 1$. An easy calculation with Theorem \ref{thm:GPXtensorsquare} and Theorem \ref{thm:MAINTHMA} shows the following:

	\begin{itemize}
		\item If $n = 2$, then $(V \otimes V^*, b_V) \cong V(2)^2$ and $(L_G(\varpi_1 + \varpi_{n-1}), b_V) \cong V(2)$.
		\item If $n = 3$, then $(L_G(\varpi_1 + \varpi_{n-1}), b_V) \cong V(4)^2$.
		\item If $n = 5$, then $(L_G(\varpi_1 + \varpi_{n-1}), b_V) \cong V(4)^2 \perp V(8)^2$.
	\end{itemize}

This proves sufficiency in Proposition \ref{prop:applicationAtilt} and Proposition \ref{prop:applicationAirr}. We show next that these are the only cases where $u$ acts on $(V \otimes V^*, b_V)$ or $(L_G(\varpi_1 + \varpi_{n-1}), b_V)$ as a distinguished unipotent element.

Let $n = \sum_{i = 1}^k (-1)^{i+1} 2^{e_i}$ be the consecutive-ones binary expansion of $n$, where $e_1 > \cdots > e_k \geq 0$. Note that $e_{k-1} > e_k+1$ if $k \geq 2$. By Theorem \ref{thm:GPXtensorsquare} $$V \otimes V^* \cong \bigoplus_{1 \leq i \leq k} V_{2^{e_i}}^{d_i}$$ as $K[u]$-modules, where $d_i = 2^{e_i} - \sum_{j = i+1}^k (-1)^{i+j} 2^{e_j+1}$ for all $1 \leq i \leq k$. 

For the other direction of Proposition \ref{prop:applicationAtilt}, suppose that $n$ is even and that $u$ acts on $(V \otimes V^*, b_V)$ as a distinguished unipotent element. Then $d_i \leq 2$ for all $1 \leq i \leq k$. We have $e_k > 0$ since $n$ is even. Thus if $k > 1$, then $d_{k-1} = 2^{e_{k-1}} - 2^{e_k+1} \geq 2^{e_k+1} > 2$, contradiction. Hence $k = 1$, so $n = 2^{e_1}$. Since $d_1 = 2^{e_1}$, we must have $n = 2$, as claimed by Proposition \ref{prop:applicationAtilt}.

For Proposition \ref{prop:applicationAirr}, suppose that $u$ acts on $(L_G(\varpi_1 + \varpi_{n-1}), b_V)$ as a distinguished unipotent element. Since $\nu_2(n) = e_k$, by Theorem \ref{thm:MAINTHMA} we have $d_k \leq 3$ and $d_i \leq 2$ for all $1 \leq i < k$. If $e_k > 1$, then $d_k = 2^{e_k} \geq 4$, contradiction. Thus $e_k \leq 1$.

Suppose that $e_k = 1$. If $k > 1$, then it follows from $e_{k-1} > e_k+1$ that $d_{k-1} = 2^{e_{k-1}} - 2^{e_k+1} \geq 4$, contradiction. Thus $k = 1$, and so $n = 2^{e_1} = 2$.

Consider next $e_k = 0$. Then $d_{k-1} = 2^{e_{k-1}} - 2 \leq 2$ implies that $e_{k-1} \leq 2$. But $e_{k-1} > e_k+1$, which forces $e_{k-1} = 2$. If $k = 2$, then $n = 2^2 - 2^0 = 3$. Suppose then that $k > 2$. In this case $d_{k-2} = 2^{e_{k-2}} - 6 \leq 2$, so we must have $e_{k-2} = 3$. If $k = 3$, then this gives $n = 2^3 - 2^2 + 2^0 = 5$. Finally if $k > 3$, then $d_{k-3} = 2^{e_{k-3}} - 9 \geq 7$, contradiction.\end{proof}

\begin{prop}\label{prop:applicationtensor}
Let $(V_1,b_1)$ and $(V_2,b_2)$ be non-degenerate alternating bilinear $K[u]$-modules, where $1 < \dim V_1 \leq \dim V_2$. Then $u$ acts on $(V_1, b_1) \otimes (V_2, b_2) = (V_1 \otimes V_2, b_1 \otimes b_2)$ as a distinguished unipotent element if and only if we have the following isomorphisms of bilinear $K[u]$-modules:
\begin{enumerate}[\normalfont (i)]
\item $(V_1, b_1) \cong V(2)$,
\item $(V_2, b_2) \cong V(2k_1) \perp \cdots \perp V(2k_t)$, where $0 < k_1 < \cdots < k_t$ are odd integers.
\end{enumerate}

Furthermore, if (i) and (ii) hold, then $(V_1, b_1) \otimes (V_2, b_2) \cong V(2k_1)^2 \perp \cdots \perp V(2k_t)^2$ as bilinear $K[u]$-modules.
\end{prop}

\begin{proof}Suppose that $u$ acts on $(V_1, b_1) \otimes (V_2, b_2)$ as a distinguished unipotent element. Then $u$ must act on $(V_i, b_i)$ as a distinguished unipotent element for $i = 1,2$. Indeed, if $(V_i,b_i)$ had any orthogonal direct summands of the form $W(m)$, then so would $(V_1, b_1) \otimes (V_2, b_2)$ by Proposition \ref{prop:pairedtensorcase}. 

We consider first the case where $(V_i,b_i)$ are orthogonally indecomposable, so suppose that $(V_1,b_1) \cong V(2l)$ and $(V_2,b_2) \cong V(2k)$ for some $1 \leq l \leq k$. Then by Theorem \ref{thm:formtensorproducts} there are at most two indecomposable summands in $V(2l) \otimes V(2k)$, as otherwise some summand would have multiplicity $> 2$ or some $W(m)$ would occur as an orthogonal direct summand. Thus $(V_1,b_1) \cong V(2)$, since the number of indecomposable summands in the $K[u]$-module $V_{2l} \otimes V_{2k}$ is $2l$. By Example \ref{example:v2tensor}, we must have $k$ odd and $(V_1, b_1) \otimes (V_2, b_2) \cong V(2k)^2$.

For the general case, let $V_1 = W_1 \perp \cdots \perp W_s$ and $V_2 = W_1' \perp \cdots \perp W_t'$, where $W_i$ and $W_j'$ are orthogonally indecomposable $K[u]$-modules for all $1 \leq i \leq s$ and $1 \leq j \leq t$. Clearly $u$ acts as a distinguished unipotent element on $(W_i \otimes W_j', b_1 \otimes b_2)$ for all $1 \leq i \leq s$ and $1 \leq j \leq t$, so it follows from the indecomposable case that $(W_i \otimes W_j', b_1 \otimes b_2) \cong V(2k)^2$ for some $k$ odd, where $(W_i,b_1) \cong V(2)$ and $(W_j',b_2) \cong V(2k)$, or $(W_i,b_1) \cong V(2k)$ and $(W_j',b_2) \cong V(2)$. From this and the fact that $\dim V_1 \leq \dim V_2$, it is straightforward to see that $(V_1, b_1) \cong V(2)$ and $(V_2, b_2) \cong V(2k_1) \perp \cdots \perp V(2k_t)$, where $0 < k_1 < \cdots < k_t$ are odd integers, as claimed by the proposition. 

The other direction of the proposition is immediate from Example \ref{example:v2tensor}, which shows that $(V_1, b_1) \otimes (V_2, b_2) \cong V(2k_1)^2 \perp \cdots \perp V(2k_t)^2$ as bilinear $K[u]$-modules.\end{proof}

\begin{prop}\label{prop:applicationCtilt}
Let $G = \Sp(V,b)$, where $\dim V = 2n$ and $n$ is even. A unipotent element $u \in G$ acts on $(\wedge^2(V), a_V)$ as a distinguished unipotent element if and only if $n = 2$ and $(V,b) \cong V(4)$ as bilinear $K[u]$-modules.
\end{prop}

\begin{prop}\label{prop:applicationCirr}
Let $G = \Sp(V,b)$, where $\dim V = 2n$ and $n \geq 2$. A unipotent element $u \in G$ acts on $(L_G(\varpi_2), a_V)$ as a distinguished unipotent element if and only if one of the following conditions hold:
	\begin{enumerate}[\normalfont (i)]
		\item $(V, b) \cong V(2n)$ as bilinear $K[u]$-modules and $n \in \{2,3,5\}$.
		\item $(V, b) \cong V(2) \perp V(2n-2)$ as bilinear $K[u]$-modules and $n \in \{2,6\}$.
	\end{enumerate}
\end{prop}

\begin{proof}[Proof of Proposition \ref{prop:applicationCtilt} and Proposition \ref{prop:applicationCirr}]
Let $u \in G$ be unipotent. We can assume that $u$ is a distinguished unipotent element (Lemma \ref{lemma:easiestlemma}), in which case $(V,b) = W_1 \perp \cdots \perp W_t$, where $(W_i,b) \cong V(2d_i)$ for all $1 \leq i \leq t$, for some integers $d_i > 0$. 

If $t > 1$, then $(W_i \wedge W_j, a_V) \cong (W_i, b) \otimes (W_j, b)$ is a non-degenerate subspace of $(\wedge^2(V), a_V)$ for all $1 \leq i < j \leq t$. In fact, since $W_i \wedge W_j$ is contained in $\langle \beta_V \rangle^\perp$ and $\left( W_i \wedge W_j \right) \cap \langle \beta_V \rangle = 0$, it follows that $(W_i \wedge W_j, a_V)$ embeds into $(L_G(\varpi_2), a_V)$ as a non-degenerate subspace. 

Thus if $u$ acts on $(\wedge^2(V), a_V)$ or $(L_G(\varpi_2), a_V)$ as a distinguished unipotent element, then $u$ acts on $(W_i \wedge W_j, a_V) \cong (W_i, b) \otimes (W_j, b)$ as a distinguished unipotent element. In this case, by Theorem \ref{prop:applicationtensor} the integers $d_i$ and $d_j$ are odd, and furthermore $d_i = 1$ or $d_j = 1$. Consequently if $u$ acts on $(\wedge^2(V), a_V)$ or $(L_G(\varpi_2), a_V)$ as a distinguished unipotent element, then $(V,b) \cong V(2n)$ or $n$ is even and $(V,b) \cong V(2) \perp V(2n-2)$.

Suppose first that $(V,b) \cong V(2n)$. If $n \geq 2$ is odd, then $V_1$ occurs in $\wedge^2(V_{2n})$ with multiplicity one by Lemma \ref{lemma:minblockinwedge2}. Thus if $u$ acts on $(L_G(\varpi_2), a_V)$ as a distinguished unipotent element, then each Jordan block size in $\wedge^2(V_{2n})$ has multiplicity at most $2$, and thus $n \in \{2,3,5\}$ by Lemma \ref{lemma:multiplicityatmosttwowedge}. Suppose next that $n$ is even, and let $\alpha = \nu_2(n)$. By Lemma \ref{lemma:minblockinwedge2} and Theorem \ref{thm:MAINTHMC}, as $K[u]$-modules $\wedge^2(V) \cong V_{2^{\alpha}} \oplus W$ and $L_G(\varpi_2) \cong V_{2^{\alpha}-2} \oplus W$, where $W$ has no Jordan blocks of size $2^{\alpha}$. Therefore if $u$ acts on $(\wedge^2(V), a_V)$ or $(L_G(\varpi_2), a_V)$ as a distinguished unipotent element, then each Jordan block size in $\wedge^2(V_{2n})$ has multiplicity at most $2$, and so $n = 2$ by Lemma \ref{lemma:multiplicityatmosttwowedge}.


Next we consider the other possibility, which is that $n$ is even and $(V,b) \cong V(2) \perp V(2n-2)$. By Lemma \ref{lemma:minblockinwedge2}, as $K[u]$-modules $\wedge^2(V_{2n-2}) \cong V_1 \oplus W$, where $W$ has no Jordan blocks of size $1$. Hence \begin{equation}\label{eq:finaleq2n2}\wedge^2(V) \cong \wedge^2(V_2) \oplus \wedge^2(V_{2n-2}) \oplus \left( V_2 \otimes V_{2n-2} \right) \cong V_1^2 \oplus W \oplus V_{2n-2}^2\end{equation} as $K[u]$-modules. Thus $u$ does not act on $(\wedge^2(V), a_V)$ as a distinguished unipotent element, since it has Jordan blocks of size $1$ in $\wedge^2(V)$. 

Note that $L_G(\varpi_2) \cong W \oplus V_{2n-2}^2$ as $K[u]$-modules by~\eqref{eq:finaleq2n2} and Theorem \ref{thm:MAINTHMC}. Thus if $u$ acts on $(L_G(\varpi_2), a_V)$ as a distinguished unipotent element, each Jordan block size in $\wedge^2(V_{2n-2}) \cong V_1 \oplus W$ has multiplicity at most two, and so $n \in \{2,3,4,6\}$ by Lemma \ref{lemma:multiplicityatmosttwowedge}. Here $n = 3$ is ruled out since we are assuming that $n$ is even. For $n = 4$, a calculation with Theorem \ref{thm:GowLaffey} shows that $L_G(\varpi_2) \cong V_6^3 \oplus V_8$ as $K[u]$-modules, so $u$ does not act on $(L_G(\varpi_2), a_V)$ as a distinguished unipotent element.

We still need to check that in the cases listed $u$ does indeed act as a distinguished unipotent element. To this end, a straightforward computation with Theorem \ref{thm:GowLaffey} and Theorem \ref{thm:MAINTHMC} shows the following.

\begin{itemize}
\item If $n = 2$ and $(V,b) \cong V(4)$, then $(\wedge^2(V), a_V) \cong V(2) \perp V(4)$.
\item If $n = 2$ and $(V,b) \cong V(4)$, then $(L_G(\varpi_2), a_V) \cong V(4)$.
\item If $n = 2$ and $(V,b) \cong V(2)^2$, then $(L_G(\varpi_2), a_V) \cong V(2)^2$.
\item If $n = 3$ and $(V,b) \cong V(6)$, then $(L_G(\varpi_2), a_V) \cong V(6) \perp V(8)$.
\item If $n = 5$ and $(V,b) \cong V(10)$, then $(L_G(\varpi_2), a_V) \cong V(6) \perp V(8) \perp V(14) \perp V(16)$.
\item If $n = 6$ and $(V,b) \cong V(2) \perp V(10)$, then $(L_G(\varpi_2), a_V) \cong V(6) \perp V(8) \perp V(10)^2 \perp V(14) \perp V(16)$.
\end{itemize}

This completes the proof of Proposition \ref{prop:applicationCtilt} and Proposition \ref{prop:applicationCirr}.\end{proof}

\bibliographystyle{alpha-abbrvsort-volume}
\bibliography{bibliography}

\begin{thebibliography}{QSSS76}

\bibitem[Alp86]{Alperin}
J.~L. Alperin.
\newblock {\em Local representation theory}, Volume~11 of {\em Cambridge
  Studies in Advanced Mathematics}.
\newblock Cambridge University Press, Cambridge, 1986.

\bibitem[DB10]{DeBruyn}
B.~De~Bruyn.
\newblock On the {G}rassmann modules for the symplectic groups.
\newblock {\em J. Algebra}, 324(2):218--230, 2010.

\bibitem[Fon74]{Fong}
P.~Fong.
\newblock On decomposition numbers of {$J_{1}$} and {$R(q)$}.
\newblock In {\em Symposia {M}athematica, {V}ol. {XIII} ({C}onvegno di {G}ruppi
  e loro {R}appresentazioni, {INDAM}, {R}ome, 1972)}, pages 415--422. Academic
  Press, London, 1974.

\bibitem[Ger61]{Gerstenhaber}
M.~Gerstenhaber.
\newblock Dominance over the classical groups.
\newblock {\em Ann. of Math. (2)}, 74:532--569, 1961.

\bibitem[GPX15]{GlasbyPraegerXiapart}
S.~P. Glasby, C.~E. Praeger, and B.~Xia.
\newblock Decomposing modular tensor products: `{J}ordan partitions', their
  parts and {$p$}-parts.
\newblock {\em Israel J. Math.}, 209(1):215--233, 2015.

\bibitem[GPX16]{GlasbyPraegerXia}
S.~P. Glasby, C.~E. Praeger, and B.~Xia.
\newblock Decomposing modular tensor products, and periodicity of `{J}ordan
  partitions'.
\newblock {\em J. Algebra}, 450:570--587, 2016.

\bibitem[GL06]{GowLaffey}
R.~Gow and T.~J. Laffey.
\newblock On the decomposition of the exterior square of an indecomposable
  module of a cyclic {$p$}-group.
\newblock {\em J. Group Theory}, 9(5):659--672, 2006.

\bibitem[GW95]{GowWillemsGreenCorrespondence}
R.~Gow and W.~Willems.
\newblock A note on {G}reen correspondence and forms.
\newblock {\em Comm. Algebra}, 23(4):1239--1248, 1995.

\bibitem[Gre59]{GreenIndecomposables}
J.~A. Green.
\newblock On the indecomposable representations of a finite group.
\newblock {\em Math. Z.}, 70:430--445, 1958/59.

\bibitem[Gre62]{GreenModular}
J.~A. Green.
\newblock The modular representation algebra of a finite group.
\newblock {\em Illinois J. Math.}, 6:607--619, 1962.

\bibitem[GM14]{GuralnickMalle}
R.~Guralnick and G.~Malle.
\newblock Rational rigidity for {$E_8(p)$}.
\newblock {\em Compos. Math.}, 150(10):1679--1702, 2014.

\bibitem[Hes79]{Hesselink}
W.~H. Hesselink.
\newblock Nilpotency in classical groups over a field of characteristic {$2$}.
\newblock {\em Math. Z.}, 166(2):165--181, 1979.

\bibitem[Hum72]{Humphreys}
J.~E. Humphreys.
\newblock {\em Introduction to {L}ie algebras and representation theory}.
\newblock Springer-Verlag, New York-Berlin, 1972.
\newblock Graduate Texts in Mathematics, Vol. 9.

\bibitem[Hum75]{HumphreysGroupBook}
J.~E. Humphreys.
\newblock {\em Linear algebraic groups}.
\newblock Springer-Verlag, New York-Heidelberg, 1975.
\newblock Graduate Texts in Mathematics, No. 21.

\bibitem[Jan03]{JantzenBook}
J.~C. Jantzen.
\newblock {\em Representations of algebraic groups}, Volume 107 of {\em
  Mathematical Surveys and Monographs}.
\newblock American Mathematical Society, Providence, RI, second edition, 2003.

\bibitem[Kor18]{KorhonenUP}
M.~Korhonen.
\newblock Unipotent elements forcing irreducibility in linear algebraic groups.
\newblock {\em J. Group Theory}, 21(3):365--396, 2018.

\bibitem[Kor19]{KorhonenJordanGood}
M.~Korhonen.
\newblock Jordan blocks of unipotent elements in some irreducible
  representations of classical groups in good characteristic.
\newblock {\em Proc. Amer. Math. Soc.}, 147(10):4205--4219, 2019.

\bibitem[Law95]{Lawther}
R.~Lawther.
\newblock Jordan block sizes of unipotent elements in exceptional algebraic
  groups.
\newblock {\em Comm. Algebra}, 23(11):4125--4156, 1995.

\bibitem[Law98]{LawtherCorrection}
R.~Lawther.
\newblock Correction to: ``{J}ordan block sizes of unipotent elements in
  exceptional algebraic groups''.
\newblock {\em Comm. Algebra}, 26(8):2709, 1998.

\bibitem[Law09]{LawtherFusion}
R.~Lawther.
\newblock Unipotent classes in maximal subgroups of exceptional algebraic
  groups.
\newblock {\em J. Algebra}, 322(1):270--293, 2009.

\bibitem[LLS14]{LawtherOuter}
R.~Lawther, M.~W. Liebeck, and G.~M. Seitz.
\newblock Outer unipotent classes in automorphism groups of simple algebraic
  groups.
\newblock {\em Proc. Lond. Math. Soc. (3)}, 109(3):553--595, 2014.

\bibitem[LS12]{LiebeckSeitzClass}
M.~W. Liebeck and G.~M. Seitz.
\newblock {\em Unipotent and nilpotent classes in simple algebraic groups and
  {L}ie algebras}, Volume 180 of {\em Mathematical Surveys and Monographs}.
\newblock American Mathematical Society, Providence, RI, 2012.

\bibitem[L{\"u}b01]{Lubeck}
F.~L{\"u}beck.
\newblock Small degree representations of finite {C}hevalley groups in defining
  characteristic.
\newblock {\em LMS J. Comput. Math.}, 4:135--169, 2001.

\bibitem[McN98]{McNinch}
G.~J. McNinch.
\newblock Dimensional criteria for semisimplicity of representations.
\newblock {\em Proc. London Math. Soc. (3)}, 76(1):95--149, 1998.

\bibitem[Mur16]{MurraySymmetricVertices}
J.~Murray.
\newblock Symmetric bilinear forms and vertices in characteristic 2.
\newblock {\em J. Algebra}, 462:338--374, 2016.

\bibitem[PM18]{PforteMurray}
L.~Pforte and J.~Murray.
\newblock The indecomposable symplectic and quadratic modules of the
  {K}lein-four group.
\newblock {\em J. Algebra}, 505:92--124, 2018.

\bibitem[QSSS76]{Quebbemann}
H.~G. Quebbemann, R.~Scharlau, W.~Scharlau, and M.~Schulte.
\newblock Quadratische {F}ormen in additiven {K}ategorien.
\newblock {\em Bull. Soc. Math. France Suppl. Mem.}, (48):93--101, 1976.

\bibitem[Sax98]{SaxlSpecial}
J.~Saxl.
\newblock Overgroups of special elements in simple algebraic groups and finite
  groups of {L}ie type.
\newblock In {\em Algebraic groups and their representations ({C}ambridge,
  1997)}, Volume 517 of {\em NATO Adv. Sci. Inst. Ser. C Math. Phys. Sci.},
  pages 291--300. Kluwer Acad. Publ., Dordrecht, 1998.

\bibitem[SS97]{SaxlSeitz}
J.~Saxl and G.~M. Seitz.
\newblock Subgroups of algebraic groups containing regular unipotent elements.
\newblock {\em J. London Math. Soc. (2)}, 55(2):370--386, 1997.

\bibitem[Sei87]{SeitzClassical}
G.~M. Seitz.
\newblock The maximal subgroups of classical algebraic groups.
\newblock {\em Mem. Amer. Math. Soc.}, 67(365):iv+286, 1987.

\bibitem[Spa82]{Spaltenstein}
N.~Spaltenstein.
\newblock {\em Classes unipotentes et sous-groupes de {B}orel}, Volume 946 of
  {\em Lecture Notes in Mathematics}.
\newblock Springer-Verlag, Berlin-New York, 1982.

\bibitem[TZ13]{TestermanZalesski}
D.~M. Testerman and A.~Zalesski.
\newblock Irreducibility in algebraic groups and regular unipotent elements.
\newblock {\em Proc. Amer. Math. Soc.}, 141(1):13--28, 2013.

\bibitem[TZ02]{TiepZalesski}
P.~H. Tiep and A.~E. Zalesski\u\i.
\newblock Mod {$p$} reducibility of unramified representations of finite groups
  of {L}ie type.
\newblock {\em Proc. London Math. Soc. (3)}, 84(2):439--472, 2002.

\bibitem[Wil76]{WillemsThesis}
W.~Willems.
\newblock {\em Metrische Moduln \"uber Gruppenringen}.
\newblock PhD thesis, Johannes Gutenberg-Universit\"at, Mainz, 1976.

\bibitem[Wil77]{Willems}
W.~Willems.
\newblock Metrische {$G$}-{M}oduln \"uber {K}\"orpern der {C}harakteristik
  {$2$}.
\newblock {\em Math. Z.}, 157(2):131--139, 1977.

\end{thebibliography}

\end{document}